\documentclass[12pt]{article}

\usepackage{amssymb,amsmath,bm}
\usepackage{mathrsfs}
\usepackage{constants,color}
\usepackage{enumerate}
\usepackage{appendix}
\usepackage{tikz}
\usepackage{ulem}
\def\qed{\hfill \mbox{\rule{0.5em}{0.5em}}}
\usepackage{graphicx}
\usepackage[colorlinks=false,linkcolor=false,urlcolor=black,bookmarksopen=true]{hyperref}

\usepackage{bookmark}

\newcommand{\be}{\begin{equation}}
\newcommand{\ee}{\end{equation}}
\newcommand{\bes}{\begin{equation*}}
\newcommand{\ees}{\end{equation*}}
\newcommand{\ba}{\begin{aligned}}
\newcommand{\ea}{\end{aligned}}
\newcommand{\bi}{\begin{itemize}}
\newcommand{\ei}{\end{itemize}}

\newcommand{\NN}{\mathbb{N}}
\newcommand{\EE}{\mathbb{E}}
\newcommand{\PP}{\mathbb{P}}

\newcommand{\RR}{\mathbb{R}}

\newcommand{\Var}{{\rm Var}}

\allowdisplaybreaks

\numberwithin{equation}{section}

\newtheorem{theorem}{Theorem}%[section]

\newtheorem{proposition}{Proposition}%[section]
\newtheorem{lemma}{Lemma}%[section]

\newtheorem{remark}{Remark}[section]

\title{Large Deviations of Cancer Recurrence Timing}
\author{Pranav Hanagal, Kevin Leder, Zicheng Wang}
\begin{document}
\maketitle
\begin{abstract}
We study large deviation events in the timing of disease recurrence. In particular, we are interested in modeling cancer treatment failure due to mutation induced drug resistance. We first present a two-type branching process model of this phenomenon, where an initial population of cells that are sensitive to therapy can produce mutants that are resistant to the therapy. In this model we investigate two random times, the recurrence time and the crossover time. Recurrence time is defined as the first time that the population size of mutant cells exceeds a given proportion of the initial population size of drug-sensitive cells. Crossover time is defined as the first time that the resistant cell population dominates the total population. We establish convergence in probability results for both recurrence and crossover time. We then develop expressions for the large deviations rate of early recurrence and early crossover events. We characterize how the large deviation rates and rate functions depend on the initial size of the mutant cell population.
%We then apply our results to analyze the number of mutant clones at recurrence time conditioned on the event of early recurrence. \textcolor{red}{We project that the distribution of the number of mutant clones at recurrence time which will survive forever conditioned on the event of early recurrence stochastically dominates that without conditioning.}
We finally look at the large deviations rate of early recurrence conditioned on the number of mutant clones present at recurrence in the special case of deterministically decaying sensitive population. We find that if recurrence occurs before the predicted law of large numbers limit then there will likely be an increase in the number of clones present at recurrence time.
%{\bf MSC:} 60F10, 60J60, 65C05, 60J85.
\end{abstract}
%\tableofcontents

\section{Introduction}\label{introduction}

In the past two decades, targeted therapies have been developed and applied to treat many types of cancers (chronic myeloid leukemia, non-small cell lung cancer, etc). Many of these therapies lead to a substantial decline in tumor burden. However, mutation-induced drug resistance often occurs and results in cancer recurrence.
Therefore, it is of particular importance to understand the evolutionary dynamics of cancer recurrence. Currently, little is known about the scenarios where cancers recur abnormally earlier than expected. In this work, we develop a stylized model to analyze these rare events and provide information on the number of clones of mutants conditioned on early recurrence.

Mathematical modeling can help in the understanding of complicated cellular dynamics. A large number
of mathematical models have been proposed and analyzed to study the dynamics
of bacterial and cancer cell populations and the evolution of drug resistance. Luria and Delbr\"{u}ck \cite{Luria1943} built a simple model (LD model) in which sensitive and resistant bacteria grow deterministically, while mutation occurs randomly. They derived the distribution of the number of resistant bacteria, and formulated a function to estimate the mutation rate. Iwasa, Nowak and Michor \cite{Michor2006} analyzed a stochastic LD model (with some deterministic approximations) and obtained the probability of resistance, the mean number of mutants, and the distribution of the number of mutants when the cell population reached a detection size. Kessler and Levine \cite{Levine2013} studied the fully stochastic LD model and derived the probability distribution for the number of mutants when the cell population reached a large size (large population size limit). Keller and Antal \cite{Keller2015} studied the LD model with stochastic growth of the resistant population. They obtained an exact formula for the generating function of the number of mutants when the number of drug-sensitive cells reach a given threshold. Cheek and Antal \cite{Cheek2018} investigated the fully stochastic LD model and obtained valuable results for the  ``mutation times, clone sizes and number of mutants" in different settings. All of these works provide valuable insights into how mutations accumulate during tumor expansion (prior to treatment), while our work focuses on what happens during treatment.

%Michor and co-authors \cite{Michor2005} constructed a four-compartment model to analyze the evolution of chronic myeloid leukaemia under targeted therapy. They provided significant insights into the factors which would cause treatment failure. " 

%Williams and Bjerknes \cite{Williams1972} constructed a spatial model to analyze tumor growth. In their model, normal and abnormal cells grow on a spatial lattice. Williams-Bjerknes model was later studied in details by Bramson and Griffeath \cite{Bramson1980,Bramson1981}. They showed that conditioned on non-extinction, the tumor will contain ``a ball of linearly expanding radius" in the long run.

Our work is based on a stream of literature studying multi-type branching processes. Durrett \cite{Durrett} summarized a lot of important results concerning two stochastic times: the time of the first type k mutation, and the time of the first type k mutation that founds a family line that does not go extinct. Such results can be applied to cancer treatment, such as finding a ``window of opportunity for screening." Jagers and co-authors \cite{JKS2007a,JKS2007b} obtained results for the ``path to extinction" in subcritical branching processes, including convergence of finite dimensional distributions and its extension to general inter-arrival times. In this work, we will utilize the same time scale as considered in \cite{JKS2007a,JKS2007b}. Foo and Leder \cite{JK2013} used a branching process model to analyze the dynamics of escape from extinction. They studied the time at which the total cell population begins to rebound (the population stops declining and starts increasing) as well as the crossover time (the first time at which the total cell population is dominated by the mutant cell population). %which represents the first time that the total cell population is dominated by the mutant cell population. 
Foo, Leder and Zhu \cite{JKJ2014} extended the study of escape times to include random mutational fitness advantage, which means that each mutation results in the creation of a supercritical birth-death process with random birth rate. The authors further established a functional central limit theorem for a scaled version of the mutant cell process.  Lastly Antal and Avanzini \cite{avanzini2019cancer} investigated time to disease recurrence due to possibly occult metastasis in a branching process model of cancer.

In this work we develop a number of results concerning the recurrence and crossover times. We examine the dynamics between drug-sensitive cells and mutant cells during treatment. We also incorporate the initial size of mutant population prior to treatment which haven't been studied in \cite{JK2013} and \cite{JKJ2014}. Clinical studies have shown that drug resistant mutations can be present prior to treatment \cite{Oh2011}, and it is therefore important to incorporate the initial size of mutant population.We focus on the particular rare event of early recurrence such that recurrence happens $y$ units of time earlier than the deterministic limit. We formulate the prediction of recurrence time in Section \ref{model}, and establish a convergence in probability result in Theorem \ref{cip_recurrence}. In Theorem \ref{recurrence_LD}, we obtain the large deviations rates and rate functions of early recurrence with respect to different sizes of initial mutant population. In particular, we examine three possible cases: (1) resistance is driven by mutants created during treatment, (2) resistance is driven by both pre-existing and acquired mutants, and (3) resistance is driven by pre-existing mutants. We observe that in case (1) and (2), rate functions depend on the mutation parameter $\mu$, while it doesn't in case (3). Moreover, in all three cases, the decay rate increases in $y$. In Section \ref{results_crossover}, the large deviations results have also been established for the crossover time.
%where crossover time represents the first time that the total cell population is dominated by the mutant cell population.
We find that the large deviations rates are similar for the two random times, but the large deviations rate for crossover time is more sensitive to the decay rate of the sensitive cell population.
%We then examine the clonal diversity present at recurrence time which is a key factor in the treatment of recurrent cancer. In Propositions \ref{dist_1}, \ref{dist_1_2} and \ref{dist_2}, we study the distributions of the number of clones of mutants at the recurrence time which will survive forever with or without conditioning on the event of early recurrence. Our results indicate that the distribution of the number of mutant clones at recurrence time which will survive forever conditioned on the event of early recurrence stochastically dominates that without conditioning. 
In Section \ref{LD_condition_sec}, we study
%explore another approach to study clonal diversity by examining 
the rare event of early recurrence conditioned on a fixed number of mutant clones at $y$ units of time before the deterministic limit. We then use this result to find the number of clones which minimizes the decay rate, i.e., the number of clones that is most likely to lead to an early recurrence. We observe that this optimal number of clones increases in $y$, which indicates that early recurrence leads to a larger number of mutant clones present at recurrence.

%The main contribution of this work is as follows. Theorem \ref{cip_recurrence} establishes convergence in probability results for recurrence time. In Theorem \ref{recurrence_LD}, we show that the probability of early recurrence satisfies large deviation principle. In Proposition \ref{dist_1}, \ref{dist_1_2} and \ref{dist_2}, we study the distributions of number of clones of mutants which will survive forever at the recurrence time with or without conditioning on the event of early recurrence. In Proposition \ref{LD_condition}, we obtain a similar result to Theorem \ref{recurrence_LD} conditioning on the number of clones at the early recurrence time with a simpler setting.

The remainder of this work is organized as follows. In Section \ref{model}, we describe our model and present important results from previous works. In Section \ref{results}, we state the main results and their implications. In Section \ref{proof}, we provide proofs of our main results.

\section{Model and Previous Results}\label{model}

In this section, we introduce the multi-type branching process model, as well as the recurrence time and crossover time. Our model is similar to the models studied in \cite{JK2013} and \cite{JKJ2014}.

Consider a subcritical birth-death process $\left(Z_0^n\left(t\right)\right)_{t\ge 0}$ with birth rate $r_0$, death rate $d_0$ and net growth rate $\lambda_0 = r_0-d_0 < 0$. $Z_0^n$ represents the drug-sensitive wild-type population with $Z_0^n\left(0\right)=n$. For our convenience, we will additionally define $r=-\lambda_0$. Consider a supercritical birth-death process $\left(Z_1^n\left(t\right)\right)_{t\ge 0}$ with birth rate $r_1$, death rate $d_1$ and net growth rate $\lambda_1 = r_1-d_1 > 0$. $Z_1^n$ represents the drug-resistant mutants that preexist treatment (type-$1$ mutants) with $Z_1^n\left(0\right)=X\left(n\right)$. Drug-resistant mutants are also generated during treatment at time $t$ at rate $Z_0^n\left(t\right)\mu n^{-\alpha}$ for $\alpha \in \left(0,1\right)$, and each of these mutations results in the creation of a clone which is modeled as a supercritical birth-death process with birth rate $r_1$, death rate $d_1$ and net growth rate $\lambda_1 = r_1-d_1 > 0$. We denote this population (type-$2$ mutants) by $\left(Z_2^n\left(t\right)\right)_{t\ge 0}$. Then $Z_2^n$ is a supercritical branching process with immigration. To clarify the model, we note that mutants modeled by $Z_1^n$ and $Z_2^n$ behave exactly the same. They are only different in terms of their ancestors.
%We can think of $Z_2^n$ as the mutants (and their descendants) that are created during treatment, while $Z_1^n$ are the mutants (and their descendants) that existed prior to treatment.
For each $n\geq 1$, the processes $(Z_0^n,Z_1^n,Z_2^n)$ are defined on a common probability space $(\Omega_n, \mathcal{F}^n,\PP_n)$. For ease of notation we will write $\PP$ instead of $\PP_n$.

To simplify notation, we define for $i=0,1,2$ the functions
\begin{equation*}
z_i^n\left(t\right)=\EE Z_i^n\left(t\right).
\end{equation*}
If we use the time scale $t_n=\frac{1}{r}\text{log} \left(n\right)$, then the following results are known from \cite{JK2013}, \cite{JKJ2014}, and \cite{bp}:
\begin{align*}
& z_0^n\left(u t_n\right)=ne^{-rut_n}=n^{1-u},\\
& z_1^n\left(u t_n\right)=X\left(n\right)e^{\lambda_1 ut_n}=X\left(n\right)n^{\frac{\lambda_1}{r}u},\\
& z_2^n\left(u t_n\right)=\frac{\mu}{\lambda_1+r}n^{1-\alpha}e^{\lambda_1 ut_n}\left(1-e^{\left(\lambda_0-\lambda_1\right)u t_n}\right)=\frac{\mu}{\lambda_1+r}n^{1-\alpha}\left(n^{\frac{\lambda_1}{r}u}-n^{-u}\right),\\
& \Var\left(Z_0^n\left(u t_n\right)\right)=\frac{r_0+d_0}{r}n^{1-u}\left(1-n^{-u}\right),\\
& \Var\left(Z_1^n\left(u t_n\right)\right)=X\left(n\right)\frac{r_1+d_1}{\lambda_1}n^{\frac{\lambda_1}{r}u}\left(n^{\frac{\lambda_1}{r}u}-1\right).
\end{align*}

We also define the moment generating function of a binary branching process starting from a single cell. Let $Z=\{Z\left(t\right), t\ge 0\}$ denote a binary branching process where $Z\left(0\right)=1$ and each individual cell has birth rate $r_1$, death rate $d_1$, and net growth rate $\lambda_1=r_1-d_1$. The moment generating function is given by 

\begin{align}\label{generating_function}
\phi_t\left(\theta\right) = \EE \exp\left(\theta Z\left(t\right)\right)=\left\{ \begin{array}{cc} 
\frac{d_1\left(e^{\theta}-1\right)-e^{-\lambda_1t}\left(r_1e^{\theta}-d_1\right)}{r_1\left(e^{\theta}-1\right)-e^{-\lambda_1t}\left(r_1e^{\theta}-d_1\right)}, & \hspace{5mm} \theta<\bar{\theta}_t \\
\infty & \hspace{5mm} \theta\ge\bar{\theta}_t \\
\end{array} \right.
\end{align}
where
\begin{equation}\label{bar_theta}
\bar{\theta}_t\doteq \log\left(\frac{r_1e^{\lambda_1t}-d_1}{r_1e^{\lambda_1 t}-r_1}\right)
\end{equation}
(see page 109 of \cite{bp}). Throughout this paper, we will repeatedly use $\phi_t$ to denote the moment generating function of $Z(t)$. 

We are interested in the asymptotic properties of the recurrence time which we define as
\begin{equation*}
\gamma_n\left(a\right)=\inf\{t\ge 0: Z_1^n\left(t\right)+Z_2^n\left(t\right)>an \}
\end{equation*}
for $a>0$. We will often be interested in $\gamma_n(1)$, and use the notation $\gamma_n\equiv \gamma_n(1)$. Recurrence time represents the first time that the mutant cell population given by $Z_1^n\left(t\right)+Z_2^n\left(t\right)$ exceeds a proportion $a$ of the initial population size of drug-sensitive cells. We denote by $\zeta_n\left(a\right)$ the unique solution to $z_1^n\left(t\right)+z_2^n\left(t\right)=an$. We will show that $\gamma_n\left(a\right)-\zeta_n\left(a\right)\rightarrow 0$ in probability. Due to the complexity of the equation $z_1^n\left(t\right)+z_2^n\left(t\right)=an$, we cannot get an explicit form of $\zeta_n\left(a\right)$. Instead we solve the equation (by noting that $z_2^n\left(t\right)\ge \frac{\mu}{\lambda_1+r}n^{1-\alpha}e^{\lambda_1 t}\left(1-e^{-\lambda_1 t}\right)$)
\begin{equation}\label{upper_bound_equation}
\left(X\left(n\right)+\frac{\mu}{\lambda_1+r}n^{1-\alpha}\right)e^{\lambda_1 t}-\frac{\mu}{\lambda_1+r}n^{1-\alpha}=an
\end{equation}
to get an upper bound $\bar{u}_n^{upper}=\frac{1}{\lambda_1}\log \left(\frac{an+\frac{\mu}{\lambda_1+r}n^{1-\alpha}}{X\left(n\right)+\frac{\mu}{\lambda_1+r}n^{1-\alpha}}\right)$.
Similarly we solve the equation (by noting that $z_2^n\left(t\right)\le \frac{\mu}{\lambda_1+r}n^{1-\alpha}e^{\lambda_1 t}$)
\begin{equation}\label{lower_bound_equation}
\left(X\left(n\right)+\frac{\mu}{\lambda_1+r}n^{1-\alpha}\right)e^{\lambda_1 t}=an
\end{equation}
%If we write this equation explicitly:
%\begin{equation*}
%\left(X\left(n\right)+\frac{\mu}{\lambda_1+r}n^{1-\alpha}\right)e^{\lambda_1 t}-\frac{\mu}{\lambda_1+r}n^{1-\alpha}e^{\lambda_0 t}=an,
%\end{equation*}
to get a lower bound $\bar{u}_n^{lower}=\frac{1}{\lambda_1}\log \left(\frac{an}{X\left(n\right)+\frac{\mu}{\lambda_1+r}n^{1-\alpha}}\right)$,
which satisfies,
\begin{align*}
\bar{u}_n^{upper}-\bar{u}_n^{lower}=\frac{1}{\lambda_1}\log \left(\frac{an+\frac{\mu}{\lambda_1+r}n^{1-\alpha}}{an}\right) \rightarrow 0, \text{ as } n\rightarrow \infty.
\end{align*}

%It is easy to notice that for $y>0$, $\PP\left(\gamma_n\left(a\right)\le \zeta_n\left(a\right)-y\right)\rightarrow 0$ as $n\rightarrow \infty$. 
With the convergence in probability result (Theorem \ref{cip_recurrence}), it is easy to notice that an `early recurrence', which we denote by $\{\gamma_n\left(a\right)\le \zeta_n\left(a\right)-y\}$, is a rare event if we take $n$ sufficiently large. %Similarly, $\PP\left(\tau_n\le \xi_n-y\right)\rightarrow 0$ as $n\rightarrow \infty$. Therefore an `early crossover', which we denote by $\{\tau_n\le \xi_n-y\}$, is also a rare event. 
In this work, we will show that the probability of an early recurrence decays exponentially fast, and the corresponding rate and rate function depend on the initial number of mutant cells. %We will establish a similar result for an early crossover as well.

Foo, Leder and Zhu \cite{JKJ2014}  study a different stochastic time for a closely related model. In particular define the crossover time as
$$
\tau_n=\inf\{t\geq 0: Z_1^n(t)+Z_2^n(t)>Z_0^n(t)\}.
$$
In addition define $\xi_n$ as the unique solution to $z_1^n(t)+z_2^n(t)=z_0^n(t)$. In \cite{JKJ2014} they establish that $\xi_n-\tau_n\to 0$ as $n\to\infty$ and establish a central limit theorem for fluctuations of $\tau_n$ away from $\xi_n$. In this work we focus on the recurrence time, but our main results have also been established for the crossover time and this can be found in Section \ref{results_crossover}.

Throughout this work we will use the following notation for the asymptotic behavior of positive functions:
\begin{align*}
f\left(t\right) \sim g\left(t\right) & \quad\hbox{if $f\left(t\right)/g\left(t\right) \to 1$ as $t \to \infty$}, \\
f\left(t\right) = o\left(g\left(t\right)\right) & \quad\hbox{if $f\left(t\right)/g\left(t\right) \to 0$ as $t \to \infty$}, \\
%f\left(t\right) \gg g\left(t\right) & \quad\hbox{if $f\left(t\right)/g\left(t\right) \to \infty$ as $t \to \infty$} \\
f\left(t\right) = O\left(g\left(t\right)\right) & \quad \hbox{if $f\left(t\right) \leq C g\left(t\right)$ for all $t$}, \\
f\left(t\right) = \Theta\left(g\left(t\right)\right) & \quad \hbox{if $c g\left(t\right) \le f\left(t\right) \leq C g\left(t\right)$ for all $t$},
\end{align*}
where $C$ and $c$ are positive constants.

\section{Results}\label{results}

We first establish the convergence in probability result for the recurrence time.

\begin{theorem}\label{cip_recurrence}
	Assume that $\alpha\in \left(0,1\right)$, $a>0$, $X\left(n\right)<n$, and $X\left(n\right)$ is non-decreasing in $n$, then for every $\epsilon>0$ we have that
	\begin{center}
		$\lim\limits_{n\rightarrow \infty}\PP\left(|\gamma_n\left(a\right)-\zeta_n\left(a\right)|>\epsilon\right)=0.$
	\end{center}
\end{theorem}
\begin{proof}
Section \ref{proof_cip_recurrence}.
\end{proof}

The proof of Theorem \ref{cip_recurrence} is similar to Theorem $1$ of \cite{JKJ2014} but it extends the result by incorporating the initial size of mutant population. %Theorem \ref{cip_recurrence} establishes a similar result for recurrence time. 
Recurrence time (sometimes called relapse time in literature) is very important in cancer treatment. Theorem \ref{cip_recurrence} shows that if we have enough information about the mutation rate and the initial size of mutant population to certain therapy, then we can have a good approximation of the recurrence time.

Next, we establish the large deviations results for the probability of an early recurrence. In short, a large deviation result can be thought as finding $c\ge 0$ and $L\left(y\right)$ such that
\begin{align*}
\PP \left(\gamma_n\left(a\right)\le \zeta_n\left(a\right)-y\right)= \exp\left[-n^{c}\left(L\left(y\right)+o(1)\right) \right],
\end{align*}
where we call $n^c$ the large deviation rate and $L\left(y\right)$ the rate function. We investigate three cases: (1) $X\left(n\right)=o\left(n^{1-\alpha}\right)$, (2) $X\left(n\right) \sim n^{1-\alpha}$ and (3) $X\left(n\right) \sim n^{1-\beta}$, where $0<\beta<\alpha$. Case (1) is the setting where there is an insignificant amount of mutants present at diagnosis and resistance is driven by mutants created during treatment. Case (3) is the setting where resistance is driven by pre-existing mutants, and Case (2) is the critical case where both pre-existing and acquired mutants drive resistance.

%\textcolor{red}{ZICHENG: Why don't these results depend on $a$? Is there a mistake somewhere?}
\begin{theorem}\label{recurrence_LD}
	Assume that $\alpha\in \left(0,1\right)$.\\
	
	(1) When $X\left(n\right)=o\left(n^{1-\alpha}\right)$,
	\begin{equation*}
	\lim\limits_{n\rightarrow \infty}\frac{1}{n^{1-\alpha}}\log \PP \left(\gamma_n\left(a\right)\le \zeta_n\left(a\right)-y\right)= -\sup\limits_{\theta\in\left(0,1\right)}\left[\frac{\lambda_1 \mu  \theta e^{y \lambda_1}}{r_1\left(\lambda_1+r\right)}-\frac{\lambda_1\mu\theta}{r_1}\int_{0}^{\infty}\frac{e^{-rs}}{e^{\lambda_1s}-\theta}ds\right].
	\end{equation*}
	
	(2) When $X\left(n\right) \sim n^{1-\alpha} $,
	\begin{align*}
	& \quad \lim\limits_{n\rightarrow \infty}\frac{1}{n^{1-\alpha}}\log \PP \left(\gamma_n\left(a\right)\le \zeta_n\left(a\right)-y\right)\\\nonumber
	& = -\sup\limits_{\theta\in\left(0,1\right)}\left[\frac{ \lambda_1 \theta \left(\mu+\lambda_1+r\right)e^{\lambda_1y}}{r_1\left(\lambda_1+r\right)}-\log\left(1-\frac{\lambda_1}{r_1}\frac{\theta}{\theta-1}\right)-\frac{\lambda_1 \mu \theta}{r_1}\int_{0}^{\infty}\frac{ e^{-rs}}{e^{\lambda_1 s}-\theta}ds\right].
	\end{align*}
	
	(3) When $X\left(n\right) \sim n^{1-\beta} $, where $0<\beta<\alpha$,
	\begin{equation*}
	\lim\limits_{n\rightarrow \infty}\frac{1}{n^{1-\beta}}\log \PP \left(\gamma_n\left(a\right)\le \zeta_n\left(a\right)-y\right)= -\sup\limits_{\theta\in\left(0,1\right)}\left[\frac{ \lambda_1 \theta e^{\lambda_1 y}}{r_1}-\log\left(1-\frac{\lambda_1}{r_1}\frac{\theta}{\theta-1}\right)\right].
	\end{equation*}
\end{theorem}

For case (3), we can show that the optimal value of $\theta$ is $\frac{r_1+d_1-\sqrt{{\lambda_1}^2+4r_1d_1 e^{-\lambda_1 y}}}{2d_1}$. Surprisingly, in all three cases, the rate functions do not depend on the value of $a$. The main reason is that the recurrence time is of order $\log\left(an\right)=\log\left(a\right)+\log\left(n\right)$, and hence the value of $a$ is negligible.

From Theorem \ref{recurrence_LD} we can see that the large deviation rates and rate functions depend on the initial size of mutant cell population and the mutation rate. We can observe that the function within the supremum operator in Case (2) is the summation of those in Case (1) and Case (3), which is reasonable as Case (2) is the critical case. In all three cases, the rate function increases in $y$ which is intuitive as an earlier recurrence occurs with a lower probability. In Figure \ref{ratefunction_y}, we plot an example of the rate function for $y\in \left(0, 2\right)$ in Case (1) and Case (3).

We can also observe that in Case (1) and (2), the rate functions depend on $\mu$, while it doesn't in Case (3). That is because mutation rate is irrelevant when resistance is driven by pre-existing mutants. Moreover, in Case (1), the rate function increases in $\mu$, which implies that when resistance is driven by mutants created during treatment, higher mutation rate leads to tighter concentration of recurrence time.

If we compare the decay rates (large deviation rate multiplied by rate function) of Case (1) and Case (3), we can observe that when $n$ is large, the decay rate of case (3) dominates (in magnitude) that of Case (1). The intuition is that the probability of early recurrence will have higher decay rate when more mutant cells accumulate early on. The reason is that a branching process will behave more deterministically when its initial size is larger. In Figure \ref{compare_1_3}, we plot $n^{\beta-\alpha}L_1(y)/L_3(y)$ (which is the ratio of decay rates between Case (1) and Case (3)) for different values of $\beta$, where $L_1$ and $L_3$ are the rate functions in Cases (1) and (3). Note that recurrence is more likely to occur early in Case (3), what we are showing is that the recurrence time is more tightly concentrated around the limiting behavior in Case (3).

In Case (1), we can also observe that the tail probability decay rate decreases in $r_1$ when $\lambda_1$ is fixed. The intuition is that given a fixed net growth rate, the decay rate will decrease in birth rate and death rate as higher birth rate/death rate will lead to higher fluctuation. In Figure \ref{fluctuation}, we plot an example of rate function for $r_1\in \left(0.2, 2\right)$.

\begin{figure}[h!]
\begin{center}
	\includegraphics[width=1\textwidth]{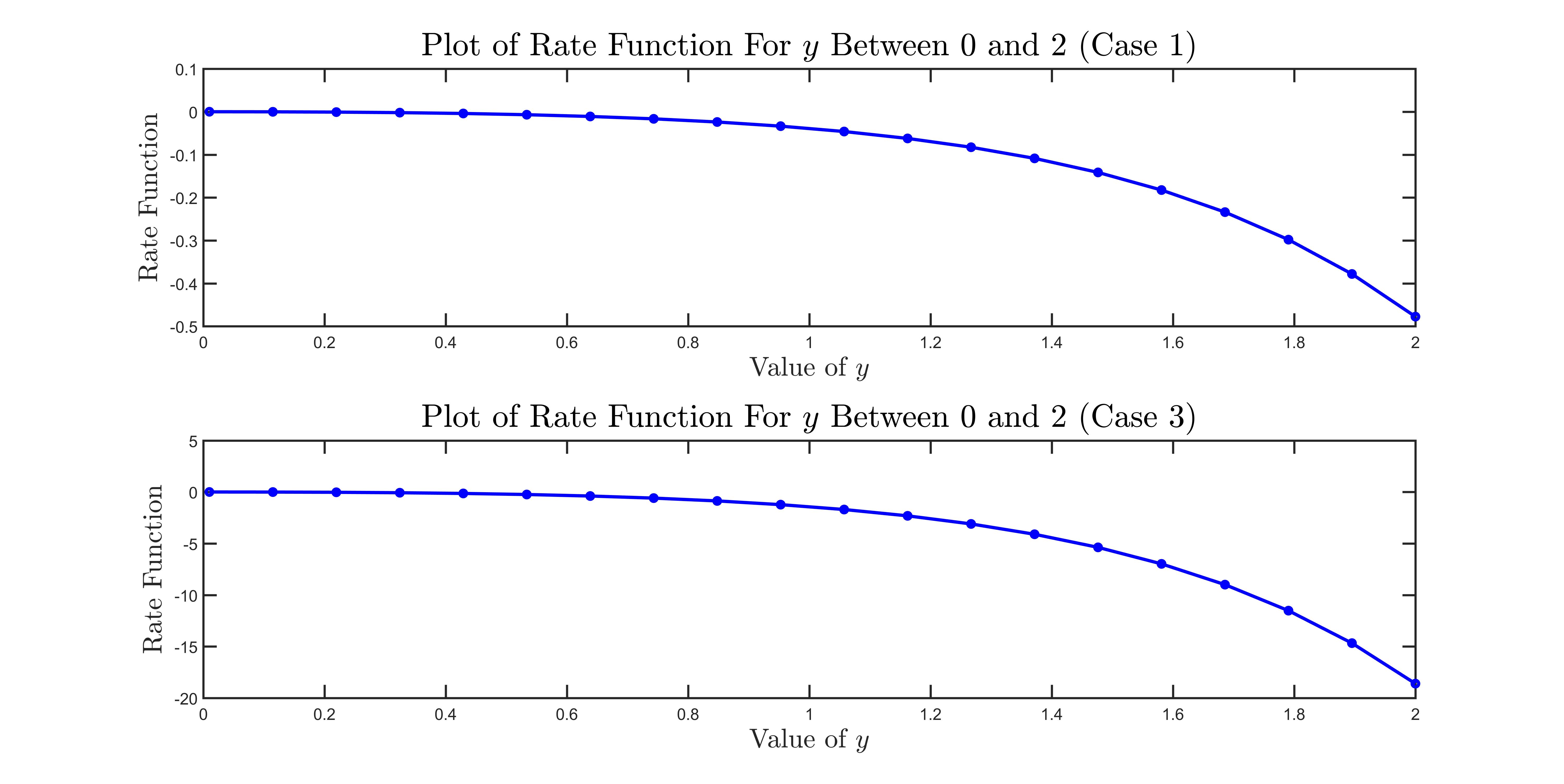}
\end{center}
\caption{Demonstration of rate function, where $\lambda_1 = 2, r_1=5, r = 2, $ and $\mu = 0.1.$}
\label{ratefunction_y}
\end{figure}

\begin{figure}[h!]
\begin{center}
	\includegraphics[width=1\textwidth]{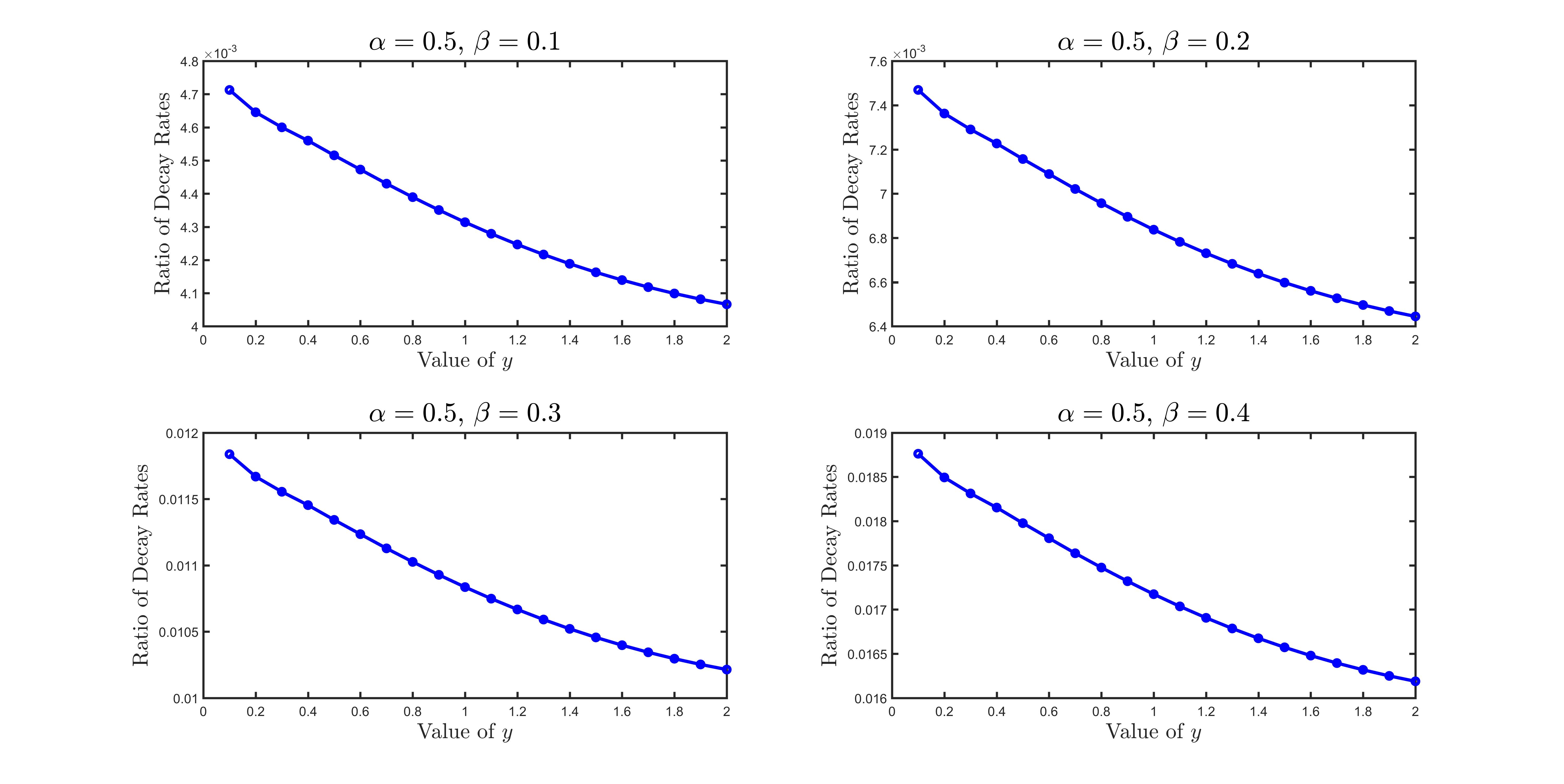}
\end{center}
\caption{Demonstration of ratios of decay rates between Case (1) and Case (3), where $\lambda_1 = 2, r_1=5, r = 2, \mu = 0.1$ and $n=100$.}
\label{compare_1_3}
\end{figure}

\begin{figure}[h!]
\begin{center}
	\includegraphics[width=1\textwidth]{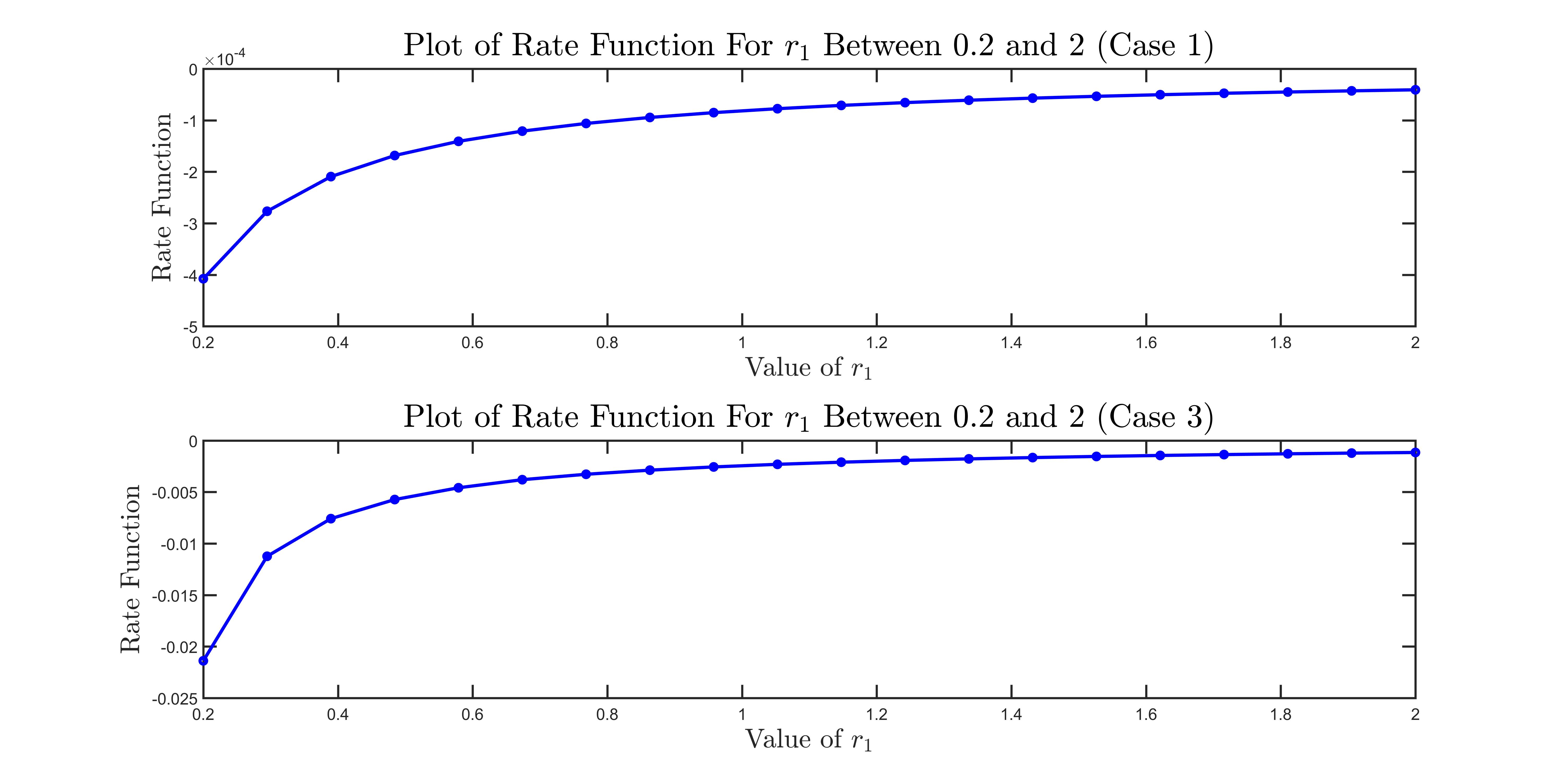}
\end{center}
\caption{Demonstration of rate function, where $\lambda_1 =0.2, y=1, r =0.2$ and $\mu = 0.01$.}
\label{fluctuation}
\end{figure}
\clearpage

Since the proofs for three cases are very similar, we will combine them together with the following manner: whenever the proofs of the three cases differ, we mark the difference by using (1), (2), or (3) to indicate which case is being dealt with.

Before proving Theorem \ref{recurrence_LD} we state some upper bounds that we will use throughout the paper. The proof of the following three results can be found in Section \ref{proof_upper}. 

This first result establishes an upper bound on a functional of the subcritical process.
\begin{proposition}\label{proposition_exp}
For $\theta\in \left(-\infty,1 \right)$, there exists $k_1>0$ such that
	\begin{equation*}
	\EE\exp\left(\frac{\mu}{n^{\alpha}}\int_{0}^{bt_n}\left(\phi_{bt_n-s}\left(v_{n,\theta,b}\right)-1\right)\left(Z_0^n\left(s\right)-ne^{-rs}\right)ds\right)\le \exp\left(k_1\left(\log n\right)^2n^{1-2\alpha}\right)
	\end{equation*}
for all $b>0$.
\end{proposition}

We next establish that recurrence that occurs too early can be safely ignored.
\begin{proposition}\label{prop_recurrence_1}
If $b\in (0,\frac{\alpha r}{\lambda_1})$ for case (1) and (2), and $b\in (0,\frac{\beta r}{\lambda_1})$ for case (3), then there exists $C>0$ such that 
$$
\PP\left(\sup\limits_{u\in\left[0,b\right]}\left(Z_1^n\left(ut_n\right)+Z_2^n\left(ut_n\right)\right)-an>0\right)=O\left(e^{-Cn^{1-b\left(\lambda_1/r\right)}}\right).
$$
\end{proposition}

We finally establish a Chernoff type upper bound on early recurrence.
\begin{proposition}\label{proposition_recurrence_1} 
For $\theta\in (0,1)$, 
%define $\theta_n=\frac{\lambda_1 \theta n^{1-\alpha}}{r_1}$, and $v_{n,\theta,u_n\left(y\right)}=\theta_n n^{\alpha-1-\lambda_1 u_n\left(y\right)/r}$, then
	\begin{align*}
	& \quad \PP\left(\sup\limits_{u\in\left[b,u_n\left(y\right)\right]}n^{\alpha-1-\lambda_1 u/r}\left(Z_1^n\left(ut_n\right)+Z_2^n\left(ut_n\right)\right)> an^{\alpha-\lambda_1 u_n\left(y\right)/r}\right)\\
	& \le \EE \exp\left(\frac{\mu}{n^{\alpha}}\int_{0}^{u_n\left(y\right)t_n}Z_0^n\left(s\right)\left(\phi_{u_n\left(y\right)t_n-s}\left(v_{n,\theta,u_n\left(y\right)}\right)-1\right)ds\right)\EE\exp\left(v_{n,\theta,u_n\left(y\right)}Z_1^n\left(u_n\left(y\right)t_n\right)\right)\\
	& \quad \times \exp\left(-anv_{n,\theta,u_n\left(y\right)}\right).
	\end{align*}
	%\textcolor{red}{ZICHENG: I added the restriction $\theta<1$, is that correct?}
\end{proposition}

We will break up the proof of Theorem \ref{recurrence_LD} into matching upper and lower bounds. We establish the upper bound first.
\begin{proposition}
\label{recurrence_LDUB}
For $y>0$ and $a>0$, when $X\left(n\right)=o\left(n^{1-\alpha}\right)$,
\begin{equation}\label{upperbound_recurrence_1}
\limsup\limits_{n\rightarrow \infty}\frac{1}{n^{1-\alpha}}\log \PP \left(\gamma_n\left(a\right)\le \zeta_n\left(a\right)-y\right)\le -\sup\limits_{\theta\in\left(0,1\right)}\left[\frac{\mu \lambda_1 \theta e^{y \lambda_1}}{r_1\left(\lambda_1+r\right)}-\frac{\lambda_1\mu\theta}{r_1}\int_{0}^{\infty}\frac{e^{-rs}}{e^{\lambda_1s}-\theta}ds\right];
\end{equation}
when $X\left(n\right) \sim n^{1-\alpha}$,
\begin{align}
	& \quad \limsup\limits_{n\rightarrow \infty}\frac{1}{n^{1-\alpha}}\log \PP \left(\gamma_n\left(a\right)\le \zeta_n\left(a\right)-y\right)\nonumber\\
	& \le -\sup\limits_{\theta\in\left(0,1\right)}\left[\frac{ \lambda_1 \theta \left(\mu+\lambda_1+r\right)e^{\lambda_1y}}{r_1\left(\lambda_1+r\right)}-\log\left(1-\frac{\lambda_1}{r_1}\frac{\theta}{\theta-1}\right)-\frac{\lambda_1 \mu \theta}{r_1}\int_{0}^{\infty}\frac{ e^{-rs}}{e^{\lambda_1 s}-\theta}ds\right]; \label{upperbound_recurrence_2}
\end{align}
when $X\left(n\right) \sim n^{1-\beta}$, where $0<\beta<\alpha$,
\begin{equation}\label{upperbound_recurrence_3}
\limsup\limits_{n\rightarrow \infty}\frac{1}{n^{1-\beta}}\log \PP \left(\gamma_n\left(a\right)\le \zeta_n\left(a\right)-y\right) \le -\sup\limits_{\theta\in\left(0,1\right)}\left[\frac{ \lambda_1 \theta e^{\lambda_1 y}}{r_1}-\log\left(1-\frac{\lambda_1}{r_1}\frac{\theta}{\theta-1}\right)\right].
\end{equation}
\end{proposition}

\begin{proof}\\
We define $t_n=\frac{1}{r}\log \left(n\right)$, and $u_n\left(y\right)=\left(\zeta_n\left(a\right)-y\right)/t_n$. Recall the definition of the moment-generating function $\phi_t$ from \eqref{generating_function}. For $\theta\in (0,1)$ and $b>0$ define 
\begin{align}
\label{eq:vdef}
& v_{n,\theta,b}=\frac{\lambda_1\theta}{r_1e^{\lambda_1 bt_n}}.
\end{align}

Observe that $\gamma_n\left(a\right)\le \zeta_n\left(a\right)-y\Leftrightarrow \sup\limits_{t\le \zeta_n\left(a\right)-y}\left(Z_1^n\left(t\right)+Z_2^n\left(t\right)\right)-an>0$. It then follows that 
\begin{equation*}
\PP\left(\gamma_n\left(a\right)\le \zeta_n\left(a\right)-y\right)=\PP\left(\sup\limits_{t\le \zeta_n\left(a\right)-y}\left(Z_1^n\left(t\right)+Z_2^n\left(t\right)\right)-an>0\right).
\end{equation*}
Recall that $t_n=\frac{1}{r}\log n$, and $u_n\left(y\right)=\left(\zeta_n\left(a\right)-y\right)/t_n$, then
\begin{equation*}
\PP\left(\sup\limits_{t\le \zeta_n\left(a\right)-y}\left(Z_1^n\left(t\right)+Z_2^n\left(t\right)\right)-an>0\right)=\PP\left(\sup\limits_{u\in\left[0,u_n\left(y\right)\right]}\left(Z_1^n\left(ut_n\right)+Z_2^n\left(ut_n\right)\right)-an>0\right).
\end{equation*}
It can be seen that for case (1) and (2), $u_n\left(y\right)\rightarrow \frac{\alpha r}{\lambda_1}$ as $n\rightarrow \infty$; for case (3), $u_n\left(y\right)\rightarrow \frac{\beta r}{\lambda_1}$ as $n\rightarrow \infty$. Therefore, if we define $b>0$ such that $b<\frac{\alpha r}{\lambda_1}$ for case (1) and (2), and $b<\frac{\beta r}{\lambda_1}$ for case (3), then apply Proposition \ref{prop_recurrence_1}
to see that
\begin{align*}
& \quad \PP\left(\sup\limits_{u\in\left[0,u_n\left(y\right)\right]}\left(Z_1^n\left(ut_n\right)+Z_2^n\left(ut_n\right)\right)-an>0\right)\\
& \le \PP\left(\sup\limits_{u\in\left[0,b\right]}\left(Z_1^n\left(ut_n\right)+Z_2^n\left(ut_n\right)\right)-an>0\right)+\PP\left(\sup\limits_{u\in\left[b,u_n\left(y\right)\right]}\left(Z_1^n\left(ut_n\right)+Z_2^n\left(ut_n\right)\right)-an>0\right)\\
&=
O\left(\exp\left[-Cn^{1-b\lambda_1/r}\right]\right)+\PP\left(\sup\limits_{u\in\left[b,u_n\left(y\right)\right]}\left(Z_1^n\left(ut_n\right)+Z_2^n\left(ut_n\right)\right)-an>0\right).
\end{align*}
We now analyze the case of an early recurrence occurring in the interval $\left[bt_n, \zeta_n\left(a\right)-y\right]$. We observe that
\begin{align}
& \quad \PP\left(\sup\limits_{u\in\left[b,u_n\left(y\right)\right]}\left(Z_1^n\left(ut_n\right)+Z_2^n\left(ut_n\right)\right)-an>0\right) \nonumber\\
& \le \PP\left(\sup\limits_{u\in\left[b,u_n\left(y\right)\right]}n^{\alpha-1-\lambda_1 u/r}\left(Z_1^n\left(ut_n\right)+Z_2^n\left(ut_n\right)\right)>an^{\alpha-\lambda_1 u_n\left(y\right)/r}\right). \label{eq:tail}
\end{align}
Note that we can modify the proof of Proposition \ref{proposition_exp} to obtain that there exists $k_2$ such that for $\theta\in \left(0,1\right)$,
\begin{equation}\label{eq:exponential_integral}
	\EE\exp\left(\frac{\mu}{n^{\alpha}}\int_{0}^{u_n\left(y\right)t_n}\left(\phi_{u_n\left(y\right)t_n-s}\left(v_{n,\theta,u_n\left(y\right)}\right)-1\right)\left(Z_0^n\left(s\right)-ne^{-rs}\right)ds\right)\le \exp\left(k_2\left(\log n\right)^2n^{1-2\alpha}\right),
\end{equation}
where $v_{n,\theta,u_n\left(y\right)}$ is defined in \ref{eq:vdef}. We now use Proposition \ref{proposition_recurrence_1} to bound the right hand side of \eqref{eq:tail}.
In particular, Proposition \ref{proposition_recurrence_1} tells us we can consider
\begin{align*}
& \quad \EE \exp\left(\frac{\mu}{n^{\alpha}}\int_{0}^{u_n\left(y\right)t_n}Z_0^n\left(s\right)\left(\phi_{u_n\left(y\right)t_n-s}\left(v_{n,\theta,u_n\left(y\right)}\right)-1\right)ds\right)\\
& =
\exp\left(\frac{\mu}{n^{\alpha-1}}\int_{0}^{u_n\left(y\right)t_n}e^{-rs}\left(\phi_{u_n\left(y\right)t_n-s}\left(v_{n,\theta,u_n\left(y\right)}\right)-1\right)ds\right)\\
& \quad \times \EE \exp\left(\frac{\mu}{n^{\alpha}}\int_{0}^{u_n\left(y\right)t_n}\left(Z_0^n\left(s\right)-ne^{-rs}\right)\left(\phi_{u_n\left(y\right)t_n-s}\left(v_{n,\theta,u_n\left(y\right)}\right)-1\right)ds\right)\\
&\leq
\exp\left(\frac{\mu}{n^{\alpha-1}}\int_{0}^{u_n\left(y\right)t_n}e^{-rs}\left(\phi_{u_n\left(y\right)t_n-s}\left(v_{n,\theta,u_n\left(y\right)}\right)-1\right)ds+k_2\left(\log n\right)^2n^{1-2\alpha}\right),
\end{align*}
where the inequality follows from \eqref{eq:exponential_integral}. To summarize we now have
\begin{align*}
& \quad \PP\left(\gamma_n\left(a\right)\le \zeta_n\left(a\right)-y\right)\\
& \le \PP\left(\sup\limits_{u\in\left[b,u_n\left(y\right)\right]}n^{\alpha-1-\lambda_1 u/r}\left(Z_1^n\left(ut_n\right)+Z_2^n\left(ut_n\right)\right)>an^{\alpha-\lambda_1 u_n\left(y\right)/r}\right)+O\left(\exp\left(-Cn^{1-b \left(\lambda_1/r\right)}\right)\right).
\end{align*}
%Define
%\begin{equation}\label{c_2}
%c_2\left(y;n\right)=c_1\left(y;n\right)+n^{\alpha-1-\lambda_1u_n\left(y\right)/r}\left(z_1^n\left(u_n\left(y\right)t_n\right)+z_2^n\left(u_n\left(y\right)t_n\right)\right).
%\end{equation}
Applying Proposition \ref{proposition_recurrence_1} and \eqref{eq:exponential_integral}, we obtain that
\begin{align*}
& \quad \PP\left(\sup\limits_{u\in\left[b,u_n\left(y\right)\right]}n^{\alpha-1-\lambda_1 u/r}\left(Z_1^n\left(ut_n\right)+Z_2^n\left(ut_n\right)\right)>an^{\alpha-\lambda_1 u_n\left(y\right)/r}\right)\\
& \le \EE \exp\left(\frac{\mu}{n^{\alpha}}\int_{0}^{u_n\left(y\right)t_n}Z_0^n\left(s\right)\left(\phi_{u_n\left(y\right)t_n-s}\left(v_{n,\theta,u_n\left(y\right)}\right)-1\right)ds\right)\\
& \quad \times \EE\exp\left(v_{n,\theta,u_n\left(y\right)}Z_1^n\left(u_n\left(y\right)t_n\right)\right) \exp\left(-anv_{n,\theta,u_n\left(y\right)}\right)\\
& \le \exp\left(\frac{\mu}{n^{\alpha-1}}\int_{0}^{u_n\left(y\right)t_n}e^{-rs}\left(\phi_{u_n\left(y\right)t_n-s}\left(v_{n,\theta,u_n\left(y\right)}\right)-1\right)ds\right)\exp\left(k_2\left(\log n\right)^2n^{1-2\alpha}\right)\\
& \quad \times \EE\exp\left(v_{n,\theta,u_n\left(y\right)}Z_1^n\left(u_n\left(y\right)t_n\right)\right) \exp\left(-anv_{n,\theta,u_n\left(y\right)}\right).
%&=\exp\left(-n^{1-\alpha}\left(\frac{\lambda_1 \theta}{r_1}c_2\left(y;n\right)-\int_{0}^{u_n\left(y\right)t_n}\mu e^{-rs}\left(\phi_{u_n\left(y\right)t_n-s}\left(v_{n,\theta,u_n\left(y\right)}\right)-1\right)ds+o\left(1\right)\right)\right),
\end{align*}
%where we define 
%(\textcolor{red}{ZICHENG: What about the exponential moment of $Z_1$? Shouldn't you cite an earlier lemma to say it is neglible?})
%\begin{equation*}
%F\left(n,\delta, \theta\right)=\exp\left(-\theta_n\left[an^{\alpha-\lambda_1 u_n\left(y\right)/r}-\frac{\mu}{\lambda_1+r}n^{-b\left(1+\lambda_1/r\right)}\right]\right).
%\end{equation*}
Lemma \ref{z1_moment_generate} gives us that for case (1),
\begin{align*}
& \quad \PP\left(\gamma_n\left(a\right)\le \zeta_n\left(a\right)-y\right)\\
& =\exp\left(-n^{1-\alpha}\left(\frac{\lambda_1 \theta}{r_1}an^{\alpha-\lambda_1 u_n\left(y\right)/r}-\int_{0}^{u_n\left(y\right)t_n}\mu e^{-rs}\left(\phi_{u_n\left(y\right)t_n-s}\left(v_{n,\theta,u_n\left(y\right)}\right)-1\right)ds+o\left(1\right)\right)\right)\\
& \quad +O\left(\exp\left(-Cn^{1-b \left(\lambda_1/r\right)}\right)\right);
\end{align*}
for case (2),
\begin{align*}
& \quad \PP\left(\gamma_n\left(a\right)\le \zeta_n\left(a\right)-y\right)\\
& =\exp\left(-n^{1-\alpha}\left(\frac{\lambda_1 \theta}{r_1}an^{\alpha-\lambda_1 u_n\left(y\right)/r}-\int_{0}^{u_n\left(y\right)t_n}\mu e^{-rs}\left(\phi_{u_n\left(y\right)t_n-s}\left(v_{n,\theta,u_n\left(y\right)}\right)-1\right)ds+o\left(1\right)\right)\right)\\
& \quad \times \EE\exp\left(v_{n,\theta,u_n\left(y\right)}Z_1^n\left(u_n\left(y\right)t_n\right)\right)+O\left(\exp\left(-Cn^{1-b \left(\lambda_1/r\right)}\right)\right);
\end{align*}
for case (3)
\begin{align*}
& \quad \PP\left(\gamma_n\left(a\right)\le \zeta_n\left(a\right)-y\right)\\
& =\exp\left(-n^{1-\beta}\left(\frac{\lambda_1 \theta}{r_1}n^{\beta-\alpha}an^{\alpha-\lambda_1 u_n\left(y\right)/r}+o\left(1\right)\right)\right)\times \EE\exp\left(v_{n,\theta,u_n\left(y\right)}Z_1^n\left(u_n\left(y\right)t_n\right)\right)\\
& \quad +O\left(\exp\left(-Cn^{1-b \left(\lambda_1/r\right)}\right)\right).
\end{align*}
%Lemma \ref{c_2_recurrence_1} shows that
%\begin{equation*}
%\frac{\lambda_1 \theta}{r_1}c_2\left(y;n\right)-\int_{0}^{u_n\left(y\right)t_n}\mu e^{-rs}\left(\phi_{u_n\left(y\right)t_n-s}\left(v_{n,\theta,u_n\left(y\right)}\right)-1\right)ds
%\end{equation*}
%is bounded (\textcolor{red}{ZICHENG: I'm confused. Why does it matter whether this quantity is bounded? Also you never explain taking the limit of the integral.} \textcolor{blue}{If this term is not bounded, then the first exponential term might be smaller than the second term}) and therefore it follows that
By Lemma \ref{c_2_recurrence_1}, we have that for case (1),
%\begin{align*}
%& \quad \PP\left(\gamma_n\left(a\right)\le \zeta_n\left(a\right)-y\right)\\
%&\le \exp\left(-n^{1-\alpha}\left(\frac{\lambda_1 \theta}{r_1}c_2\left(y;n\right)-\int_{0}^{u_n\left(y\right)t_n}\mu e^{-rs}\left(\phi_{u_n\left(y\right)t_n-s}\left(v_{n,\theta,u_n\left(y\right)}\right)-1\right)ds+o\left(1\right)\right)\right)\left(1+o\left(1\right)\right).
%\end{align*}
%We then have
\begin{equation*}
\limsup\limits_{n\rightarrow \infty}\frac{1}{n^{1-\alpha}}\log \PP\left(\gamma_n\left(a\right)\le \zeta_n\left(a\right)-y\right)\le \frac{\mu \lambda_1}{r_1}\int_{0}^{\infty}\frac{\theta e^{-rs}}{e^{\lambda_1 s}-\theta}ds-\frac{\mu \lambda_1 \theta e^{y \lambda_1}}{r_1\left(\lambda_1+r\right)};
\end{equation*}
for case (2), 
\begin{equation*}
\limsup\limits_{n\rightarrow \infty}\frac{1}{n^{1-\alpha}}\log \PP\left(\gamma_n\left(a\right)\le \zeta_n\left(a\right)-y\right)\le \frac{\mu \lambda_1}{r_1}\int_{0}^{\infty}\frac{\theta e^{-rs}}{e^{\lambda_1 s}-\theta}ds-\frac{\left(\mu+\lambda_1+r \right)\lambda_1 \theta e^{y \lambda_1}}{r_1\left(\lambda_1+r\right)}+\log\left(1-\frac{\lambda_1}{r_1} \frac{\theta}{\theta-1}\right);
\end{equation*}
for case (3),
\begin{equation*}
\limsup\limits_{n\rightarrow \infty}\frac{1}{n^{1-\beta}}\log \PP\left(\gamma_n\left(a\right)\le \zeta_n\left(a\right)-y\right)\le -\frac{\lambda_1 \theta e^{y \lambda_1}}{r_1}+\log\left(1-\frac{\lambda_1}{r_1} \frac{\theta}{\theta-1}\right).
\end{equation*}
The results \eqref{upperbound_recurrence_1}, \eqref{upperbound_recurrence_2} and \eqref{upperbound_recurrence_3} then follow by maximizing the upper bounds over $\theta$. \qed
\end{proof} %proof checked
\\

We next establish a lower bound that matches Proposition \ref{recurrence_LDUB}.
\begin{proposition}
\label{recurrence_LDLB}
For $y>0$ and $a>0$, when $X\left(n\right)=o\left(n^{1-\alpha}\right)$,
\begin{equation}\label{lowerbound_recurrence_1}
\liminf\limits_{n\rightarrow \infty}\frac{1}{n^{1-\alpha}}\log \PP \left(\gamma_n\left(a\right)\le \zeta_n\left(a\right)-y\right)\ge -\sup\limits_{\theta\in\left(0,1\right)}\left[\frac{\mu \lambda_1 \theta e^{y \lambda_1}}{r_1\left(\lambda_1+r\right)}-\frac{\lambda_1\mu\theta}{r_1}\int_{0}^{\infty}\frac{e^{-rs}}{e^{\lambda_1s}-\theta}ds\right];
\end{equation}
when $X\left(n\right) \sim n^{1-\alpha}$,
\begin{align}
	& \quad \liminf\limits_{n\rightarrow \infty}\frac{1}{n^{1-\alpha}}\log \PP \left(\gamma_n\left(a\right)\le \zeta_n\left(a\right)-y\right)\nonumber\\
	& \ge -\sup\limits_{\theta\in\left(0,1\right)}\left[\frac{ \lambda_1 \theta \left(\mu+\lambda_1+r\right)e^{\lambda_1y}}{r_1\left(\lambda_1+r\right)}-\log\left(1-\frac{\lambda_1}{r_1}\frac{\theta}{\theta-1}\right)-\frac{\lambda_1 \mu \theta}{r_1}\int_{0}^{\infty}\frac{ e^{-rs}}{e^{\lambda_1 s}-\theta}ds\right]; \label{lowerbound_recurrence_2}
\end{align}
when $X\left(n\right) \sim n^{1-\beta}$, where $0<\beta<\alpha$,
\begin{equation}\label{lowerbound_recurrence_3}
\liminf\limits_{n\rightarrow \infty}\frac{1}{n^{1-\beta}}\log \PP \left(\gamma_n\left(a\right)\le \zeta_n\left(a\right)-y\right) \ge -\sup\limits_{\theta\in\left(0,1\right)}\left[\frac{ \lambda_1 \theta e^{\lambda_1 y}}{r_1}-\log\left(1-\frac{\lambda_1}{r_1}\frac{\theta}{\theta-1}\right)\right].
\end{equation}
\end{proposition}
\begin{proof}\\
We first note that
\begin{align}
&\quad \PP\left(\gamma_n\left(a\right)\le \zeta_n\left(a\right)-y\right)\nonumber\\
&= \PP\left(\sup\limits_{u\in \left[0, u_n\left(y\right)\right]}Z_1^n\left(ut_n\right)+Z_2^n\left(ut_n\right)-an>0\right)\nonumber\\
&\ge
\PP\left(Z_1^n\left(u_n\left(y\right)t_n\right)+Z_2^n\left(u_n\left(y\right)t_n\right)-an>0\right)\nonumber\\
&=
\PP\left(\frac{\lambda_1}{r_1}n^{\alpha-1-\lambda_1u_n\left(y\right)/r}\left(Z_1^n\left(u_n\left(y\right)t_n\right)+Z_2^n\left(u_n\left(y\right)t_n\right)\right)>\frac{\lambda_1}{r_1}an^{\alpha-\lambda_1u_n\left(y\right)/r}\right)\label{eq:lowerbound_drop_sup_1}\\
&=
\PP\left(\frac{\lambda_1}{r_1}n^{\beta-1-\lambda_1u_n\left(y\right)/r}\left(Z_1^n\left(u_n\left(y\right)t_n\right)+Z_2^n\left(u_n\left(y\right)t_n\right)\right)>\frac{\lambda_1}{r_1}an^{\beta-\lambda_1u_n\left(y\right)/r}\right).\label{eq:lowerbound_drop_sup_2}
\end{align}
For case (1), we analyze \eqref{eq:lowerbound_drop_sup_1}. If we define
\begin{equation}\label{c_recurrence_1}
c\left(y;n\right)=\frac{\lambda_1}{r_1}\frac{\mu e^{y\lambda_1}}{\lambda_1+r}-\frac{\lambda_1}{r_1}an^{\alpha-\lambda_1u_n\left(y\right)/r},
\end{equation}
and the sequence of random variables
\begin{equation}\label{rv_recurrence_1}
Z_n=\frac{\lambda_1}{r_1}n^{\alpha-1-\lambda_1u_n\left(y\right)/r}\left(Z_1^n\left(u_n\left(y\right)t_n\right)+Z_2^n\left(u_n\left(y\right)t_n\right)\right)+c\left(y;n\right), n\geq 1
\end{equation}
we can obtain that
\begin{align*}
&\quad \PP\left(\frac{\lambda_1}{r_1}n^{\alpha-1-\lambda_1u_n\left(y\right)/r}\left(Z_1^n\left(u_n\left(y\right)t_n\right)+Z_2^n\left(u_n\left(y\right)t_n\right)\right)>\frac{\lambda_1}{r_1}an^{\alpha-\lambda_1u_n\left(y\right)/r}\right)\\
&=\PP\left(\frac{\lambda_1}{r_1}n^{\alpha-1-\lambda_1u_n\left(y\right)/r}\left(Z_1^n\left(u_n\left(y\right)t_n\right)+Z_2^n\left(u_n\left(y\right)t_n\right)\right)+c\left(y;n\right)>\frac{\lambda_1}{r_1}\frac{\mu e^{y\lambda_1}}{\lambda_1+r}\right)\\
&=
\PP\left(Z_n>\frac{\lambda_1}{r_1}\frac{\mu e^{y\lambda_1}}{\lambda_1+r}\right).
\end{align*}
From the previous display and \eqref{eq:lowerbound_drop_sup_1} it suffices to prove that
\begin{align}
\label{eq:lowerbound_recurrence_goal}
&\liminf_{n\to\infty}\frac{1}{n^{1-\alpha}}\log\PP\left(Z_n>\frac{\lambda_1\mu e^{y\lambda_1}}{(\lambda_1+r)r_1}\right)\geq
 -\sup\limits_{\theta\in\left(0,1\right)}\left[\frac{\mu \lambda_1 \theta e^{y \lambda_1}}{r_1\left(\lambda_1+r\right)}-\frac{\lambda_1\mu\theta}{r_1}\int_{0}^{\infty}\frac{e^{-rs}}{e^{\lambda_1s}-\theta}ds\right].
\end{align}
In pursuit of this goal, we show that the random variables $Z_n$ defined in \eqref{rv_recurrence_1} satisfy the conditions of the Gartner-Ellis Theorem \cite{Hollander}. Specifically for $\theta\in\mathbb{R}$ define
$$
\Lambda_n(\theta)=\log\EE\exp\left(\theta Z_n\right),
$$
then the conditions of the Gartner-Ellis theorem can be stated as
\begin{enumerate}
\item There exists a function $\Lambda\left(\theta\right)\in \left[-\infty, \infty\right]$ such that $\lim\limits_{n\rightarrow \infty}n^{\alpha-1}\Lambda_n\left(\theta n^{1-\alpha}\right)=\Lambda\left(\theta\right)$ for all $\theta \in \RR$,
\item $0\in int\left(D_{\Lambda}\right)$ where $int\left(D_{\Lambda}\right)=\{\theta\in \RR: \Lambda\left(\theta\right)<\infty \}$,
\item $\Lambda$ is lower semi-continuous on $\RR$,
\item $\Lambda$ is differentiable on $int\left(D_{\Lambda}\right)$,
\item	$\Lambda$ is steep at $\partial D_{\Lambda}$.
\end{enumerate}
Gartner-Ellis theorem tells us that if all the five conditions are satisfied, then for any open set $O\in \RR$, the sequence of probability measures $\{\PP_n\}_{n\ge 1}$ where $\PP_n\left(\cdot\right)=P\left(Z_n\in \cdot\right)$ satisfies
$$
\liminf\limits_{n\rightarrow \infty}\frac{1}{n^{1-\alpha}}\log \PP_n\left(O\right)\ge -\Lambda^*\left(O\right),
$$
where $\Lambda^*\left(x\right)=\sup\limits_{\theta\in \RR}\{\theta x -\Lambda \left(\theta\right)\}$ and we define $\Lambda^*\left(S\right)=\inf\limits_{x\in S}\Lambda^*\left(x\right)$.

The next result is to verify that this Theorem does in fact apply in our setting, and the proof can be found in Section \ref{proof_lower}.
\begin{proposition}\label{GE_cond}
The sequence of random variables $\{Z_n\}_{n\ge 1}$ as defined in (\ref{rv_recurrence_1}) satisfy the conditions of the Gartner-Ellis Theorem with $$
\Lambda(\theta)=\begin{cases}\frac{\lambda_1\mu\theta}{r_1}\int_0^\infty\frac{e^{-rs}}{e^{\lambda_1s}-\theta}ds,&\enskip \theta<1\\
\infty,&\enskip \theta\geq 1.
\end{cases}
$$
\end{proposition}

From Proposition \ref{GE_cond} we know that 
$$
\liminf_{n\to\infty}n^{\alpha-1}\log\PP\left(Z_n>\frac{\lambda_1\mu e^{y\lambda_1}}{(\lambda_1+r)r_1}\right)\geq -\inf_{x\in (\frac{\lambda_1\mu e^{y\lambda_1}}{(\lambda_1+r)r_1},\infty)}\sup_{\theta\in\mathbb{R}}\left(\theta x - \Lambda(\theta)\right).
$$
Our next result proves the existence and uniqueness of a maximizing $\theta,$ and relates the maximizing $\theta$ to the function $\Lambda$. The proof can be found in Section \ref{proof_lower}.
\begin{proposition}\label{proposition_recurrence_3}
	For any $x\in \left(\frac{\mu\lambda_1e^{y\lambda_1}}{r_1\left(\lambda_1+r\right)},\infty \right)$, there exists $\theta^*(x)\in \left(0,1\right)$, such that $x\theta^*(x)-\Lambda\left(\theta^*(x)\right)=\sup\limits_{\theta\in \RR}\left[\theta x-\Lambda\left(\theta\right)\right]$.
	
Furthermore, 
$$
\theta^*(x)=\Lambda^{\prime-1}(x),
$$
and in particular $\theta^*$ is a continuous function.
\end{proposition}

From Proposition \ref{proposition_recurrence_3} we know that 
$$
\inf_{x\in (\frac{\lambda_1\mu e^{y\lambda_1}}{(\lambda_1+r)r_1},\infty)}\sup_{\theta\in\mathbb{R}}\left(\theta x - \Lambda(\theta)\right)=\inf_{x\in (\frac{\lambda_1\mu e^{y\lambda_1}}{(\lambda_1+r)r_1},\infty)}\sup_{\theta\in (0,1)}\left(\theta x - \Lambda(\theta)\right).
$$
Define the function 
$$
h(x)=\sup_{\theta\in (0,1)}\left(\theta x-\Lambda(\theta)\right),
$$
then standard convex analysis tells us that $h$ is a convex function on $\mathbb{R}$ with minimum at 
$$
\Lambda^\prime(0)=\frac{\lambda_1\mu}{(\lambda_1+r)r_1}.
$$
In particular, $h$ is increasing on the set $(\frac{\lambda_1\mu}{(\lambda_1+r)r_1},\infty)$ and we conclude that
$$
\inf_{x\in (\frac{\lambda_1\mu e^{y\lambda_1}}{(\lambda_1+r)r_1},\infty)}\sup_{\theta\in (0,1)}\left(\theta x - \Lambda(\theta)\right)=\sup_{\theta\in (0,1)}\left(\frac{\mu \lambda_1 \theta e^{y \lambda_1}}{r_1\left(\lambda_1+r\right)}-\Lambda\left(\theta\right)\right),
$$
which recalling the definition of $\Lambda$ establishes \eqref{eq:lowerbound_recurrence_goal}.

The proofs for case (2) and case (3) are similar and thus we only point out the key differences. For case (2), we analyze \eqref{eq:lowerbound_drop_sup_1}. The key difference is that we need to redefine $c\left(y,n\right)=\frac{\lambda_1}{r_1}\frac{\left(\mu+\lambda_1+r\right)e^{y\lambda_1}}{\lambda_1+r}-\frac{\lambda_1}{r}an^{\alpha-\lambda_1 u_n\left(y\right)/r}$ (originally defined in \ref{c_recurrence_1}). The rest of the proof directly follows that of case (1). For case (3), we analyze \eqref{eq:lowerbound_drop_sup_2}. The key difference is that we need to redefine $c\left(y,n\right)=\frac{\lambda_1}{r_1}e^{\lambda_1 y}-\frac{\lambda_1}{r}an^{\beta-\lambda_1 u_n\left(y\right)/r}$. The rest of the proof directly follows that of case (1).
\end{proof}\qed %proof checked
%$******&&&&&&&&&&&&&&&&&&&&&$$$$$$$$$$$$$%%%%%%%%%%%%%%%%%%%%%%%dsf%
%$******&&&&&&&&&&&&&&&&&&&&&$$$$$$$$$$$$$%%%%%%%%%%%%%%%%%%%%%%%%
%$******&&&&&&&&&&&&&&&&&&&&&$$$$$$$$$$$$$%%%%%%%%%%%%%%%%%%%%%%%%
%$******&&&&&&&&&&&&&&&&&&&&&$$$$$$$$$$$$$%%%%%%%%%%%%%%%%%%%%%%%%
%$******&&&&&&&&&&&&&&&&&&&&&$$$$$$$$$$$$$%%%%%%%%%%%%%%%%%%%%%%%%

\subsection{Results for Crossover Time}\label{results_crossover}

In this section, we present results concerning the crossover time. We first establish the convergence in probability result for the crossover time.

\begin{theorem}\label{cip_crossover}
    Assume that $\alpha\in \left(0,1\right)$, $X\left(n\right)<n$, and $X\left(n\right)$ is non-decreasing in $n$, then for every $\epsilon>0$ we have that
	\begin{center}
		$\lim\limits_{n\rightarrow \infty}\PP\left(|\xi_n-\tau_n|>\epsilon\right)=0.$
	\end{center}
\end{theorem}

The proof of Theorem \ref{cip_crossover} is very similar to that of Theorem \ref{cip_recurrence} and Theorem 1 in \cite{JKJ2014}. Thus we omit the proof here.

We then prove three large deviations results for the probability of an early crossover. The proof for crossover time shares the main idea with that for recurrence time. However, it is more complicated due to the complexity of the stochastic term $Z_0^n\left(t\right)$ (compared to the deterministic term $an$ in recurrence time). Therefore, we find it necessary to provide a separate write-up. Similar to Theorem \ref{recurrence_LD}, we investigate three cases: (1) $X\left(n\right)=o\left(n^{1-\alpha}\right)$, (2) $X\left(n\right) \sim n^{1-\alpha}$ and (3) $X\left(n\right) \sim n^{1-\beta}$, where $0<\beta<\alpha$. Since the proofs for three cases are very similar, we shall only provide proofs for case (1).

\begin{theorem}\label{crossover_LD}
(1) When $X\left(n\right)=o\left(n^{1-\alpha}\right)$,
\begin{equation*}
\lim\limits_{n\rightarrow \infty}\frac{1}{n^{1-\alpha}}\log \PP \left(\tau_n\le \xi_n-y\right)= -\sup\limits_{\theta\in\left(0,1\right)}\left[\frac{\mu\lambda_1e^{\left(\lambda_1+r\right)y}\theta}{r_1\left(\lambda_1+r\right)}-\frac{\lambda_1\mu\theta}{r_1}\int_{0}^{\infty}\frac{e^{-rs}}{e^{\lambda_1s}-\theta}ds\right].
\end{equation*}

(2) When $X\left(n\right) \sim n^{1-\alpha}$,
\begin{align*}
& \quad \lim\limits_{n\rightarrow \infty}\frac{1}{n^{1-\alpha}}\log \PP \left(\tau_n\le \xi_n-y\right)\\
& = -\sup\limits_{\theta\in\left(0,1\right)}\left[\frac{ \lambda_1 \theta \left(\mu+\lambda_1+r\right)e^{\left(\lambda_1+r\right)y}}{r_1\left(\lambda_1+r\right)}-\log\left(1-\frac{\lambda_1}{r_1}\frac{\theta}{\theta-1}\right)-\frac{\lambda_1 \mu \theta}{r_1}\int_{0}^{\infty}\frac{ e^{-rs}}{e^{\lambda_1 s}-\theta}ds\right]
\end{align*}

(3) When $X\left(n\right) \sim n^{1-\beta}$, where $0<\beta<\alpha$,
\begin{equation*}
\lim\limits_{n\rightarrow \infty}\frac{1}{n^{1-\beta}}\log \PP \left(\tau_n\le \xi_n-y\right)= -\sup\limits_{\theta\in\left(0,1\right)}\left[\frac{ \lambda_1 \theta e^{\left(\lambda_1+r\right)y}}{r_1}-\log\left(1-\frac{\lambda_1}{r_1}\frac{\theta}{\theta-1}\right)\right].
\end{equation*}
\end{theorem}

We first establish the following upper bound. 

\begin{proposition}\label{crossover_LDUB}
For $y>0$, when $X\left(n\right)=o\left(n^{1-\alpha}\right)$,
\begin{equation}\label{upperbound_crossover_1}
\limsup\limits_{n\rightarrow \infty}\frac{1}{n^{1-\alpha}}\log \PP \left(\tau_n\le \xi_n-y\right)\le -\sup\limits_{\theta\in\left(0,1\right)}\left[\frac{\mu\lambda_1e^{\left(\lambda_1+r\right)y}\theta}{r_1\left(\lambda_1+r\right)}-\frac{\lambda_1\mu\theta}{r_1}\int_{0}^{\infty}\frac{e^{-rs}}{e^{\lambda_1s}-\theta}ds\right].
\end{equation}
\end{proposition}

\begin{proof}\\
Observe that $\tau_n\le \xi_n-y\Leftrightarrow \sup\limits_{t\le \xi_n-y}\left(Z_1^n\left(t\right)+Z_2^n\left(t\right)-Z_0^n\left(t\right)\right)>0$. It follows that 
\begin{equation*}
\PP\left(\tau_n\le \xi_n-y\right)=\PP\left(\sup\limits_{t\le \xi_n-y}\left(Z_1^n\left(t\right)+Z_2^n\left(t\right)-Z_0^n\left(t\right)\right)>0\right).
\end{equation*}
Abusing the notation, we define $t_n=\frac{1}{r}\log n$, and $u_n\left(y\right)=\left(\xi_n-y\right)/t_n$.
\begin{equation*}
\PP\left(\sup\limits_{t\le \xi_n-y}\left(Z_1^n\left(t\right)+Z_2^n\left(t\right)-Z_0^n\left(t\right)\right)>0\right)=\PP\left(\sup\limits_{u\in\left[0,u_n\left(y\right)\right]}\left(Z_1^n\left(ut_n\right)+Z_2^n\left(ut_n\right)-Z_0^n\left(ut_n\right)\right)>0\right).
\end{equation*}
It can be seen that $u_n\left(y\right)\rightarrow \alpha r/\left(\lambda_1+r\right)$ as $n\rightarrow \infty$. Therefore, if we define $a>0$ such that $a<\alpha r/\left(\lambda_1+r\right)$, then apply Proposition \ref{proposition_crossover_1} and \ref{proposition_crossover_2} to see that
\begin{align*}
& \quad \PP\left(\sup\limits_{u\in\left[0,u_n\left(y\right)\right]}\left(Z_1^n\left(ut_n\right)+Z_2^n\left(ut_n\right)-Z_0^n\left(ut_n\right)\right)>0\right)\\
& \le \PP\left(\sup\limits_{u\in\left[0,a\right]}\left(Z_1^n\left(ut_n\right)+Z_2^n\left(ut_n\right)-Z_0^n\left(ut_n\right)\right)>0\right)+\PP\left(\sup\limits_{u\in\left[a,u_n\left(y\right)\right]}\left(Z_1^n\left(ut_n\right)+Z_2^n\left(ut_n\right)-Z_0^n\left(ut_n\right)\right)>0\right)\\
& = O\left(\exp\left[-Cn^{1-a\left(1+\lambda_1/r\right)}\right]\right)+\PP\left(\sup\limits_{u\in\left[a,u_n\left(y\right)\right]}\left(Z_1^n\left(ut_n\right)+Z_2^n\left(ut_n\right)-Z_0^n\left(ut_n\right)\right)>0\right).
\end{align*}
We now analyze the case of an early recurrence occurring in the interval $\left[at_n, \xi_n-y\right]$. With Proposition \ref{proposition_crossover_3} and the quantities defined in \ref{a1_crossover_1}, \ref{a2_crossover_1}, \ref{a3_crossover_1}, and \ref{a4_crossover_1}, we have for $\delta \in \left(0,1\right)$ that
\begin{align*}
& \quad \PP\left(\sup\limits_{u\in\left[a,u_n\left(y\right)\right]}\left(Z_1^n\left(ut_n\right)+Z_2^n\left(ut_n\right)-Z_0^n\left(ut_n\right)\right)>0\right)\\
& =\PP\left(\sup\limits_{u\in\left[a,u_n\left(y\right)\right]}\left(A_1\left(u,n\right)+A_2\left(u,n\right)+A_3\left(u,n\right)+A_4\left(u,n\right)\right)>0\right)\\
&\le \PP\left(\sup\limits_{u\in\left[a,u_n\left(y\right)\right]}\left(A_1\left(u,n\right)+A_2\left(u,n\right)+\left(1-\delta\right) A_4\left(u,n\right)\right)>0\right)+\PP\left(\sup\limits_{u\in\left[a,u_n\left(y\right)\right]}\left(A_3\left(u,n\right)+\delta A_4\left(u,n\right)\right)>0\right)\\
& =\PP\left(\sup\limits_{u\in\left[a,u_n\left(y\right)\right]}\left(A_1\left(u,n\right)+A_2\left(u,n\right)+\left(1-\delta\right) A_4\left(u,n\right)\right)>0\right)+O\left(\exp\left[-Cn^{1-\alpha r/\left(\lambda_1+r\right)}\right]\right).
\end{align*}
We observe that
\begin{align*}
& \quad \PP\left(\sup\limits_{u\in\left[a,u_n\left(y\right)\right]}\left(A_1\left(u,n\right)+A_2\left(u,n\right)+\left(1-\delta\right) A_4\left(u,n\right)\right)>0\right)\\
& = \PP\left(\sup\limits_{u\in\left[a,u_n\left(y\right)\right]}n^{\alpha-1-\lambda_1 u/r}\left(A_1\left(u,n\right)+A_2\left(u,n\right)+\left(1-\delta\right) A_4\left(u,n\right)\right)>0\right)\\
& \le \PP\left(\sup\limits_{u\in\left[a,u_n\left(y\right)\right]}n^{\alpha-1-\lambda_1 u/r}\left(A_1\left(u,n\right)+A_2\left(u,n\right)\right)+ \sup\limits_{u\in\left[a,u_n\left(y\right)\right]}n^{\alpha-1-\lambda_1 u/r}\left(1-\delta\right)A_4\left(u,n\right)>0\right).
\end{align*}
If we define 
\begin{equation*}
c_1\left(y,n\right)=n^{\alpha-1-\lambda_1u_n\left(y\right)/r}\left(z_0^n\left(u_n\left(y\right)t_n\right)-z_1^n\left(u_n\left(y\right)t_n\right)-z_2^n\left(u_n\left(y\right)t_n\right)\right),
\end{equation*}
then we find that
\begin{align*}
& \quad \PP\left(\sup\limits_{u\in\left[a,u_n\left(y\right)\right]}n^{\alpha-1-\lambda_1 u/r}\left(A_1\left(u,n\right)+A_2\left(u,n\right)\right)+ \sup\limits_{u\in\left[a,u_n\left(y\right)\right]}n^{\alpha-1-\lambda_1 u/r}\left(1-\delta\right)A_4\left(u,n\right)>0\right)\\
& =\PP\left(\sup\limits_{u\in\left[a,u_n\left(y\right)\right]}n^{\alpha-1-\lambda_1 u/r}\left(A_1\left(u,n\right)+A_2\left(u,n\right)\right)> \left(1-\delta\right)c_1\left(y,n\right)\right),
\end{align*}
where the equality follows from the fact that
\begin{align*}
& \quad n^{\alpha-1-\lambda_1 u/r}\left(z_0^n\left(ut_n\right)-z_1^n\left(ut_n\right)-z_2^n\left(ut_n\right)\right)\\
& = n^{\alpha-1-\lambda_1 u/r}\left(n^{1-u}-X\left(n\right)n^{\frac{\lambda_1}{r}u}-\frac{\mu}{\lambda_1+r}n^{1-\alpha}e^{\lambda_1 u t_n}\left(1-e^{\left(\lambda_0-\lambda_1\right)u t_n}\right)\right)\\
& =n^{\alpha-\left(1+\lambda_1/r\right)u}-X\left(n\right)n^{\alpha-1}+\frac{\mu n^{-u\left(1+\lambda_1/r\right)}}{\lambda_1+r}-\frac{\mu}{\lambda_1+r}
\end{align*}
is a monotone decreasing function in $u$. Now we define
\begin{equation*}
c_2\left(y,n\right)=c_1\left(y,n\right)+n^{\alpha-1-\lambda_1u_n\left(y\right)/r}\left(z_1^n\left(u_n\left(y\right)t_n\right)+z_2^n\left(u_n\left(y\right)t_n\right)\right),
\end{equation*}
and recall that for $\theta\in \left(0,1\right)$ and $b>0$, $v_{n,\theta,b}=\frac{\lambda_1\theta}{r_1e^{\lambda_1 bt_n}}$. Applying a similar reasoning as in Proposition \ref{proposition_recurrence_1} and \eqref{eq:exponential_integral}, we obtain that
\begin{align*}
& \quad \PP\left(\sup\limits_{u\in\left[a,u_n\left(y\right)\right]}n^{\alpha-1-\lambda_1 u/r}\left(A_1\left(u,n\right)+A_2\left(u,n\right)\right)> \left(1-\delta\right)c_1\left(y,n\right)\right)\\
& \le \EE \exp\left(\frac{\mu}{n^{\alpha}}\int_{0}^{u_n\left(y\right)t_n}Z_0^n\left(s\right)\left(\phi_{u_n\left(y\right)t_n-s}\left(v_{n,\theta,u_n\left(y\right)}\right)-1\right)ds\right)\\
& \quad \times \EE\exp\left(v_{n,\theta,u_n\left(y\right)}Z_1^n\left(u_n\left(y\right)t_n\right)\right)\exp\left(-\frac{\lambda_1 \theta n^{1-\alpha}}{r_1}\left[\left(1-\delta\right)c_2\left(y,n\right)-\frac{\mu}{\lambda_1+r}n^{-a\left(1+\lambda_1/r\right)}\right]\right)\\
& \le \exp\left(\frac{\mu}{n^{\alpha-1}}\int_{0}^{u_n\left(y\right)t_n}e^{-rs}\left(\phi_{u_n\left(y\right)t_n-s}\left(v_{n,\theta,u_n\left(y\right)}\right)-1\right)ds\right)\exp\left(k_1\left(\log n\right)^2n^{1-2\alpha}\right)\\
& \quad \times \EE\exp\left(v_{n,\theta,u_n\left(y\right)}Z_1^n\left(u_n\left(y\right)t_n\right)\right) \exp\left(-\frac{\lambda_1 \theta n^{1-\alpha}}{r_1}\left[\left(1-\delta\right)c_2\left(y,n\right)-\frac{\mu}{\lambda_1+r}n^{-a\left(1+\lambda_1/r\right)}\right]\right).
\end{align*}
This gives us that 
\begin{align*}
& \quad \PP\left(\tau_n\le \xi_n-y\right)\\
& \le \PP\left(\sup\limits_{u\in\left[a,u_n\left(y\right)\right]}n^{\alpha-1-\lambda_1 u/r}\left(A_1\left(u,n\right)+A_2\left(u,n\right)\right)> \left(1-\delta\right)c_1\left(y,n\right)\right)+O\left(\exp\left(-Cn^{1-\omega}\right)\right)\\
& \le O\left(\exp\left(-Cn^{1-\omega}\right)\right)+\exp\left(\frac{\mu}{n^{\alpha-1}}\int_{0}^{u_n\left(y\right)t_n}e^{-rs}\left(\phi_{u_n\left(y\right)t_n-s}\left(v_{n,\theta,u_n\left(y\right)}\right)-1\right)ds\right)\exp\left(k_1\left(\log n\right)^2n^{1-2\alpha}\right)\\
& \quad \times \EE\exp\left(v_{n,\theta,u_n\left(y\right)}Z_1^n\left(u_n\left(y\right)t_n\right)\right) \exp\left(-\frac{\lambda_1 \theta n^{1-\alpha}}{r_1}\left[\left(1-\delta\right)c_2\left(y,n\right)-\frac{\mu}{\lambda_1+r}n^{-a\left(1+\lambda_1/r\right)}\right]\right).
\end{align*}

where $\omega = \max\left(\alpha r/\left(\lambda_1+r\right), a\left(1+\lambda_1/r\right)\right)$. Applying Lemma \ref{z1_moment_generate} and Lemma \ref{c_2_recurrence_1}, we obtain that
\begin{equation*}
\limsup\limits_{n\rightarrow \infty}\frac{1}{n^{1-\alpha}}\log \PP\left(\tau_n\le \xi_n-y\right)\le \frac{\mu \lambda_1}{r_1}\int_{0}^{\infty}\frac{\theta e^{-rs}}{e^{\lambda_1 s}-\theta}ds-\left(1-\delta\right)\frac{\mu \lambda_1 \theta e^{\left(\lambda_1+r\right)y}}{r_1\left(\lambda_1+r\right)}.
\end{equation*}
Letting $\delta\rightarrow 0$, and optimizing our upper bound over $0<\theta<1$ we complete the proof. \qed
\end{proof}
\\

We next establish a lower bound that matches Proposition \ref{crossover_LDUB}.
\begin{proposition}\label{crossover_LDLB}
For $y>0$, when $X\left(n\right)=o\left(n^{1-\alpha}\right)$,
\begin{equation}\label{lowerbound_crossover_1}
\liminf\limits_{n\rightarrow \infty}\frac{1}{n^{1-\alpha}}\log \PP \left(\tau_n\le \xi_n-y\right)\ge -\sup\limits_{\theta\in\left(0,1\right)}\left[\frac{\mu\lambda_1e^{\left(\lambda_1+r\right)y}\theta}{r_1\left(\lambda_1+r\right)}-\frac{\lambda_1\mu\theta}{r_1}\int_{0}^{\infty}\frac{e^{-rs}}{e^{\lambda_1s}-\theta}ds\right].
\end{equation}
\end{proposition}
\begin{proof}\\
We first note that
\begin{align*}
\PP\left(\tau_n\le \xi_n-y\right)&\ge  \PP\left(Z_1^n\left(u_n\left(y\right)t_n\right)+Z_2^n\left(u_n\left(y\right)t_n\right)-Z_0^n\left(u_n\left(y\right)t_n\right)>0\right).
\end{align*}
Now let $\{a_n\}_{n\in \NN}$ be a sequence such that $a_n\downarrow 0$. Then by Proposition \ref{proposition_crossover_4},
\begin{align}
&\quad \PP\left(Z_1^n\left(u_n\left(y\right)t_n\right)+Z_2^n\left(u_n\left(y\right)t_n\right)-Z_0^n\left(u_n\left(y\right)t_n\right)>0\right) \nonumber\\
&\ge \PP\left(Z_1^n\left(u_n\left(y\right)t_n\right)+Z_2^n\left(u_n\left(y\right)t_n\right)-Z_0^n\left(u_n\left(y\right)t_n\right)>0, Z_0^n\left(u_n\left(y\right)t_n\right)\le \left(1+a_n\right)z_0^n\left(u_n\left(y\right)t_n\right)\right)\nonumber\\
&\ge \PP\left(Z_1^n\left(u_n\left(y\right)t_n\right)+Z_2^n\left(u_n\left(y\right)t_n\right)>\left(1+a_n\right)z_0^n\left(u_n\left(y\right)t_n\right)\right)\nonumber\\
&\quad -\PP\left(Z_1^n\left(u_n\left(y\right)t_n\right)+Z_2^n\left(u_n\left(y\right)t_n\right)>\left(1+a_n\right)z_0^n\left(u_n\left(y\right)t_n\right), Z_0^n\left(u_n\left(y\right)t_n\right)> \left(1+a_n\right)z_0^n\left(u_n\left(y\right)t_n\right)\right)\nonumber\\
&\ge \PP\left(Z_1^n\left(u_n\left(y\right)t_n\right)+Z_2^n\left(u_n\left(y\right)t_n\right)>\left(1+a_n\right)z_0^n\left(u_n\left(y\right)t_n\right)\right)-\PP\left(Z_0^n\left(u_n\left(y\right)t_n\right)> \left(1+a_n\right)z_0^n\left(u_n\left(y\right)t_n\right)\right)\nonumber\\
&=\PP\left(Z_1^n\left(u_n\left(y\right)t_n\right)+Z_2^n\left(u_n\left(y\right)t_n\right)>\left(1+a_n\right)z_0^n\left(u_n\left(y\right)t_n\right)\right)-O\left(\exp\left(-Cn^{1-\alpha\left(2r+\lambda_1\right)/2\left(\lambda_1+r\right)}\right)\right),\label{eq:crossover_drop_sup}
\end{align}
where $C$ is some positive number. If we define
\begin{equation}\label{c_3}
c_3\left(y,n\right)=\frac{\lambda_1}{r_1}\frac{\mu e^{\left(\lambda_1+r\right)y}}{\lambda_1+r}\frac{\mu}{\mu+n^{\alpha}\left(\lambda_1+r\right)}-\frac{\lambda_1}{r_1}n^{\alpha}e^{\left(\lambda_1+r\right)y}\frac{X\left(n\right)n^{\alpha}\left(\lambda_1+r\right)}{n\left(\mu+\lambda_1n^{\alpha}+rn^{\alpha}\right)},
\end{equation}
and the sequence of random variables
\begin{equation}\label{rv}
\tilde{Z}_n=\frac{\lambda_1}{r_1}n^{\alpha-1-\lambda_1u_n\left(y\right)/r}\left(Z_1^n\left(u_n\left(y\right)t_n\right)+Z_2^n\left(u_n\left(y\right)t_n\right)-a_nz_0^n\left(u_n\left(y\right)t_n\right)\right)+c_3\left(y,n\right),
\end{equation}
we can obtain that
\begin{align*}
&\quad \PP\left(Z_1^n\left(u_n\left(y\right)t_n\right)+Z_2^n\left(u_n\left(y\right)t_n\right)>\left(1+a_n\right)z_0^n\left(u_n\left(y\right)t_n\right)\right)\\
&= \PP\left(\frac{\lambda_1}{r_1}n^{\alpha-1-\lambda_1u_n\left(y\right)/r}\left(Z_1^n\left(u_n\left(y\right)t_n\right)+Z_2^n\left(u_n\left(y\right)t_n\right)-a_nz_0^n\left(u_n\left(y\right)t_n\right)\right)>\frac{\lambda_1}{r_1}c_2\left(y,n\right)\right)\\
&=\PP\left(\frac{\lambda_1}{r_1}n^{\alpha-1-\lambda_1u_n\left(y\right)/r}\left(Z_1^n\left(u_n\left(y\right)t_n\right)+Z_2^n\left(u_n\left(y\right)t_n\right)-a_nz_0^n\left(u_n\left(y\right)t_n\right)\right)+c_3\left(y,n\right)>\frac{\lambda_1}{r_1}\frac{\mu e^{\left(\lambda_1+r\right)y}}{\lambda_1+r}\right)\\
&=\PP\left(\tilde{Z}_n>\frac{\lambda_1}{r_1}\frac{\mu e^{\left(\lambda_1+r\right)y}}{\lambda_1+r}\right).
\end{align*}
From the previous display and \eqref{eq:crossover_drop_sup} it suffices to prove that
\begin{align}
\label{eq:lowerbound_crossover_goal}
&\liminf_{n\to\infty}\frac{1}{n^{1-\alpha}}\log\PP\left(\tilde{Z}_n>\frac{\lambda_1}{r_1}\frac{\mu e^{\left(\lambda_1+r\right)y}}{\lambda_1+r}\right)\geq
-\sup\limits_{\theta\in\left(0,1\right)}\left[\frac{\mu\lambda_1e^{\left(\lambda_1+r\right)y}\theta}{r_1\left(\lambda_1+r\right)}-\frac{\lambda_1\mu\theta}{r_1}\int_{0}^{\infty}\frac{e^{-rs}}{e^{\lambda_1s}-\theta}ds\right].
\end{align}
In pursuit of this goal let us define the sequence of functions
\begin{equation*}
\tilde{\Lambda}_n\left(\theta\right)=\log\EE\exp\left(\theta \tilde{Z}_n\right).
\end{equation*}
From Proposition \ref{G_E} we know that the random variables $\{\tilde{Z}_n\}_{n\geq 1}$ satisfy the conditions of the Gartner-Ellis theorem with function
$$
\Lambda\left(\theta\right)=\begin{cases}\frac{\lambda_1\mu\theta}{r_1}\int_0^\infty\frac{e^{-rs}}{e^{\lambda_1s}-\theta}ds,&\enskip \theta<1\\
\infty,&\enskip \theta\geq 1.
\end{cases}
$$
Therefore
$$
\liminf_{n\to\infty}n^{\alpha-1}\log\PP\left(\tilde{Z}_n>\frac{\lambda_1}{r_1}\frac{\mu e^{\left(\lambda_1+r\right)y}}{\lambda_1+r}\right)\geq -\inf_{x\in \left(\frac{\lambda_1\mu e^{\left(\lambda_1+r\right)y}}{\left(\lambda_1+r\right)r_1},\infty\right)}\sup_{\theta\in\mathbb{R}}\left(\theta x - \Lambda\left(\theta\right)\right).
$$
Applying the same argument in the proof of Proposition \ref{recurrence_LDLB}, we can obtain that
%From Proposition \ref{proposition_recurrence_3} we know that 
%$$
%\inf_{x\in \left(\frac{\lambda_1\mu %e^{\left(\lambda_1+r\right)y}}{\left(\lambda_1+r\right)r_1},\infty\right)}\sup_{\theta\in\m%athbb{R}}\left(\theta x - \Lambda\left(\theta\right)\right)=\inf_{x\in %\left(\frac{\lambda_1\mu %e^{\left(\lambda_1+r\right)y}}{\left(\lambda_1+r\right)r_1},\infty\right)}\sup_{\theta\in %\left(0,1\right)}\left(\theta x - \Lambda\left(\theta\right)\right).
%$$
%Define the function 
%$$
%h\left(x\right)=\sup_{\theta\in \left(0,1\right)}\left(\theta %x-\Lambda\left(\theta\right)\right),
%$$
%then standard convex analysis tells us that $h$ is a convex function on $\mathbb{R}$ with %minimum at 
%$$
%\Lambda^\prime\left(0\right)=\frac{\lambda_1\mu}{\left(\lambda_1+r\right)r_1}.
%$$
%In particular, $h$ is increasing on the set %$\left(\frac{\lambda_1\mu}{\left(\lambda_1+r\right)r_1},\infty\right)$ and we conclude that
$$
\inf_{x\in \left(\frac{\lambda_1\mu e^{\left(\lambda_1+r\right)y}}{\left(\lambda_1+r\right)r_1},\infty\right)}\sup_{\theta\in \left(0,1\right)}\left(\theta x - \Lambda\left(\theta\right)\right)=\sup_{\theta\in \left(0,1\right)}\left(\frac{\mu \lambda_1 \theta e^{\left(\lambda_1+r\right)y}}{r_1\left(\lambda_1+r\right)}-\Lambda\left(\theta\right)\right)
$$
which recalling the definition of $\Lambda$ establishes \eqref{eq:lowerbound_crossover_goal}. \qed
\end{proof}

\subsection{Large Deviations for the Conditioned Process}\label{LD_condition_sec}
In this section, we present analysis of the large deviations rate for early recurrence conditioned on a given number of clones at the time $\zeta_n-y$. Note that we make a strong assumption in this section that sensitive cells have deterministic exponential decay (i.e. $Z_0^n\left(t\right)=z_0^n\left(t\right)$). We also assume that $X(n)=0$. 

Let $S_n$ be the number of clones generated in the time period $\left(0,\zeta_n-y\right)$, which will survive to the time $\zeta_n-y$, and let $\hat{Z}_S^n\left(t\right)$ be the number of mutants at time $t$, which belong to those clones. Let $E_n$ be the number of clones generated in the time period $\left(0,\zeta_n-y\right)$, which will go extinct before time $\zeta_n-y$, and let $\hat{Z}_E^n\left(t\right)$ be the number of mutants at time $t$, which belong to those clones. Clearly, $\left(S_n, \hat{Z}_S^n\right)$ and $\left(E_n, \hat{Z}_E^n\right)$ are independent due to the thinning property of Poisson process.

%Let $C_n$ be the number of clones generated in the time period $\left(0,\zeta_n-y\right)$, which will survive to the time $\zeta_n-y$, but will eventually extinct, and let $\hat{Z}_C^n\left(t\right)$ be the number of mutants at time $t$, which belong to those clones.

Note that the birth-death process $Z$ with initial condition 1 has the following p.m.f (see expression 8 of \cite{Durrett}),
\begin{align*}
\PP\left(Z\left(t\right)=n|Z\left(0\right)=1\right)=\left(\frac{\lambda_1}{r_1-d_1e^{-\lambda_1 t}}\right)\left(\frac{r_1\left(1-e^{-\lambda_1 t}\right)}{r_1-d_1 e^{-\lambda_1 t}}\right)^{n-1}\left(\frac{\lambda_1 e^{-\lambda_1 t}}{r_1-d_1 e^{-\lambda_1 t}}\right), n>0,
\end{align*}
and 
\begin{align*}
\PP\left(Z\left(t\right)=0|Z\left(0\right)=1\right)=\frac{d_1\left(1-e^{-\lambda_1 t}\right)}{r_1-d_1 e^{-\lambda_1 t}}.
\end{align*}
With the above p.m.f we can obtain the conditional mean of $Z_1(t)$
\begin{align*}
& \quad \EE\left[\hat{Z}_S^n\left(\zeta_n-y\right)|S_n=K\right]\\
& = K \frac{\int_{0}^{\zeta_n-y}\frac{\lambda_1}{r_1-d_1e^{-\lambda_1\left(\zeta_n-y-s\right)}}e^{-rs}\EE \left[ Z\left(\zeta_n-y-s\right) | Z\left(\zeta_n-y-s\right)>0\right]ds}{\int_{0}^{\zeta_n-y}\frac{\lambda_1}{r_1-d_1e^{-\lambda_1\left(\zeta_n-y-s\right)}}e^{-rs}ds}\\
%& = K \frac{\int_{0}^{\zeta_n-y}\frac{\lambda_1}{r_1-d_1e^{-\lambda_1\left(\zeta_n-y-s\right)}}e^{-rs} \frac{r_1-d_1 e^{-\lambda_1 \left(\zeta_n-y-s\right)}}{\lambda_1 e^{-\lambda_1 \left(\zeta_n-y-s\right)}}ds}{\int_{0}^{t}\frac{\lambda_1}{r_1-d_1e^{-\lambda_1\left(\zeta_n-y-s\right)}}e^{-rs}ds}\\
& = K \frac{\int_{0}^{\zeta_n-y}e^{\lambda_1 \left(\zeta_n-y\right)-\left(\lambda_1+r\right)s}ds}{\int_{0}^{\zeta_n-y}\frac{\lambda_1}{r_1-d_1e^{-\lambda_1\left(\zeta_n-y-s\right)}}e^{-rs}ds},
\end{align*}
where the first equality is due to the uniformity of arrival times for a Poisson process, and the integral is obtained by conditioning on the mutation time. For $\theta<\frac{\lambda_1}{r_1}$ we can compute the conditional moment generating function of $Z_1(t)$,
\begin{align*}
\EE\left[e^{\theta e^{-\lambda_1 \left(\zeta_n-y\right)} \hat{Z}_S^n\left(\zeta_n-y\right)}|S_n=K\right]=\left(\frac{\int_{0}^{\zeta_n-y}\frac{\lambda_1}{r_1-d_1e^{-\lambda_1 \left(\zeta_n-y-s\right)}}\phi_{\zeta_n-y-s}\left(\theta e^{-\lambda_1 \left(\zeta_n-y\right)}\right)e^{-rs}ds}{\int_{0}^{\zeta_n-y}\frac{\lambda_1}{r_1-d_1e^{-\lambda_1 \left(\zeta_n-y-s\right)}}e^{-rs}ds}\right)^{K},
\end{align*}
where
\begin{align*}
\phi_{\zeta_n-y-s}\left(\theta e^{-\lambda_1 \left(\zeta_n-y\right)}\right)&=\frac{e^{\theta e^{-\lambda_1 \left(\zeta_n-y\right)}}\lambda_1 e^{-\lambda_1 \left(\zeta_n-y-s\right)}}{r_1\left(1-e^{\theta e^{-\lambda_1 \left(\zeta_n-y\right)}}\right)+e^{-\lambda_1 \left(\zeta_n-y-s\right)}\left(r_1 e^{\theta e^{-\lambda_1 \left(\zeta_n-y\right)}}-d_1\right)}\\
& \rightarrow \frac{\lambda_1 e^{\lambda_1 s}}{\lambda_1 e^{\lambda_1 s}-r_1 \theta} \text{ as } n\rightarrow \infty.
\end{align*}

Define $A_{n,a}$ to be the event that $S_n=\lfloor a\frac{\lambda_1 \mu}{r r_1}n^{1-\alpha} \rfloor$, where $a\in \left(0, e^{\lambda_1 y}\right)$ and denote the corresponding conditional probability measure by $\PP_{A_{n,a}}\left(\cdot\right)=\PP\left(\cdot | A_{n,a}\right)$. Note that $a=1$ corresponds to the mean of $S_n$.
We consider $e^{\lambda_1 y}$ as the upper bound for $a$ due to the following equation
\begin{align*}
\lim\limits_{n\rightarrow \infty}\frac{1}{n}\EE[\hat{Z}_S^n\left(\zeta_n-y\right)|S_n=\lfloor e^{\lambda_1 y}\frac{\lambda_1 \mu}{r r_1}n^{1-\alpha} \rfloor]=1,
\end{align*}
which indicates that conditioned on the event that the number of clones which survive to the time $\zeta_n-y$ is greater than or equal to $e^{\lambda_1 y}\frac{\lambda_1 \mu}{r r_1}n^{1-\alpha}$, early recurrence is no longer a rare event with large deviation rate $n^{1-\alpha}$.

We then have the following result for the large deviations rate of early recurrence conditioned on the event $A_{n,a}$.

\begin{proposition}\label{LD_condition}
Assume that $\alpha\in \left(0,1\right)$ and $y>0$,
\begin{align*}
& \quad \lim\limits_{n\rightarrow \infty}\frac{1}{n^{1-\alpha}}\log \PP_{A_{n,a}}\left(\gamma_n\le \zeta_n-y\right)\\
& =-\sup\limits_{\theta\in \left(0,\frac{\lambda_1}{r_1}\right)}\left[\frac{\theta \mu e^{\lambda_1 y}}{\lambda_1+r}-a\frac{\lambda_1 \mu}{r r_1}\log\left( r\int_{0}^{\infty}\frac{\lambda_1 e^{\lambda_1 s}}{\lambda_1 e^{\lambda_1 s}-r_1 \theta}e^{-rs}ds\right)\right].
\end{align*}	
\end{proposition} 
%\textcolor{red}{ZICHENG: Can you give any intuition on why this rate function is so different from Thm 4 (1)? For example, why the log, or the integrands in the two rate functions are slightly different, why?}
We notice that the rate function in Proposition \ref{LD_condition} is different from that in Case (1) of Theorem \ref{recurrence_LD}. In particular, the integral term in Proposition \ref{LD_condition} is within a log operator, while its counterpart in Case (1) of Theorem \ref{recurrence_LD} is not. Both integral terms in Proposition \ref{LD_condition} and Theorem \ref{recurrence_LD} are related to the moment generating function of the number of type-2 mutants. The major difference is that the moment generating function is conditioned on the number of clones in Proposition \ref{LD_condition}, while it's not in Theorem \ref{recurrence_LD}. If we omit the complicated details of the moment generating function, and only focus on its general form, we can observe that in Theorem \ref{recurrence_LD}, the moment generating function has the form of an exponential function, while in Proposition \ref{LD_condition}, it has the form of a power function. Roughly speaking, when conditioned on the number of clones, the moment generating function is just a single term out of the expansion of an exponential function (the moment generating function without conditioning), and thus has the form of a power function. When the moment generating function is applied with a log operator in Theorem \ref{recurrence_LD}, the log operator is cancelled with the exponential operator. However, the log operator is preserved in Proposition \ref{LD_condition}. 

We would like to calculate the most likely number of clones given early recurrence has occurred, i.e.,
$$
\mbox{argmax}_{a\in(0,e^{\lambda_1y}]}\PP\left(A_{n,a}|\gamma_n\le \zeta_n-y\right)
$$
which is obviously equivalent to calculating
$$
\mbox{argmax}_{a\in(0,e^{\lambda_1y}]}\PP\left(\gamma_n\le \zeta_n-y |A_{n,a}\right)\PP(A_{n,a}).
$$
Since this is a difficult problem we instead focus on the slightly easier problem of finding the number of clones with the smallest exponential decay rate conditioned on early recurrence, i.e.,
\begin{align}
    \label{eq:max_clones}
\max_{a\in (0,e^{\lambda_1y}]}\lim_{n\to\infty}\frac{1}{n^{1-\alpha}}\log \PP\left(\gamma_n\le \zeta_n-y |A_{n,a}\right)\PP(A_{n,a}).
\end{align}
We can use Proposition \ref{LD_condition} for the exponential decay rate of the conditional probability term. To find the exponential decay rate of $\PP(A_{n,a})$ first notice that $S_n$ is a Poisson random variable with mean
\begin{align*}
L_n=\mu n^{1-\alpha}\int_0^{\zeta_n-y}e^{-rs}\left(\frac{\lambda_1}{r_1-d_1e^{-\lambda_1 \left(\zeta_n-y-s\right)}}\right) ds.
\end{align*}
Applying Stirling's approximation to the Poisson random variable $S_n$ and noting that $L_n\approx n^{1-\alpha}\frac{\lambda_1 \mu}{r r_1}$ we see that
\begin{align*}
\PP\left(A_{n,a}\right)= \exp\left[-\left(a \frac{\lambda_1 \mu}{r r_1}\log\left(a\right)-a \frac{\lambda_1 \mu}{r r_1}+\frac{\lambda_1 \mu}{r r_1}\right) n^{1-\alpha}(1+o(1))\right].
\end{align*}
Thus we can write \eqref{eq:max_clones} as
\begin{align}
\min\limits_{a\in \left(0,e^{\lambda_1 y}\right]}\left[\sup\limits_{\theta\in \left(0,\frac{\lambda_1}{r_1}\right)}\left[\frac{\theta \mu e^{\lambda_1 y}}{\lambda_1+r}- a \frac{\lambda_1 \mu}{r r_1}\log\left( r\int_{0}^{\infty}\frac{\lambda_1 e^{\lambda_1 s}}{\lambda_1 e^{\lambda_1 s}-r_1 \theta}e^{-rs}ds\right)\right]+\frac{\lambda_1 \mu}{r r_1}\left(a\log\left(a\right)-a+1\right)\right]. \label{optimization}
\end{align}

We can check that the minimum can be achieved in $\left(0,e^{\lambda_1 y}\right]$ as (i) the objective function is continuous in $\left(0,e^{\lambda_1 y}\right]$; (ii) the objective function is decreasing in $\left(0,1\right]$. Since 
\begin{align*}
\frac{\theta \mu e^{\lambda_1 y}}{\lambda_1+r}- a \frac{\lambda_1 \mu}{r r_1}\log\left( r\int_{0}^{\infty}\frac{\lambda_1 e^{\lambda_1 s}}{\lambda_1 e^{\lambda_1 s}-r_1 \theta}e^{-rs}ds\right)
\end{align*}
is concave on $\theta\in \left(0, \frac{\lambda_1}{r_1}\right)$, % by examining its second order derivative
%\begin{align*}
%- a \frac{\lambda_1 \mu}{r r_1}\frac{\int_{0}^{\infty}\frac{r_1^2\lambda_1 e^{\lambda_1 s}}{\left(\lambda_1 e^{\lambda_1 s}-r_1 \theta\right)^3}e^{-rs}ds\int_{0}^{\infty}\frac{\lambda_1 e^{\lambda_1 s}}{\lambda_1 e^{\lambda_1 s}-r_1 \theta}e^{-rs}ds-\left(\int_{0}^{\infty}\frac{r_1 \lambda_1 e^{\lambda_1 s}}{ \left(\lambda_1 e^{\lambda_1 s}-r_1 \theta\right)^2}e^{-rs}ds\right)^2}{\left(\int_{0}^{\infty}\frac{\lambda_1 e^{\lambda_1 s}}{\lambda_1 e^{\lambda_1 s}-r_1 \theta}e^{-rs}ds\right)^2}\le 0.
%\end{align*}
we can apply Sion's minimax theorem \cite{SION1958} and obtain that the optimization problem \ref{optimization} is equivalent to
\begin{align*}
\sup\limits_{\theta\in \left(0,\frac{\lambda_1}{r_1}\right)}\left[\min\limits_{a\in \left[1,e^{\lambda_1 y}\right]}\left[\frac{\theta \mu e^{\lambda_1 y}}{\lambda_1+r}- a \frac{\lambda_1 \mu}{r r_1}\log\left( r\int_{0}^{\infty}\frac{\lambda_1 e^{\lambda_1 s}}{\lambda_1 e^{\lambda_1 s}-r_1 \theta}e^{-rs}ds\right)+\frac{\lambda_1 \mu}{r r_1}\left(a\log\left(a\right)-a+1\right)\right]\right].
\end{align*}
Given $\theta\in \left(0, \frac{\lambda_1}{r_1}\right)$, we can obtain the solution $a^*$ to the inner minimization problem:
\begin{align*}
a^*=\left( r\int_{0}^{\infty}\frac{\lambda_1 e^{\lambda_1 s}}{\lambda_1 e^{\lambda_1 s}-r_1 \theta}e^{-rs}ds\right)\wedge e^{\lambda_1 y}.
\end{align*}
%Derivative for $a$:
%\begin{align*}
%-\frac{\lambda_1 \mu}{r r_1}\log\left( r\int_{0}^{\infty}\frac{\lambda_1 e^{\lambda_1 s}}{\lambda_1 e^{\lambda_1 s}-r_1 %\theta}e^{-rs}ds\right)+\frac{\lambda_1 \mu}{r r_1}\log\left(a\right)
%\end{align*}
Let $\theta_1$ be the solution to 
\begin{align*}
\frac{e^{\lambda_1 y}}{r}=\int_{0}^{\infty}\frac{\lambda_1 e^{\lambda_1 s}}{\lambda_1 e^{\lambda_1 s}-r_1 \theta}e^{-rs}ds.
\end{align*}
When $\theta\le \theta_1$, the optimization problem for $\theta$ reduces to
\begin{align*}
\frac{\theta \mu e^{\lambda_1 y}}{\lambda_1+r}-\frac{\lambda_1 \mu}{r r_1}\left( r\int_{0}^{\infty}\frac{\lambda_1 e^{\lambda_1 s}}{\lambda_1 e^{\lambda_1 s}-r_1 \theta}e^{-rs}ds\right)+\frac{\lambda_1 \mu}{r r_1},
\end{align*}
which is concave. Due to the concavity it suffices to consider the critical point $\theta_2$ that is the solution to 
%and has first order derivative:
%\begin{align*}
%\frac{\mu e^{\lambda_1 y}}{\lambda_1+r}-\lambda_1 \mu \left( \int_{0}^{\infty}\frac{\lambda_1 e^{\lambda_1 s}}{\left(\lambda_1 e^{\lambda_1 s}-r_1 \theta\right)^2}e^{-rs}ds\right).
%\end{align*}
\begin{align*}
\frac{ e^{\lambda_1 y}}{\lambda_1 \left(\lambda_1+r\right)}= \int_{0}^{\infty}\frac{\lambda_1 e^{\lambda_1 s}}{\left(\lambda_1 e^{\lambda_1 s}-r_1 \theta\right)^2}e^{-rs}ds.
\end{align*}

Denote by $a^{**}$ the solution to the optimization problem \ref{optimization}, we can obtain that $a^{**}= r\int_{0}^{\infty}\frac{\lambda_1 e^{\lambda_1 s}}{\lambda_1 e^{\lambda_1 s}-r_1 \theta_2}e^{-rs}ds$ from the following lemma.
\begin{lemma}\label{theta_1_theta_2}
Let $\theta_1$ be the solution to 
\begin{align*}
\frac{e^{\lambda_1 y}}{r}=\int_{0}^{\infty}\frac{\lambda_1 e^{\lambda_1 s}}{\lambda_1 e^{\lambda_1 s}-r_1 \theta}e^{-rs}ds,
\end{align*}
and $\theta_2$ be the solution to
\begin{align*}
\frac{ e^{\lambda_1 y}}{\lambda_1 \left(\lambda_1+r\right)}= \int_{0}^{\infty}\frac{\lambda_1 e^{\lambda_1 s}}{\left(\lambda_1 e^{\lambda_1 s}-r_1 \theta\right)^2}e^{-rs}ds.
\end{align*}
We have $\theta_2 < \theta_1$.
\end{lemma}

\begin{proof}\\
Define 
\begin{align*}
f_1\left(\theta\right)=r\int_{0}^{\infty}\frac{ e^{\lambda_1 s}}{e^{\lambda_1 s}-\theta}e^{-rs}ds,
\end{align*}
and 
\begin{align*}
f_2\left(\theta\right)=\left(\lambda_1+r\right)\int_{0}^{\infty}\frac{e^{\lambda_1 s}}{\left(e^{\lambda_1 s}- \theta\right)^2}e^{-rs}ds.
\end{align*}
We can obtain that for $\theta \in \left(0,1\right)$,
\begin{align*}
f_2\left(\theta\right)& > \left(\lambda_1+r\right)\int_{0}^{\infty}\frac{e^{\lambda_1 s}}{e^{\lambda_1 s} \left(e^{\lambda_1 s}- \theta\right)}e^{-rs}ds\\
& =\left(\lambda_1+r\right)\int_{0}^{\infty}\frac{e^{\lambda_1 s}}{\left(e^{\lambda_1 s}- \theta\right)}e^{-\left(\lambda_1+r\right)s}ds
\end{align*}
We then work with a random variable $X_1=\frac{e^{\lambda_1 s_1}}{e^{\lambda_1 s_1}- \theta}$, where $s_1\sim \exp\left(r\right)$, and a random variable $X_2=\frac{e^{\lambda_1 s_2}}{e^{\lambda_1 s_2}- \theta}$, where $s_2\sim \exp\left(\lambda_1+r\right)$. Hence, 
\begin{align*}
f_1\left(\theta\right) =\EE\left[X_1\right], \text{ and } f_2\left(\theta\right) > \EE\left[X_2\right].
\end{align*}
Notice that
\begin{align*}
\PP\left(X_1>x\right) = \left\{ \begin{array}{cc} 
1, & \hspace{5mm} x\le 1 \\
1-e^{-\frac{r}{\lambda_1}\log\left(\frac{x\theta}{x-1}\right)}, & \hspace{5mm} x\in \left(1, \frac{1}{1-\theta}\right) \\
0, & \hspace{5mm} x > \frac{1}{1-\theta}\\
\end{array} \right.
\end{align*}
and 
\begin{align*}
\PP\left(X_2>x\right) = \left\{ \begin{array}{cc} 
1, & \hspace{5mm} x\le 1 \\
1-e^{-\frac{\lambda_1+r}{\lambda_1}\log\left(\frac{x\theta}{x-1}\right)}, & \hspace{5mm} x\in \left(1, \frac{1}{1-\theta}\right) \\
0, & \hspace{5mm} x\ge \frac{1}{1-\theta}.\\
\end{array} \right.
\end{align*}
Hence, $\PP\left(X_2>x\right)\ge \PP\left(X_1>x\right)$, which indicates that $X_2$ is stochastically larger than $X_1$, and thus $\EE\left[X_2\right]\ge \EE\left[X_1\right]$ (See page 404-405 of \cite{Ross_SP}). Therefore, for $\theta \in \left(0,1\right)$, 
\begin{align*}
f_2\left(\theta\right) > f_1\left(\theta\right).
\end{align*}
Since $\theta_1$ is the solution to $f_1\left(\theta\right)=e^{\lambda_1 y}$, and $\theta_2$ is the solution to $f_2\left(\theta\right)=e^{\lambda_1 y}$, the desired result follows by the monotonicity of $f_1\left(\theta\right)$ and $f_2\left(\theta\right)$. \qed
\end{proof}
\\

We observe that $a^{**}$ increases in $y$, which implies that early recurrence is an indicator of a large variety in mutant clones. In Figure \ref{fig:y_a}, we report the value of $a^{**}$ for different values of $y$ obtained from one parameter set. Note that as $y$ increases, the most likely number of clones increases since $a^{**}$ increases. This means that the earlier recurrence occurs, the more likely it is to have a larger number of distinct clones in the recurrent tumor. Also, note that whenever $y$ is strictly larger than $0$, $a^{**}$ is strictly larger than $1$, so when early recurrence occurs, we are always in a setting where the number of mutant clones is above their expected value.

% Denote by $a^{**}$ the solution to the optimization problem \ref{optimization}, it's easy to check that if $\theta_2\ge \theta_1$, $a^{**}=e^{\lambda_1 y}$, and $a^{**}= r\int_{0}^{\infty}\frac{\lambda_1 e^{\lambda_1 s}}{\lambda_1 e^{\lambda_1 s}-r_1 \theta_2}e^{-rs}ds$ otherwise. Moreover, we observe that $a^{**}$ increases in $y$ in both cases, which implies that early recurrence is an indicator of a large variety in mutant clones. Ideally, we would like to give conditions on parameters to describe the relation between $\theta_1$ and $\theta_2$. Based on our computations, we conjecture that $\theta_1\ge \theta_2$. Unfortunately, we are not able to establish this result.

% In Figure \ref{fig:y_a}, we report the value of $a^{**}$ for different values of $y$ obtained from one parameter set. Note that as $y$ increases, the most likely number of clones increases since $a^{**}$ increases. This means that the earlier recurrence occurs the more likely it is to have a larger number of distinct clones in the recurrent tumor. Also note that whenever $y$ is strictly larger than $0$, $a^{**}$ is strictly larger than $1$, so when early recurrence occurs we are always in a setting where the number of mutant clones is above their expected value.
%\textcolor{red}{ZICHENG: You need to add some discussion of $a^{**}$. What's the point of proving these things if we don't learn anything about the structure of the problem?}
\begin{figure}
	\begin{center}
		\includegraphics[width=1\textwidth]{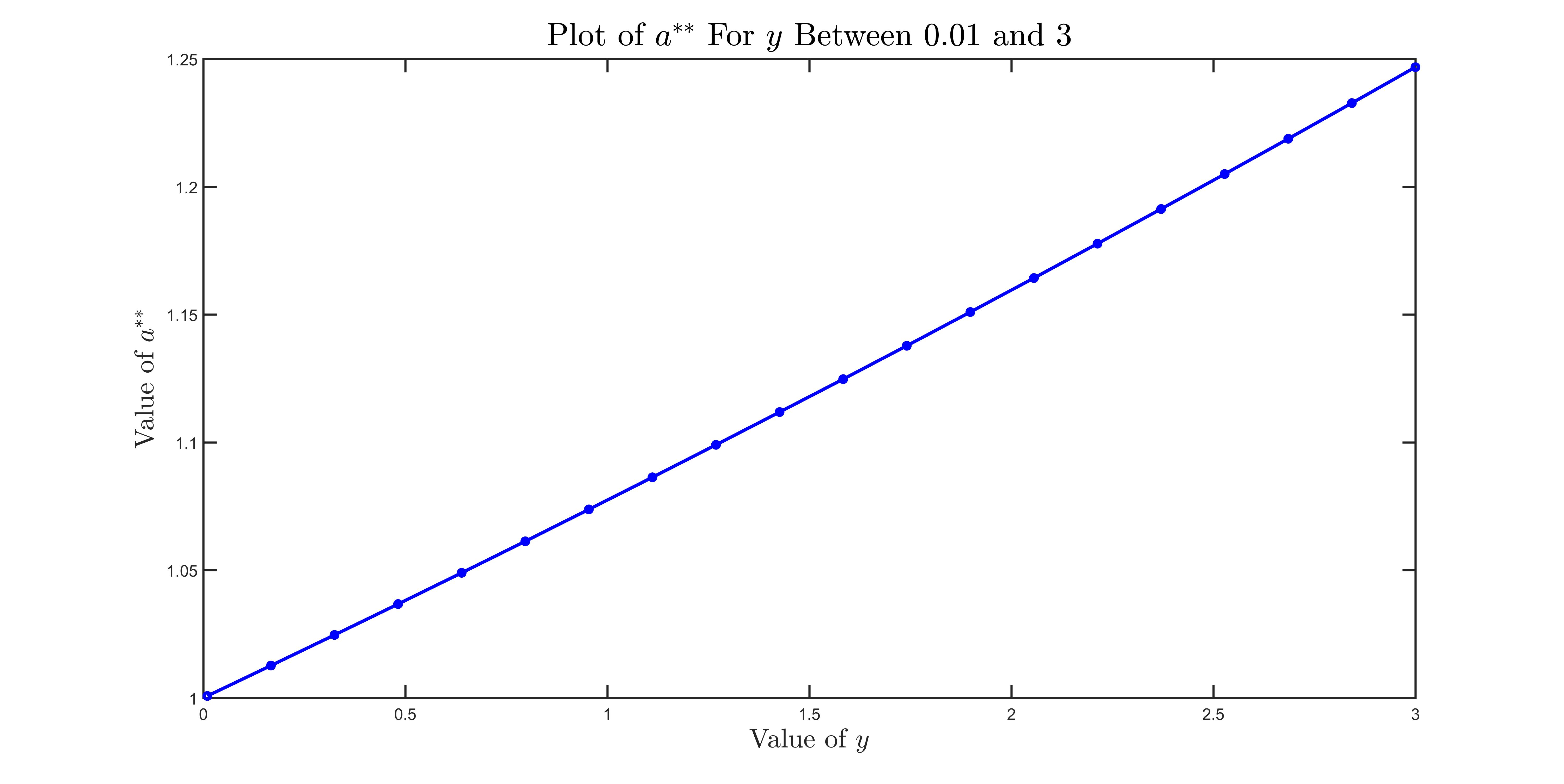}
	\end{center}
	\caption{Demonstration of $a^{**}$, where $\lambda_0=-0.2$, $r_1=0.4$, and $\lambda_1=0.2$.}
	\label{fig:y_a}
\end{figure}

\section{Appendix}\label{proof}

\subsection{Proof of Theorem \ref{cip_recurrence}}\label{proof_cip_recurrence}

We first state the following extension of Lemma 2 in \cite{JK2013}. In \cite{JK2013} they assume $b\in [0,1]$, but that restriction is not necessary, and we thus have the following.

\begin{lemma}\label{z_2_limit}
	For $b\in \left(0,\infty\right)$, $\epsilon>0$,
	
	\begin{equation*}
	\lim\limits_{n\rightarrow \infty}\PP \left(\sup\limits_{u\in [0,b]} n^{-\frac{\lambda_1 u}{r}-1+\alpha}|Z_2^n\left(ut_n\right)-z_2^n\left(ut_n\right)|> \epsilon\right)= 0.
	\end{equation*}
\end{lemma}
We then show the following result.
\begin{lemma}\label{cip_crossover_mean}
	Assume that $\lim\limits_{n\rightarrow \infty}X\left(n\right)=\infty$, and $X\left(n\right)<n$. Then for $b\in \left(0,\infty\right)$
	
	\begin{center}
		$\lim\limits_{n\rightarrow \infty}\PP \left(\sup\limits_{u\in [0,b]} \frac{1}{X\left(n\right)}n^{-\frac{\lambda_1}{r}u}|Z_1^n\left(ut_n\right)-z_1^n\left(u t_n\right)|>\epsilon\right)=0.$
	\end{center}
\end{lemma}
\begin{proof}\\
First observe that $\frac{1}{X\left(n\right)}n^{-\frac{\lambda_1}{r}u}Z_1^n\left(ut_n\right)-1$ is a martingale with respect to $u$. We have that as $n\rightarrow \infty$,
\begin{align*}
\EE [\frac{1}{X\left(n\right)}n^{-\frac{\lambda_1}{r}b}Z_1^n\left(bt_n\right)-1]^2&=\left(\frac{1}{X\left(n\right)}n^{-\frac{\lambda_1}{r}b}\right)^2\EE[Z_1^n\left(bt_n\right)]^2-1\\
&=\left(\frac{1}{X\left(n\right)}n^{-\frac{\lambda_1}{r}b}\right)^2 \text{Var}Z_1^n\left(bt_n\right)\\
&=\left(\frac{1}{X\left(n\right)}n^{-\frac{\lambda_1}{r}b}\right)^2X\left(n\right)\frac{r_1+d_1}{\lambda_1}n^{\frac{\lambda_1}{r}b}\left(n^{\frac{\lambda_1}{r}b}-1\right)\\
&=\frac{1}{X\left(n\right)}\frac{r_1+d_1}{\lambda_1}\left(1-n^{-\frac{\lambda_1}{r}b}\right)\\
&\rightarrow 0, \text{ as }n\to \infty.
\end{align*}
Hence, by Doob's martingale inequality:
\begin{align*}
& \quad \limsup\limits_{n\rightarrow \infty}\PP \left(\sup\limits_{u\in [0,b]}\frac{1}{X\left(n\right)}n^{-\frac{\lambda_1}{r}u}|Z_1^n\left(ut_n\right)-z_1^n\left(ut_n\right)|>\epsilon\right)\\
&=\limsup\limits_{n\rightarrow \infty}\PP \left(\sup\limits_{u\in [0,b]}|\frac{1}{X\left(n\right)}n^{-\frac{\lambda_1}{r}u}Z_1^n\left(ut_n\right)-1|>\epsilon\right)\\
&\le \limsup\limits_{n\rightarrow \infty}\frac{\EE [\frac{1}{X\left(n\right)}n^{-\frac{\lambda_1}{r}a}Z_1^n\left(at_n\right)-1]^2}{{\epsilon}^2}\\
&= 0.
\end{align*} \qed
\end{proof}
%The following results from Lemma 1 of \cite{JK2013} will be useful. \textcolor{red}{Zicheng: This statement looks identical to Lemma 1 of \cite{JK2013} so I deleted it as a lemma. There might be some labels messed up from this.}
%Recall that $Z_1$ is a binary branching process starting from size one with birth rate $r_1$ and death rate $d_1$.  Then for $0<s<bt_n$,
%	\begin{align}\label{moment_z2}
%	\EE[Z_2^n\left(bt_n\right)^2]&= \left(\mu n^{-\alpha}\right)^2  \int_{0}^{bt_n} \int_{0}^{bt_n}\EE[Z_0^n\left(s\right)Z_0^n\left(y\right)]e^{\lambda_1 \left(bt_n-s\right)} e^{\lambda_1 \left(bt_n-y\right)}dsdy\nonumber\\
%	& \quad +
%\mu n^{-\alpha}\int_{0}^{bt_n}\EE Z_0^n\left(s\right)\EE [Z_1\left(bt_n-s\right)^2]ds\nonumber\\
%\EE[Z_0^n\left(s\right)Z_2\left(bt_n\right)]&=\mu n^{-\alpha} \int_{0}^{bt_n}\EE[Z_0^n\left(y\right)Z_0^n\left(s\right)]e^{\lambda_1\left(bt_n-y\right)}dy\nonumber\\
%	\Var[Z_2^n\left(bt_n\right)]&= \left(\mu n^{-\alpha}\right)^2 \int_{0}^{bt_n}\int_{0}^{bt_n}\Cov \left(Z_0^n\left(s\right), Z_0^n\left(y\right)\right)e^{\lambda_1\left(2bt_n-\left(s+y\right)\right)}dsdy\nonumber\\
%	&\quad + \mu n^{-\alpha} \int_{0}^{bt_n} \EE Z_0^n\left(s\right)\EE[Z_1\left(bt_n-s\right)^2]ds.
%\end{align}

We now begin the proof of Theorem \ref{cip_recurrence}.\\
\begin{proof}\\
We will only prove the setting where $X(n)\to\infty$ as $n\to\infty$, if $\sup_nX(n)<\infty$ the proof is similar and in fact easier since $Z_1^n$ is dominated by $Z_2^n$. For $\epsilon>0$, define
\begin{center}
	$\hat{u}_n^-\left(\epsilon\right)=\frac{\zeta_n\left(a\right)-\epsilon}{t_n}$ and $\hat{u}_n^+\left(\epsilon\right)=\frac{\zeta_n\left(a\right)+\epsilon}{t_n}$.
\end{center}
%Since $X\left(n\right)$ is non-decreasing in $n$, by monotone convergence theorem, $\{X\left(n\right)\}$ has a finite limit if and only if the sequence is bounded. First we assume that $X\left(n\right)$ increases to infinity. 
Notice that
\begin{align*}
\PP \left(\gamma_n\left(a\right)<\zeta_n\left(a\right)-\epsilon\right)&=\PP \left(\frac{\gamma_n\left(a\right)}{t_n}<\hat{u}_n^-\left(\epsilon\right)\right)\\
&=\PP \left(\sup\limits_{u\le \hat{u}_n^-\left(\epsilon\right)}\left(Z_1^n\left(ut_n\right)+Z_2^n\left(ut_n\right)-an>0\right)\right)\\
&\le \PP \left(\hat{A}_1\left(n,\epsilon\right)+\hat{A}_2\left(n,\epsilon\right)+\hat{A}_3\left(n,\epsilon\right)>0\right),
\end{align*}
where
\begin{align*}
& \hat{A}_1\left(n,\epsilon\right)= \sup\limits_{u\le \hat{u}_n^-\left(\epsilon\right)} \min\{\frac{\lambda_1+r}{\mu}n^{-\frac{\lambda_1u}{r}-1+\alpha}, \frac{1}{X\left(n\right)}n^{-\frac{\lambda_1}{r}u},\frac{1}{an}\}\left(Z_1^n\left(ut_n\right)-z_1^n\left(ut_n\right)\right)\\
& \hat{A}_2\left(n,\epsilon\right)= \sup\limits_{u\le \hat{u}_n^-\left(\epsilon\right)}\min\{\frac{\lambda_1+r}{\mu}n^{-\frac{\lambda_1u}{r}-1+\alpha}, \frac{1}{X\left(n\right)}n^{-\frac{\lambda_1}{r}u},\frac{1}{an}\}\left(Z_2^n\left(ut_n\right)-z_2^n\left(ut_n\right)\right)\\
& \hat{A}_3\left(n,\epsilon\right)= \sup\limits_{u\le \hat{u}_n^-\left(\epsilon\right)} \min\{\frac{\lambda_1+r}{\mu}n^{-\frac{\lambda_1u}{r}-1+\alpha}, \frac{1}{X\left(n\right)}n^{-\frac{\lambda_1}{r}u},\frac{1}{an}\}\left(z_1^n\left(ut_n\right)+z_2^n\left(ut_n\right)-an\right).
\end{align*}
Notice that when $u\le \hat{u}_n^-\left(\epsilon\right)$, we have $\min\{\frac{\lambda_1+r}{\mu}n^{-\frac{\lambda_1u}{r}-1+\alpha}, \frac{1}{X\left(n\right)}n^{-\frac{\lambda_1}{r}u},\frac{1}{an}\}=\frac{1}{an}$, since we assume that $X(n)<n$. Lemma \ref{cip_crossover_mean} shows that $\hat{A}_1\left(n,\epsilon\right)\rightarrow 0$ in probability. Lemma \ref{z_2_limit} shows that $\hat{A}_2\left(n,\epsilon\right)\rightarrow 0$ in probability. It now remains to establish that $\hat{A}_3\left(n,\epsilon\right)$ is a negative number bounded away from zero. First note that via monotonicity, $\hat{A}_3\left(n,\epsilon\right)$ can be simplified to 
\begin{align*}
\hat{A}_3\left(n,\epsilon\right)&=\frac{1}{an}\left(z_1^n\left(\hat{u}_n^-\left(\epsilon\right)t_n\right)+z_2^n\left(\hat{u}_n^-\left(\epsilon\right)t_n\right)-an\right)\\
& =\frac{1}{an}\left(z_1^n\left(\hat{u}_n^-\left(\epsilon\right)t_n\right)+z_2^n\left(\hat{u}_n^-\left(\epsilon\right)t_n\right)\right)-1\\
& \le e^{-\lambda_1\epsilon}-1
\end{align*}
and the desired result follows. It now remains to study
\begin{align*}
\PP \left(\gamma_n\left(a\right)>\zeta_n\left(a\right)+\epsilon\right)&=\PP \left(\sup\limits_{u\le \hat{u}_n^+\left(\epsilon\right)}\left(Z_1^n\left(ut_n\right)+Z_2^n\left(ut_n\right)-an\right)<0\right)\\
&\le \PP \left(Z_1^n\left(\hat{u}_n^+\left(\epsilon\right)t_n\right)+Z_2^n\left(\hat{u}_n^+\left(\epsilon\right)t_n\right)-an<0\right)\\
&=\PP \left(B_1\left(n,\epsilon\right)+B_2\left(n,\epsilon\right)+B_3\left(n,\epsilon\right)<0\right),
\end{align*} 
where
\begin{align*}
& B_1\left(n,\epsilon\right)=\min\{\frac{\lambda_1+r}{\mu}n^{-\frac{\lambda_1\hat{u}_n^+}{r}-1+\alpha}, \frac{1}{X\left(n\right)}n^{-\frac{\lambda_1}{r}\hat{u}_n^+},\frac{1}{an}\} \left(Z_1^n\left(\hat{u}_n^+\left(\epsilon\right)t_n\right)-z_1^n\left(\hat{u}_n^+\left(\epsilon\right)t_n\right)\right)\\
& B_2\left(n,\epsilon\right)=\min\{\frac{\lambda_1+r}{\mu}n^{-\frac{\lambda_1\hat{u}_n^+}{r}-1+\alpha}, \frac{1}{X\left(n\right)}n^{-\frac{\lambda_1}{r}\hat{u}_n^+},\frac{1}{an}\}\left(Z_2^n\left(\hat{u}_n^+\left(\epsilon\right)t_n\right)-z_2^n\left(\hat{u}_n^+\left(\epsilon\right)t_n\right)\right)\\
& B_3\left(n,\epsilon\right)=\min\{\frac{\lambda_1+r}{\mu}n^{-\frac{\lambda_1\hat{u}_n^+}{r}-1+\alpha}, \frac{1}{X\left(n\right)}n^{-\frac{\lambda_1}{r}\hat{u}_n^+},\frac{1}{an}\}\left(z_1^n\left(\hat{u}_n^+\left(\epsilon\right)t_n\right)+z_2^n\left(\hat{u}_n^+\left(\epsilon\right)t_n\right)-an\right).
\end{align*}
Again using Lemma \ref{cip_crossover_mean} and Lemma \ref{z_2_limit}, we can show that as $n\rightarrow \infty$, $B_1\left(n,\epsilon\right)\rightarrow 0$, and $B_2\left(n,\epsilon\right)\rightarrow 0$. When  $\min\{\frac{\lambda_1+r}{\mu}n^{-\frac{\lambda_1\hat{u}_n^+}{r}-1+\alpha}, \frac{1}{X\left(n\right)}n^{-\frac{\lambda_1}{r}\hat{u}_n^+},\frac{1}{an}\}=\frac{1}{an}$,
\begin{align*}
& \quad \frac{1}{an}\left(z_1^n\left(\hat{u}_n^+\left(\epsilon\right)t_n\right)+z_2^n\left(\hat{u}_n^+\left(\epsilon\right)t_n\right)-an\right)\\
&=\frac{1}{an}\left(z_1^n\left(\hat{u}_n^+\left(\epsilon\right)t_n\right)+z_2^n\left(\hat{u}_n^+\left(\epsilon\right)t_n\right)\right)-1\\
&\ge e^{\lambda_1\epsilon}-1.
\end{align*}
When  $\min\{\frac{\lambda_1+r}{\mu}n^{-\frac{\lambda_1\hat{u}_n^+}{r}-1+\alpha}, \frac{1}{X\left(n\right)}n^{-\frac{\lambda_1}{r}\hat{u}_n^+},\frac{1}{an}\}=n^{-\frac{\lambda_1\hat{u}_n^+}{r}-1+\alpha}$, and $n$ large enough,
\begin{align*}
& \quad \frac{\lambda_1+r}{\mu}n^{-\frac{\lambda_1\hat{u}_n^+}{r}-1+\alpha}\left(z_1^n\left(\hat{u}_n^+\left(\epsilon\right)t_n\right)+z_2^n\left(\hat{u}_n^+\left(\epsilon\right)t_n\right)-an\right)\\
&=\frac{\lambda_1+r}{\mu}n^{-\frac{\lambda_1\hat{u}_n^+}{r}-1+\alpha} an \left(\frac{1}{an}\left(z_1^n\left(\hat{u}_n^+\left(\epsilon\right)t_n\right)+z_2^n\left(\hat{u}_n^+\left(\epsilon\right)t_n\right)\right)-1\right)\\
&\ge \frac{1-n^{-\beta}}{1-n^{-\beta}e^{\frac{\lambda_0-\lambda_1}{r}\epsilon}}e^{-\lambda_1 \epsilon}\left(e^{\lambda_1\epsilon}-1\right) \quad \text{ for some }0<\beta<\alpha\\
&\ge ke^{-\lambda_1 \epsilon}\left(e^{\lambda_1\epsilon}-1\right)\quad \text{ for some }0<k<1.
\end{align*}
When  $\min\{\frac{\lambda_1+r}{\mu}n^{-\frac{\lambda_1\hat{u}_n^+}{r}-1+\alpha}, \frac{1}{X\left(n\right)}n^{-\frac{\lambda_1}{r}\hat{u}_n^+},\frac{1}{an}\}=\frac{1}{X\left(n\right)}n^{-\frac{\lambda_1}{r}\hat{u}_n^+}$,
\begin{align*}
& \quad \frac{1}{X\left(n\right)}n^{-\frac{\lambda_1}{r}\hat{u}_n^+}\left(z_1^n\left(\hat{u}_n^+\left(\epsilon\right)t_n\right)+z_2^n\left(\hat{u}_n^+\left(\epsilon\right)t_n\right)-an\right)\\
&=\frac{1}{X\left(n\right)}n^{-\frac{\lambda_1}{r}\hat{u}_n^+} an \left(\frac{1}{an}\left(z_1^n\left(\hat{u}_n^+\left(\epsilon\right)t_n\right)+z_2^n\left(\hat{u}_n^+\left(\epsilon\right)t_n\right)\right)-1\right)\\
&\ge e^{-\lambda_1 \epsilon}\left(e^{\lambda_1\epsilon}-1\right)
\end{align*}
Hence, the desired result follows. \qed\end{proof}
\\

%%%%%%%%%%%%%%%%%%%%%%%%%%%%%%%%%%%%%%%%%%%%%%%%%%%%%%%%%%%%%%%%%%%%%%%%%%%%%%%%%%%%%%%%%%%%%%%%%%%%%%%%%%%%%%%%%%%%%%%%%%%%%%%%%%%%%%%%%%%%%%%%%%%%%%%

\subsection{Lemmas and Propositions for Proof of Theorem \ref{recurrence_LD}}\label{recurrence_proof}
%Since proofs for the three cases are similar, we shall only provide a complete proof for the first case: $X\left(n\right)=o\left(n^{1-\alpha}\right)$. 
Recall that, we define $t_n=\frac{1}{r}\log \left(n\right)$, and $u_n\left(y\right)=\left(\zeta_n\left(a\right)-y\right)/t_n$. Recall the definition of the moment-generating function $\phi_t$ from \eqref{generating_function}, and the upper bound on the domain of finiteness $\bar\theta_t$. For $\theta\in (0,1)$ and $b>0$, we have defined in \ref{eq:vdef} that 
\begin{align*}
& v_{n,\theta,b}=\frac{\lambda_1\theta}{r_1e^{\lambda_1 bt_n}}.
\end{align*}

\begin{remark}\label{finite_mgf}
	In order for the moment $\phi_{bt_n-s}\left(v_{n,\theta,b}\right)$ to be defined for all $s\in \left(0, bt_n\right)$, we require
	\begin{equation*}
	v_{n,\theta,b}\le \log\left(1+\frac{\lambda_1}{r_1\left(e^{\lambda_1 bt_n}-1\right)}\right).
	\end{equation*}
	Using the inequality $\log\left(1+x\right)\ge x/\left(x+1\right)$, which holds for $x\ge 0$, we obtain that	\begin{align*}
 \log\left(1+\frac{\lambda_1}{r_1\left(e^{\lambda_1 bt_n}-1\right)}\right)\geq\frac{\lambda_1}{r_1e^{\lambda_1 bt_n}-d_1}.
	\end{align*}
Hence, for the exponential moment to exist for all $s\in \left(0, bt_n\right)$, it will be sufficient for the purpose of our analysis to require that
	\begin{equation*}
	\frac{\theta \lambda_1}{r_1 e^{\lambda_1 b t_n}}<\frac{\lambda_1}{r_1e^{\lambda_1 bt_n}-d_1},
	\end{equation*}
which holds for $\theta< 1$.
\end{remark} %proof checked

\subsubsection{Results for Upper Bound Proof}\label{proof_upper}
We first establish that the moment generating function of $Z$ is uniformly bounded if we choose the argument appropriately.
\begin{lemma}\label{phi_property}
For $\theta\in \left(-\infty,1\right)$ and $b>0$, 
$$
\sup\left\{|\phi_{bt_n-s}\left(v_{n,\theta,b}\right)-1|; s\in (0,bt_n), n\geq 1\right\}<\infty.
$$
\end{lemma}
%\textcolor{blue}{change from $\theta \in (0,1)$ to $\theta\in \left(-\infty,1\right)$}
\begin{proof}\\
If $v_{n,\theta,b}\le 0$, then $|\phi_{bt_n-s}\left(v_{n,\theta,b}\right)-1|\le 1$. If $v_{n,\theta,b}>0$ then we know that $\phi_t\left(v_{n,\theta,b}\right)\geq 1$ for all $t$. Thus it suffices to establish an upper bound on $\phi_{bt_n-s}(v_{n,\theta,b})$.

First fix $n\geq 1$, and observe that $Z=\{Z(t);t\geq 0\}$ is a submartingale; therefore, $\{\exp[v_{n,\theta,b}Z(t)];0\leq t\leq bt_n\}$ is also a submartingale. In particular
$$
\sup\left\{\phi_{bt_n-s}\left(v_{n,\theta,b}\right); s\in (0,bt_n), n\geq 1\right\}=\sup\left\{\phi_{bt_n}\left(v_{n,\theta,b}\right); n\geq 1\right\}.
$$
Using the formula for the moment generating function \eqref{generating_function} and the result $e^x=1+x+O(x^2)$ as $x\to 0$, we can see that
$$
\lim_{n\to\infty}\phi_{bt_n}\left(v_{n,\theta,b}\right)=\frac{1-d_1\theta/r_1}{1-\theta}.
$$\qed
\end{proof} %proof checked

Next we define the function
\begin{align*}
\psi_s\left(\theta\right)=\EE\left[\exp\left(\theta\left(Z_0\left(s\right)-e^{-rs}\right)\right)\right],
\end{align*}
where $Z_0$ evolves according to a sub-critical binary branching process starting from a single individual, and each individual has offspring distribution $\{p_n\}_{n\geq 0\}}$. In addition we assume that $p_0+p_1<1$. If we define the generating function $f(s)=\sum_{n}s^np_n$, we further assume that there exists $s_0>1$ such that $f(s)<\infty$ for $s\in [0,s_0)$. 
We will show that the second derivative of $\psi$ is bounded under these assumptions.

\begin{lemma}\label{second_derivative}
There exists $z_0>0$ such that 
$$
\sup\left\{\psi_t''(z);z\in [-z_0,z_0], t>0\right\}=k_1<\infty.
$$
\end{lemma}
%\textcolor{blue}{change from $z\le z_0$ to $z\in [-z_0,z_0]$}
\begin{proof}\\
Denote by $K$ the total number of individuals that ever live in the branching process $Z_0$.
Theorem 2.1 of \cite{Nakayama} states that under our assumptions on $Z_0$ there exists $z_0>0$ such that $z\leq z_0$ implies $\EE\exp\left(zK\right)<\infty$. Since $Z_0(t)\leq K$ for any $t\geq 0$ we know that
$\EE\exp\left(zZ_0(t)\right)<\infty$ for $z\leq z_0$.

For $z\in [-z_0,z_0]$ we may differentiate the moment generating function of $Z_0\left(t\right)$ to obtain that
\begin{align*}
\psi_t''\left(z\right)&=e^{-2rt}\exp\left(-ze^{-rt}\right)\EE \exp\left(zZ_0\left(t\right)\right)+\exp\left(-ze^{-rt}\right)\frac{d^2}{dz^2}\EE \exp\left(zZ_0\left(t\right)\right)\\
&\quad \quad -2e^{-rt}\exp\left(-ze^{-rt}\right)\frac{d}{dz}\EE \exp\left(zZ_0\left(t\right)\right)\\
&\leq
\exp\left(z_0\right)\EE \exp\left(zZ_0\left(t\right)\right)+\exp\left(z_0\right)\frac{d^2}{dz^2}\EE \exp\left(zZ_0\left(t\right)\right)\\
&\leq
\exp\left(z_0\right)\EE \exp\left(z_0Z_0\left(t\right)\right)+\exp\left(z_0\right)\frac{d^2}{dz^2}\EE \exp\left(z_0Z_0\left(t\right)\right),
\end{align*}
where we used the fact that $\frac{d^k}{dz^k}\EE\exp\left(zZ_0(t)\right)>0$ for any positive integer $k$ in the two inequalities. Since $\EE\exp\left(z_0Z_0(t)\right)\leq \EE\exp\left(z_0K\right)$ and $\EE\left[Z_0(t)^2e^{z_0Z_0(t)}\right]\leq\EE\left[K^2e^{z_0K}\right]$ for all $t>0$ the result follows.
 \end{proof} %proof checked
\qed
\\

We can use the previous two Lemmas to obtain an upper bound on a functional of the subcritical process $Z_0^n$. We restate Proposition \ref{proposition_exp} here and provide the proof.
\\

\noindent \textbf{Proposition \ref{proposition_exp}} For $\theta\in \left(-\infty,1 \right)$, there exists $k_1>0$ such that
	\begin{equation*}
	\EE\exp\left(\frac{\mu}{n^{\alpha}}\int_{0}^{bt_n}\left(\phi_{bt_n-s}\left(v_{n,\theta,b}\right)-1\right)\left(Z_0^n\left(s\right)-ne^{-rs}\right)ds\right)\le \exp\left(k_1\left(\log n\right)^2n^{1-2\alpha}\right)
	\end{equation*}
for all $b>0$.

%\textcolor{blue}{change from $\theta\in (0,1)$ to $\theta\in (-\infty, 1)$}
\begin{proof}\\
If we define $E_U$ as the expectation operator with respect to the Uniform($0,bt_n$) probability measure, we have that
\begin{align*}
& \quad \EE\exp\left(\frac{\mu}{n^{\alpha}}\int_{0}^{bt_n}\left(\phi_{bt_n-s}\left(v_{n,\theta,b}\right)-1\right)\left(Z_0^n\left(s\right)-ne^{-rs}\right)ds\right)\\
&=\EE\exp\left(\frac{\mu}{n^{\alpha}}bt_nE_U\left[\left(\phi_{bt_n-s}\left(v_{n,\theta,b}\right)-1\right)\left(Z_0^n\left(s\right)-ne^{-rs}\right)\right]\right),
\end{align*}
and we may use Jensen's Inequality to see that
\begin{align*}
& \quad \EE\exp\left(\frac{\mu}{n^{\alpha}}bt_nE_U\left[\left(\phi_{bt_n-s}\left(v_{n,\theta,b}\right)-1\right)\left(Z_0^n\left(s\right)-ne^{-rs}\right)\right]\right)\\
&\le \frac{1}{bt_n}\int_{0}^{bt_n}\EE \exp\left(\frac{\mu bt_n}{n^{\alpha}}\left(\phi_{bt_n-s}\left(v_{n,\theta,b}\right)-1\right)\left(Z_0^n\left(s\right)-ne^{-rs}\right)\right)ds.
\end{align*}
If we define
\begin{equation}\label{w_n}
w_n\left(s,bt_n\right)=\frac{\mu bt_n}{n^{\alpha}}\left(\phi_{bt_n-s}\left(v_{n,\theta,b}\right)-1\right),
\end{equation} 
then from the independence of the $n$ initial cells we have
\begin{equation}\label{upper_1}
\EE \exp\left(\frac{\mu bt_n}{n^{\alpha}}\left(\phi_{bt_n-s}\left(v_{n,\theta,b}\right)-1\right)\left(Z_0^n\left(s\right)-ne^{-rs}\right)\right)=\left(\psi_s\left(w_n\left(s,bt_n\right)\right)\right)^n.
\end{equation}

We will develop an upper bound from (\ref{upper_1}) through a Taylor Expansion. Applying Lemma \ref{phi_property} to (\ref{w_n}), we obtain for $n$ sufficiently large that
\begin{equation*}
|w_n\left(s,bt_n\right)|\le \frac{M_1bt_n}{n^{\alpha}},
\end{equation*}
where $M_1$ is a sufficiently large positive constant. This leads to the following inequality
\begin{equation}\label{inequality_k}
|w_n\left(s,bt_n\right)|\le k_0n^{-\alpha}\log \left(n\right),
\end{equation}
where $k_0$ is a sufficiently large positive constant.

From Lemma \ref{second_derivative} we know that for $\theta\in [-z_0,z_0]$, $\psi_s(\theta)<\infty$ for all $s>0$.
We also know from (\ref{inequality_k}), that $|w_n\left(s,bt_n\right)|\rightarrow 0$ as $n\rightarrow \infty$ uniformly for all $s\in \left(0, bt_n\right)$. Therefore we may conclude that for $n$ sufficiently large $|w_n\left(s,bt_n\right)|$ will be sufficiently small so that $\psi_s$ is infinitely differentiable on $\left(-|w_n\left(s,bt_n\right)|,|w_n\left(s,bt_n\right)| \right)$ for all $s\in \left(0, bt_n\right)$. Hence if we take $n$ sufficiently large and $z\in \left(0\wedge w_n\left(s,bt_n\right), 0 \vee w_n\left(s,bt_n\right)\right)$, we may perform the following Taylor expansion of $\psi_s$ around $0$:
\begin{align*}
\psi_s\left(w_n\left(s,bt_n\right)\right)&=1+\frac{\left(w_n\left(s,bt_n\right)\right)^2}{2}\psi_s''\left(z\right)\\
&\leq
1+\frac{k_1\left(\log n\right)^2}{n^{2\alpha}}\le \exp\left(k_1\left(\log n\right)^2n^{-2\alpha}\right),
\end{align*}
where the first inequality follows from Lemma \ref{second_derivative}, and $k_1$ is a sufficiently large positive constant. We can apply this bound along with (\ref{upper_1}) to conclude that
\begin{equation*}
\EE\exp\left(\frac{\mu}{n^{\alpha}}\int_{0}^{bt_n}\left(\phi_{bt_n-s}\left(v_{n,\theta,b}\right)-1\right)\left(Z_0^n\left(s\right)-ne^{-rs}\right)ds\right)\le \exp\left(k_1\left(\log n\right)^2n^{1-2\alpha}\right).
\end{equation*}\end{proof} \qed %proof checked
\\

Our next result concerns the exponential moments of $Z_1$. 
%In the case $X(n)=o(n^{1-\alpha})$ we show that pre-existing mutants can be safely ignored.
\begin{lemma}\label{z1_moment_generate}
If $\theta\in \left(-\infty,1\right)$, then
$$\lim\limits_{n\rightarrow \infty} \frac{1}{X(n)}\log\left(\EE\left[\exp \left(v_{n,\theta,b} Z_1^n\left(b t_n\right)\right)\right]\right)=\log\left(1-\frac{\lambda_1}{r_1} \frac{\theta}{\theta-1}\right).
$$
%Moreover, if $b\in (0,\frac{\alpha r}{\lambda_1})$, then
%$$
%\log\left(\EE\left[\exp \left(\frac{\theta \lambda_1}{r_1}n^{-b\lambda_1/r}Z_1^n\left(bt_n\right)\right)\right]\right)=o\left(n^{1-\frac{b \lambda_1}{r}}\right).
%$$
\end{lemma}
%\textcolor{red}{ZICHENG: Please recheck the statement and proof of this lemma}
\begin{proof}\\
Recall that $v_{n,\theta,b}=\frac{\theta \lambda_1}{r_1 e^{\lambda_1 b t_n}}$, and note that
\begin{align*}
& \quad \log\left(\EE\left[\exp \left(\frac{\theta \lambda_1}{r_1}n^{-b\lambda_1/r}Z_1^n\left(b t_n\right)\right)\right]\right)\\
& =\log\left(\left[\frac{d_1\left(e^{v_{n,\theta,b}}-1\right)-e^{-\lambda_1 b t_n}\left(r_1e^{v_{n,\theta,b}}-d_1\right)}{r_1\left(e^{v_{n,\theta,b}}-1\right)-e^{-\lambda_1 b t_n}\left(r_1e^{v_{n,\theta,b}}-d_1\right)}\right]^{X\left(n\right)}\right)\\
& =X\left(n\right)\log\left(\left[\frac{d_1\left(e^{v_{n,\theta,b}}-1\right)-e^{-\lambda_1 b t_n}\left(r_1e^{v_{n,\theta,b}}-d_1\right)}{r_1\left(e^{v_{n,\theta,b}}-1\right)-e^{-\lambda_1 b t_n}\left(r_1e^{v_{n,\theta,b}}-d_1\right)}\right]\right).
\end{align*}

If we define $M\left(n\right)=e^{-\lambda_1 b t_n}$, then
\begin{align*}
& \quad \log\left(\left[\frac{d_1\left(e^{v_{n,\theta,b}}-1\right)-e^{-\lambda_1 b t_n}\left(r_1e^{v_{n,\theta,b}}-d_1\right)}{r_1\left(e^{v_{n,\theta,b}}-1\right)-e^{-\lambda_1 b t_n}\left(r_1e^{v_{n,\theta,b}}-d_1\right)}\right]\right)\\
&=\log\left(\left[\frac{d_1\left(e^{\frac{\theta \lambda_1}{r_1}M\left(n\right)}-1\right)-{M\left(n\right)}\left(r_1e^{\frac{\theta \lambda_1}{r_1}M\left(n\right)}-d_1\right)}{r_1\left(e^{\frac{\theta \lambda_1}{r_1}M\left(n\right)}-1\right)-{M\left(n\right)}\left(r_1e^{\frac{\theta \lambda_1}{r_1}M\left(n\right)}-d_1\right)}\right]\right)\\
& \rightarrow \log\left(1-\frac{\lambda_1}{r_1}\frac{\theta}{\theta-1}\right) \text{ as } n\rightarrow \infty.
\end{align*} \qed
\end{proof}\\ %proof checked

%Since $b<\frac{\alpha r}{\lambda_1}$, and $X\left(n\right)=o\left(n^{1-\alpha}\right)$, the desired result follows. \end{proof}\qed\\

We now use the previous four results to show that recurrence does not occur too early. We restate Proposition \ref{prop_recurrence_1} here and provide the proof.
\\

\noindent \textbf{Proposition \ref{prop_recurrence_1}} If $b\in (0,\frac{\alpha r}{\lambda_1})$ for case (1) and (2), and $b\in (0,\frac{\beta r}{\lambda_1})$ for case (3), then there exists $C>0$ such that 
$$
\PP\left(\sup\limits_{u\in\left[0,b\right]}\left(Z_1^n\left(ut_n\right)+Z_2^n\left(ut_n\right)\right)-an>0\right)=O\left(e^{-Cn^{1-b\left(\lambda_1/r\right)}}\right).
$$

\begin{proof}\\
First we observe that
\begin{align*}
& \quad \PP\left(\sup\limits_{u\in\left[0,b\right]}\left(Z_1^n\left(ut_n\right)+Z_2^n\left(ut_n\right)\right)-an>0\right)\\
& \le \PP\left(\sup\limits_{u\in\left[0,b\right]}n^{-u\lambda_1 /r}\left(Z_1^n\left(ut_n\right)+Z_2^n\left(ut_n\right)\right)>an^{1-b\lambda_1/r}\right)\\
& =\PP\left(\sup\limits_{u\in\left[0,b\right]}\exp\left(\frac{\theta \lambda_1}{r_1}n^{-u\lambda_1 /r}\left(Z_1^n\left(ut_n\right)+Z_2^n\left(ut_n\right)\right)\right)>\exp\left(\frac{\theta \lambda_1}{r_1}an^{1-b\lambda_1/r}\right)\right),
\end{align*}
where $\theta>0$. It's not hard to show that $n^{-u\lambda_1/r}\left(Z_1^n\left(ut_n\right)+Z_2^n\left(ut_n\right)\right)$ is a submartingale in $u$. Therefore by Doob's  Inequality,
\begin{align}
\label{eq:DoobInequality1}
& \quad \PP\left(\sup\limits_{u\in\left[0,b\right]}\exp \left(\frac{\theta \lambda_1}{r_1}n^{-u\lambda_1/r}\left(Z_1^n\left(ut_n\right)+Z_2^n\left(ut_n\right)\right)\right)>\exp\left(\frac{\theta \lambda_1}{r_1}an^{1-b\lambda_1/r}\right)\right)\nonumber\\
& \le \EE\left[\exp \left(\frac{\theta \lambda_1}{r_1}n^{-b\lambda_1/r}\left(Z_1^n\left(bt_n\right)+Z_2^n\left(bt_n\right)\right)\right)\right]\exp\left(-\frac{\theta \lambda_1}{r_1}an^{1-b\lambda_1/r}\right).
\end{align}
%\textcolor{red}{ZICHENG: This doesn't seem correct as stated. I'm sure the reference you are mentioning only allows $\theta\geq 0$. Also, don't state it as a lemma, since we didn't prove it.}
Recall that we define $v_{n,\theta,b}=\frac{\lambda_1\theta}{r_1}e^{-\lambda_1bt_n}$. We state the following result (similar results can be found in\cite{Durrett} and \cite{Keller2015}): for $\theta< 1$,
\begin{equation}\label{Durrett_result}
	\EE\left[\exp\left(v_{n,\theta,b} Z^n_2\left(bt_n\right)\right)\right]=\EE\exp\left(\frac{\mu}{n^{\alpha}}\int_{0}^{bt_n}Z_0^n\left(s\right)\left(\phi_{bt_n-s}\left(v_{n,\theta,b}\right)-1\right)ds\right).
\end{equation}

Hence, from \eqref{Durrett_result} we have that
\begin{align*}
& \quad \EE\left[\exp \left(\frac{\theta \lambda_1}{r_1}n^{-b\lambda_1/r}\left(Z_1^n\left(bt_n\right)+Z_2^n\left(bt_n\right)\right)\right)\right]\\
& =\EE\exp\left(\frac{\mu}{n^{\alpha}}\int_{0}^{bt_n}Z_0^n\left(s\right)\left(\phi_{bt_n-s}\left(v_{n,\theta,b}\right)-1\right)ds\right) \EE\left[\exp \left(\frac{\theta \lambda_1}{r_1}n^{-b\lambda_1/r}Z_1^n\left(bt_n\right)\right)\right].
\end{align*}
For the first term on the right hand side, we decompose it into a mean behavior term and a fluctuation term. We then apply Proposition \ref{proposition_exp} to the fluctuation term and obtain
\begin{align*}
& \quad \EE\exp\left(\frac{\mu}{n^{\alpha}}\int_{0}^{bt_n}Z_0^n\left(s\right)\left(\phi_{bt_n-s}\left(v_{n,\theta,b}\right)-1\right)ds\right)\\
& =\exp\left(\frac{\mu}{n^{\alpha-1}}\int_{0}^{bt_n}e^{-rs}\left(\phi_{bt_n-s}\left(v_{n,\theta,b}\right)-1\right)ds\right)\\
& \quad \times \EE \exp \left(\frac{\mu}{n^{\alpha}}\int_{0}^{bt_n}\left(\phi_{bt_n-s}\left(v_{n,\theta,b}\right)-1\right)\left(Z_0^n\left(s\right)-ne^{-rs}\right)ds\right)\\
&\leq
\exp\left(\frac{\mu}{n^{\alpha-1}}\int_{0}^{bt_n}e^{-rs}\left(\phi_{bt_n-s}\left(v_{n,\theta,b}\right)-1\right)ds+k_1\left(\log n\right)^2n^{-2\alpha}\right).
\end{align*}
We can combine the previous display with \eqref{eq:DoobInequality1} to conclude that
\begin{align}
& \quad \PP\left(\sup\limits_{u\in\left[0,b\right]}\left(Z_1^n\left(ut_n\right)+Z_2^n\left(ut_n\right)\right)-an>0\right)\nonumber\\
& \le
 \EE\left[\exp \left(\frac{\theta \lambda_1}{r_1}n^{-b\lambda_1/r}\left(Z_1^n\left(bt_n\right)+Z_2^n\left(bt_n\right)\right)\right)\right]\exp\left(-\frac{\theta \lambda_1}{r_1}an^{1-b\lambda_1/r}\right)\nonumber\\
& \le
\exp\left(k_1n^{1-2\alpha}\log \left(n\right)^2\right)\exp\left(\frac{\mu}{n^{\alpha-1}}\int_{0}^{bt_n}e^{-rs}\left(\phi_{bt_n-s}\left(v_{n,\theta,b}\right)-1\right)ds\right)\EE\left[\exp \left(\frac{\theta \lambda_1}{r_1}n^{-b\lambda_1/r}Z_1^n\left(bt_n\right)\right)\right]\nonumber\\
& \quad \times \exp \left(-\frac{\theta \lambda_1}{r_1} an^{1-b\left(\lambda_1/r\right)}\right).\label{exp_num}
\end{align}
%\textcolor{red}{ZICHENG:I think we should add a lemma about the convergence of the integral.}
For the integral in \eqref{exp_num}, we have the following Lemma.
\begin{lemma} \label{lim_int}
For $\theta\in \left(-\infty,1\right)$,
	\begin{equation*}
	\lim\limits_{n\rightarrow \infty}\int_{0}^{bt_n}e^{-rs}\left(\phi_{bt_n-s}\left(v_{n,\theta,b}\right)-1\right)ds=\frac{\lambda_1}{r_1}\int_{0}^{\infty}\frac{\theta e^{-rs}}{e^{\lambda_1 s}-\theta}ds.
	\end{equation*}
\end{lemma}

\begin{proof}\\
From Lemma \ref{phi_property}, it follows that for any $\theta\in \left(0,1\right)$ we have that
\begin{equation*}
\sup\left\{|\phi_{bt_n-s}\left(v_{n,\theta,b}\right)-1|; s\in (0,bt_n), n\geq 1\right\}<\infty.
\end{equation*}
Hence,
\begin{equation*}
e^{-rs}|\phi_{bt_n-s}\left(v_{n,\theta,b}\right)-1| I_{\left(0, bt_n\right)}\left(s\right)\le Me^{-rs}I_{\left(0,\infty\right)}\left(s\right),
\end{equation*}
where $M$ is a sufficiently large positive constant. Since $\int_{0}^{\infty} Me^{-rs}ds<\infty$, and some algebra yields that

\begin{equation*}
\lim\limits_{n\rightarrow \infty}\left(\phi_{bt_n-s}\left(v_{n,\theta,b}\right)-1\right)I_{\left(0, bt_n\right)}=\frac{\lambda_1 \theta}{r_1 \left(e^{\lambda_1 s}-\theta\right)}I_{\left(0,\infty\right)}.
\end{equation*}
We may apply the Dominated Convergence Theorem to conclude that 
\begin{equation*}
\lim\limits_{n\rightarrow \infty}\int_{0}^{bt_n}e^{-rs}\left(\phi_{bt_n-s}\left(v_{n,\theta,b}\right)-1\right)ds=\frac{\lambda_1}{r_1}\int_{0}^{\infty}\frac{\theta e^{-rs}}{e^{\lambda_1 s}-\theta}ds.
\end{equation*} \qed %proof checked
\end{proof} 

Using the fact that for case (1) and (2), $b\left(\lambda_1/r\right)<\alpha$, and for case (3), $b\left(\lambda_1/r\right)<\beta<\alpha$, Lemma \ref{z1_moment_generate}, and Lemma \ref{lim_int}, the term inside the exponential of \eqref{exp_num} can be written as
\begin{equation*}
-n^{1-b\left(\lambda_1/r\right)}\left(\frac{a \theta  \lambda_1}{r_1}-o\left(1\right)\right).
\end{equation*}
We then have that 
\begin{equation*}
\PP\left(\sup\limits_{u\in\left[0,b\right]}\left(Z_1^n\left(ut_n\right)+Z_2^n\left(ut_n\right)\right)-an>0\right)\leq \exp\left[-n^{1-b\left(\lambda_1/r\right)}\left(\frac{a\theta\lambda_1}{r_1}-o(1)\right)\right],
\end{equation*}
for $\theta\in (0,1)$.\qed\end{proof}\\ %proof checked

We next state some auxiliary results for analyzing the probability of recurrence in the interval $[bt_n, \zeta_n(a)-y]$. We restate Proposition \ref{proposition_recurrence_1} here which develops an upper bound on the probability of early recurrence.
\\

\noindent \textbf{Proposition \ref{proposition_recurrence_1}} For $\theta\in (0,1)$, 
%define $\theta_n=\frac{\lambda_1 \theta n^{1-\alpha}}{r_1}$, and $v_{n,\theta,u_n\left(y\right)}=\theta_n n^{\alpha-1-\lambda_1 u_n\left(y\right)/r}$, then
	\begin{align*}
	& \quad \PP\left(\sup\limits_{u\in\left[b,u_n\left(y\right)\right]}n^{\alpha-1-\lambda_1 u/r}\left(Z_1^n\left(ut_n\right)+Z_2^n\left(ut_n\right)\right)> an^{\alpha-\lambda_1 u_n\left(y\right)/r}\right)\\
	& \le \EE \exp\left(\frac{\mu}{n^{\alpha}}\int_{0}^{u_n\left(y\right)t_n}Z_0^n\left(s\right)\left(\phi_{u_n\left(y\right)t_n-s}\left(v_{n,\theta,u_n\left(y\right)}\right)-1\right)ds\right)\EE\exp\left(v_{n,\theta,u_n\left(y\right)}Z_1^n\left(u_n\left(y\right)t_n\right)\right)\\
	& \quad \times \exp\left(-anv_{n,\theta,u_n\left(y\right)}\right).
	\end{align*}
	%\textcolor{red}{ZICHENG: I added the restriction $\theta<1$, is that correct?}
%Recall the function 
%\begin{equation*}
%c_2\left(y;n\right)=c_1\left(y;n\right)+n^{\alpha-1-\lambda_1u_n\left(y\right)/r}\left(z_1^n\left(u_n\left(y\right)t_n\right)+z_2^n\left(u_n\left(y\right)t_n\right)\right),
%\end{equation*}
%defined in the proof of Theorem \ref{recurrence_LD}.

\begin{proof}\\
First we observe that
\begin{align*}
& \quad \PP\left(\sup\limits_{u\in\left[b,u_n\left(y\right)\right]}n^{\alpha-1-\lambda_1 u/r}\left(Z_1^n\left(ut_n\right)+Z_2^n\left(ut_n\right)\right) > an^{\alpha-\lambda_1 u_n\left(y\right)/r}\right)\\
& =\PP\left(\sup\limits_{u\in\left[b,u_n\left(y\right)\right]} \exp\left(\frac{\lambda_1 \theta}{r_1}n^{-\lambda_1 u/r}\left(Z_1^n\left(ut_n\right)+Z_2^n\left(ut_n\right)\right)\right) >\exp\left(anv_{n,\theta,u_n\left(y\right)}\right)\right).
\end{align*}
Observe that $\exp\left(\frac{\lambda_1 \theta}{r_1}n^{-\lambda_1 u/r}\left(Z_1^n\left(ut_n\right)+Z_2^n\left(ut_n\right)\right)\right)$ is a submartingale in $u$. Therefore, by Doob's martingale inequality,
\begin{align*}
& \quad \PP\left(\sup\limits_{u\in\left[b,u_n\left(y\right)\right]} \exp\left(\frac{\lambda_1 \theta}{r_1}n^{-\lambda_1 u/r}\left(Z_1^n\left(ut_n\right)+Z_2^n\left(ut_n\right)\right)\right) >\exp\left(anv_{n,\theta,u_n\left(y\right)}\right)\right)\\
&\le \EE\left[\exp\left(v_{n,\theta,u_n\left(y\right)}\left(Z_1^n\left(u_n\left(y\right)t_n\right)+Z_2^n\left(u_n\left(y\right)t_n\right)\right)\right)\right]\exp\left(-anv_{n,\theta,u_n\left(y\right)}\right).
\end{align*}
Applying \eqref{Durrett_result}, and the definition of $v_{n,\theta,u_n\left(y\right)}$ gives us the result. \qed\end{proof} %proof checked
\\

We next identify the pointwise limit as $n\to\infty$ of the sequence of $an^{\alpha-\lambda_1 u_n\left(y\right)/r}$.
%of functions $c_2(y;n)$ as defined in \eqref{c_2}.%\textcolor{red}{ZICHENG: We should make the $n$ dependence in $c_2(y)$ explicit, i.e., $c_2(y;n)$. Also we should give $c_2(y)$ an equation number so we can refer back to it.}
\begin{lemma}\label{c_2_recurrence_1}
For $y>0$,
\begin{align*}
& \text{in case (1), }\lim\limits_{n\rightarrow \infty}an^{\alpha-\lambda_1 u_n\left(y\right)/r}=e^{\lambda_1 y}\frac{\mu}{\lambda_1+r};\\
& \text{in case (2), }\lim\limits_{n\rightarrow \infty}an^{\alpha-\lambda_1 u_n\left(y\right)/r}=e^{\lambda_1 y}\left(1+\frac{\mu}{\lambda_1+r}\right);\\
& \text{in case (3), }\lim\limits_{n\rightarrow \infty}n^{\beta-\alpha}an^{\alpha-\lambda_1 u_n\left(y\right)/r}=e^{\lambda_1 y}.
\end{align*}
\end{lemma}
\begin{proof}\\
Since the proofs are similar, we only provide the proof for case (1) here. 
%Write out the definition of $c_2$ to see that
%\begin{align*}
%c_2\left(y;n\right)&=c_1\left(y;n\right)+n^{\alpha-1-\lambda_1u_n\left(y\right)/r}\left(z_1^n\left(u_n\left(y\right)t_n\right)+z_2^n\left(u_n\left(y\right)t_n\right)\right)\\
%&=n^{\alpha-1-\lambda_1u_n\left(y\right)/r}an\\
%&=an^{\alpha-\lambda_1u_n\left(y\right)/r}\\
%&=an^{\alpha}e^{\lambda_1 y}e^{-\lambda_1 \zeta_n\left(a\right)}.
%\end{align*} 
Recall that we have defined $\bar{u}_n^{upper}$ and $\bar{u}_n^{lower}$ to be the unique solutions to Equation \ref{upper_bound_equation} and \ref{lower_bound_equation} respectively. We then have that
\begin{equation*}
an^{\alpha}e^{\lambda_1 y}e^{-\lambda_1 \bar{u}_n^{upper}}\le an^{\alpha-\lambda_1 u_n\left(y\right)/r} \le an^{\alpha}e^{\lambda_1 y}e^{-\lambda_1 \bar{u}_n^{lower}}.
\end{equation*}

If we take the limits of the lower bound and the upper bound, we obtain that
\begin{align*}
\lim\limits_{n\rightarrow \infty}an^{\alpha}e^{\lambda_1 y}e^{-\lambda_1 \bar{u}_n^{upper}}&=\lim\limits_{n\rightarrow \infty}an^{\alpha}e^{\lambda_1 y} \frac{X\left(n\right)+\frac{\mu}{\lambda_1+r}n^{1-\alpha}}{an+\frac{\mu}{\lambda_1+r}n^{1-\alpha}}\\
&=e^{\lambda_1 y}\frac{\mu}{\lambda_1+r},
\end{align*}
and
\begin{align*}
\lim\limits_{n\rightarrow \infty}an^{\alpha}e^{\lambda_1 y}e^{-\lambda_1 \bar{u}_n^{lower}}&=\lim\limits_{n\rightarrow \infty}an^{\alpha}e^{\lambda_1 y}\frac{X\left(n\right)+\frac{\mu}{\lambda_1+r}n^{1-\alpha}}{an}\\
&=e^{\lambda_1 y}\frac{\mu}{\lambda_1+r}.
\end{align*}
The desired result then follows. \qed\end{proof} %proof checked

\subsubsection{Results for Lower Bound Proof}\label{proof_lower}
%\textcolor{red}{ZICHENG: This Lemma should just be combined with Proposition \ref{proposition_exp}.}
%\begin{lemma} \label{lim_1}
	%For $\theta<1$,
	%\begin{equation*}
	%\lim\limits_{n\rightarrow \infty}n^{\alpha-1}\log\EE \exp\left(\frac{\mu}{n^{\alpha}}\int_{0}^{u_n\left(y\right)t_n}\left(\phi_{u_n\left(y\right)t_n-s}\left(v_{n,\theta,u_n\left(y\right)}\right)-1\right)\left(Z_0^n\left(s\right)-ne^{-rs}\right)ds\right)=0.
	%\end{equation*}
%end{lemma}

We restate Proposition \ref{GE_cond} here and provide the proof. Recall that the conditions of the Gartner-Ellis theorem are
\begin{enumerate}
\item There exists a function $\Lambda\left(\theta\right)\in \left[-\infty, \infty\right]$ such that $\lim\limits_{n\rightarrow \infty}n^{\alpha-1}\Lambda_n\left(\theta n^{1-\alpha}\right)=\Lambda\left(\theta\right)$ for all $\theta \in \RR$,
\item $0\in int\left(D_{\Lambda}\right)$ where $int\left(D_{\Lambda}\right)=\{\theta\in \RR: \Lambda\left(\theta\right)<\infty \}$,
\item $\Lambda$ is lower semi-continuous on $\RR$,
\item $\Lambda$ is differentiable on $int\left(D_{\Lambda}\right)$,
\item	$\Lambda$ is steep at $\partial D_{\Lambda}$.
\end{enumerate}

\noindent \textbf{Proposition \ref{GE_cond}} The sequence of random variables $\{Z_n\}_{n\ge 1}$ as defined in (\ref{rv_recurrence_1}) satisfy the conditions of the Gartner-Ellis Theorem with 
$$
\Lambda(\theta)=\begin{cases}\frac{\lambda_1\mu\theta}{r_1}\int_0^\infty\frac{e^{-rs}}{e^{\lambda_1s}-\theta}ds,&\enskip \theta<1\\
\infty,&\enskip \theta\geq 1.
\end{cases}
$$

\begin{proof}\\
\textit{Proof of Condition 1}. Recall that we define
%If we re-define \textcolor{red}{ZICHENG: You have to be more precise with terminology. You can't redefine things in the same paper. If someone is perusing the paper, how do they know what $\hat\theta_n$ is?}
\begin{equation*}
v_{n,\theta,u_n\left(y\right)}=\frac{\lambda_1\theta n^{-\lambda_1u_n\left(y\right)/r}}{r_1}
\end{equation*}
for all $\theta\in \RR$. We have that
\begin{align*}
& \quad n^{\alpha-1}\Lambda_n\left(\theta n^{1-\alpha}\right)\\
&=n^{\alpha-1}\left(\log \EE \exp\left(v_{n,\theta,u_n\left(y\right)}\left(Z_1^n\left(u_n\left(y\right)t_n\right)+Z_2^n\left(u_n\left(y\right)t_n\right)\right)\right)+\theta n^{1-\alpha}c\left(y;n\right)\right)\\
&=n^{\alpha-1}\log \EE \exp\left(v_{n,\theta,u_n\left(y\right)}\left(Z_1^n\left(u_n\left(y\right)t_n\right)+Z_2^n\left(u_n\left(y\right)t_n\right)\right)\right)+\theta c\left(y;n\right).
\end{align*}

Using the expression for $c\left(y;n\right)$ given in \ref{c_recurrence_1}, and Lemma \ref{c_2_recurrence_1}, we find that the second term goes to zero as $n\rightarrow \infty$. Hence, we will focus on the limit of the first term.
\begin{align}
& \quad n^{\alpha-1}\log \EE \exp\left(v_{n,\theta,u_n\left(y\right)}\left(Z_1^n\left(u_n\left(y\right)t_n\right)+Z_2^n\left(u_n\left(y\right)t_n\right)\right)\right)\nonumber\\
&=n^{\alpha-1}\log \EE \exp\left(v_{n,\theta,u_n\left(y\right)}Z_1^n\left(u_n\left(y\right)t_n\right)\right)+n^{\alpha-1}\log \EE \exp\left(v_{n,\theta,u_n\left(y\right)} Z_2^n\left(u_n\left(y\right)t_n\right)\right). \label{term_1}
\end{align}

From Remark \ref{finite_mgf} and Lemma \ref{z1_moment_generate} we know that for $\theta< 1$,

\begin{align*}
& \quad n^{\alpha-1}\log \EE \exp\left(v_{n,\theta,u_n\left(y\right)}Z_1^n\left(u_n\left(y\right)t_n\right)\right)\\
&= n^{\alpha-1}X\left(n\right)\log\left(\frac{d_1\left(e^{v_{n,\theta,u_n\left(y\right)}}-1\right)-e^{-\lambda_1u_n\left(y\right)t_n}\left(r_1e^{v_{n,\theta,u_n\left(y\right)}}-d_1\right)}{r_1\left(e^{v_{n,\theta,u_n\left(y\right)}}-1\right)-e^{-\lambda_1u_n\left(y\right)t_n}\left(r_1e^{v_{n,\theta,u_n\left(y\right)}}-d_1\right)}\right)\\
& \rightarrow 0 \text{ as } n\rightarrow \infty.
\end{align*}

Hence, for $\theta<1$, we only need to consider the second term of (\ref{term_1}). Recall from \eqref{Durrett_result} that
\begin{equation*}
\EE\left[\exp\left(v_{n,\theta,u_n\left(y\right)} Z^n_2\left(u_n\left(y\right)t_n\right)\right)\right]=\EE\exp\left(\frac{\mu}{n^{\alpha}}\int_{0}^{u_n\left(y\right)t_n}Z_0^n\left(s\right)\left(\phi_{u_n\left(y\right)t_n-s}\left(v_{n,\theta,u_n\left(y\right)}\right)-1\right)ds\right).
\end{equation*}
Therefore the second term satisfies the following relation
\begin{align*}
& \quad n^{\alpha-1}\log \EE\left[\exp\left(v_{n,\theta,u_n\left(y\right)} Z^n_2\left(u_n\left(y\right)t_n\right)\right)\right]\\
&= n^{\alpha-1}\log \exp\left(\frac{\mu}{n^{\alpha-1}}\int_{0}^{u_n\left(y\right)t_n}e^{-rs}\left(\phi_{u_n\left(y\right)t_n-s}\left(v_{n,\theta,u_n\left(y\right)}\right)-1\right)ds\right)\\
& \quad + n^{\alpha-1}\log\EE \exp\left(\frac{\mu}{n^{\alpha}}\int_{0}^{u_n\left(y\right)t_n}\left(\phi_{u_n\left(y\right)t_n-s}\left(v_{n,\theta,u_n\left(y\right)}\right)-1\right)\left(Z_0^n\left(s\right)-ne^{-rs}\right)ds\right).
\end{align*}

When $\theta<1$, we may ignore the second term  by Proposition \ref{proposition_exp} and focus on the first term. We see that
\begin{align*}
& \quad n^{\alpha-1}\log \exp\left(\frac{\mu}{n^{\alpha-1}}\int_{0}^{u_n\left(y\right)t_n}e^{-rs}\left(\phi_{u_n\left(y\right)t_n-s}\left(v_{n,\theta,u_n\left(y\right)}\right)-1\right)ds\right)\\
&= \mu \int_{0}^{u_n\left(y\right)t_n}e^{-rs}\left(\phi_{u_n\left(y\right)t_n-s}\left(v_{n,\theta,u_n\left(y\right)}\right)-1\right)ds.
\end{align*}

By Lemma \ref{lim_int}, we know that 

%For the case $\theta\in \left(0,1\right)$, by dominated convergence theorem \textcolor{red}{(YOU NEED TO CItE A LEMMA saying this is bounded}), it's easy to show that
\begin{equation*}
\mu \int_{0}^{u_n\left(y\right)t_n}e^{-rs}\left(\phi_{u_n\left(y\right)t_n-s}\left(v_{n,\theta,u_n\left(y\right)}\right)-1\right)ds\rightarrow \frac{\lambda_1 \mu \theta}{r_1}\int_{0}^{\infty}\frac{ e^{-rs}}{e^{\lambda_1 s}-\theta}ds.
\end{equation*}

%For the case $\theta\le 0$, $|\phi_{u_n\left(y\right)t_n-s}\left(v_{n,\theta,u_n\left(y\right)}\right)-1|\le 1$, and we may repeat the dominated convergence argument to show that the same limit is obtained. 

We now consider the case that $\theta\ge 1$. First we note that for $\theta\ge 1$ and a sequence of numbers $\{\tilde{\theta}_1, \tilde{\theta}_2,...\}$ such that $\tilde{\theta}_i<1$, and $\lim\limits_{i\to \infty}\tilde{\theta}_i=1$, 
%\textcolor{red}{ZICHENG: What is $\tilde\theta$?} $\tilde{\theta}<1$,
\begin{equation*}
\frac{\lambda_1 \mu \tilde{\theta}_i}{r_1}\int_{0}^{\infty}\frac{ e^{-rs}}{e^{\lambda_1 s}-\tilde{\theta}_i}ds=\liminf\limits_{n\rightarrow \infty}n^{\alpha-1}\Lambda_n\left(\tilde{\theta}_in^{1-\alpha}\right)\le \liminf\limits_{n\rightarrow \infty}n^{\alpha-1}\Lambda_n\left(\theta n^{1-\alpha}\right).
\end{equation*}

As $i\rightarrow \infty$, $\frac{\lambda_1 \mu \tilde{\theta}_i}{r_1}\int_{0}^{\infty}\frac{ e^{-rs}}{e^{\lambda_1 s}-\tilde{\theta}_i}ds \rightarrow \infty$. Therefore, $\liminf\limits_{n\rightarrow \infty}n^{\alpha-1}\Lambda_n\left(\theta n^{1-\alpha}\right)=\infty$ and consequently $\lim\limits_{n\rightarrow \infty}n^{\alpha-1}\Lambda_n\left(\theta n^{1-\alpha}\right)=\infty$. Hence if we define 
%\begin{equation*}
%\Lambda\left(\theta\right)=\frac{\lambda_1 \mu \theta}{r_1}\int_{0}^{\infty}\frac{ e^{-rs}}{e^{\lambda_1 s}-\theta}ds,
%\end{equation*}

\begin{align*}
\Lambda\left(\theta\right)=\left\{ \begin{array}{cc} 
\frac{\lambda_1 \mu \theta}{r_1}\int_{0}^{\infty}\frac{ e^{-rs}}{e^{\lambda_1 s}-\theta}ds, & \hspace{5mm} \theta<1 \\
\infty & \hspace{5mm} \theta\ge 1, \\
\end{array} \right.
\end{align*}

then $n^{\alpha-1}\Lambda_n\left(\theta n^{1-\alpha}\right)\rightarrow \Lambda \left(\theta\right)$ for all $\theta\in \RR$. \\
%\textcolor{red}{ZICHENG: What if $\tilde{\theta}>1$? Also what is definition of $\Lambda(\theta)$ for $\theta\geq 1$?}\\

\textit{Proof of Condition 2} For $\theta <1$, 
\begin{equation*}
\int_{0}^{\infty}|\frac{\lambda_1 \mu \theta}{r_1}\frac{ e^{-rs}}{e^{\lambda_1 s}-\theta}|ds\le \frac{\lambda_1 \mu |\theta|}{r_1\left(1-\theta\right)}\int_{0}^{\infty}e^{-rs}ds<\infty,
\end{equation*}
and we may conclude that $\Lambda\left(\theta\right)<\infty$ for $\theta <1$. We have proved in the proof of Condition (1) that $\Lambda\left(\theta\right)=\infty$ for $\theta \ge 1$. Therefore $D_{\Lambda}=\left(-\infty, 1\right)=\text{int}\left(D_{\Lambda}\right)$, which includes $0$.\\

\textit{Proof of Condition 3}. It's easy to show that $\Lambda\left(\theta\right)$ is differentiable for $\theta <1$. It is therefore continuous and hence lower semi-continuous on $\left(-\infty, 1\right)$. We now consider the case $\theta_i\rightarrow \theta^*$ where $\theta^*\ge 1$. We have shown in the proof of Condition (1) that $\lim\limits_{\theta \uparrow 1}\Lambda\left(\theta\right)=\infty$ and that $\Lambda\left(\theta\right)=\infty$ for $\theta \ge 1$. Therefore for any $\theta^*\ge 1$, $\lim\limits_{i\rightarrow \infty}\Lambda\left(\theta_i\right)=\Lambda\left(\theta^*\right)=\infty$. We may conclude that $\Lambda\left(\theta\right)$ is lower semi-continuous on the entire real line. 
\\

\textit{Proof of Condition 4}. We have shown in the proof of Condition (2) that $\textbf{int}\left(D_{\Lambda}\right)=\left(-\infty, 1\right)$. It's easy to show that $\Lambda\left(\theta\right)$ is differentiable for $\theta <1$. 
\\

\textit{Proof of Condition 5}. We can show that for $\theta <1$.
\begin{equation*}
|\nabla \Lambda\left(\theta\right)|=\frac{\lambda_1 \mu}{r_1}\int_{0}^{\infty}\frac{e^{\left(\lambda_1-r\right)s}}{\left(e^{\lambda_1 s}-\theta\right)^2}ds.
\end{equation*}

By inspection, $\lim\limits_{\theta \rightarrow \partial D_{\Lambda}: \theta \in D_{\Lambda}}|\nabla \Lambda\left(\theta\right)|=\lim\limits_{\theta \uparrow 1}\Lambda' \left(\theta\right)=\infty$. \qed
\end{proof}

We next state a result about the second derivative of $\Lambda$.
\begin{lemma}\label{lambda_second_derivative}
	The function $\Lambda\left(\theta\right)=\frac{\lambda_1 \mu \theta}{r_1}\int_{0}^{\infty}\frac{e^{-rs}}{e^{\lambda_1 s}-\theta}ds$ is twice differentiable for $\theta \in \left(0,1\right)$, and
	\begin{equation*}
	\Lambda''\left(\theta\right)=\frac{2\lambda_1 \mu}{r_1}\int_{0}^{\infty}\frac{e^{\left(\lambda_1-r\right)s}}{\left(e^{\lambda_1 s}-\theta\right)^3}ds.
	\end{equation*}
\end{lemma}
%\textcolor{red}{Zicheng: You can just state something about the integrand being integrable, and an application of Dominated convergence.}
\begin{proof}\\
It suffices to check that the derivative of the integrand exists and is dominated by an integrable function, which is straightforward. See Proposition 9.2.1 of \cite{Rosenthal}. \qed
\end{proof}

Next we restate Proposition \ref{proposition_recurrence_3} and provide the proof.
\\

\noindent \textbf{Proposition \ref{proposition_recurrence_3}} For any $x\in \left(\frac{\mu\lambda_1e^{y\lambda_1}}{r_1\left(\lambda_1+r\right)},\infty \right)$, there exists $\theta^*(x)\in \left(0,1\right)$, such that $x\theta^*(x)-\Lambda\left(\theta^*(x)\right)=\sup\limits_{\theta\in \RR}\left[\theta x-\Lambda\left(\theta\right)\right]$.

Furthermore, 
$$
\theta^*(x)=\Lambda^{\prime-1}(x),
$$
and in particular $\theta^*$ is a continuous function.

\begin{proof}\\
Note that for any $x\in \left(\frac{\lambda_1}{r_1}\frac{\mu e^{\lambda_1 y}}{\lambda_1+r},\infty\right)$ and $\theta\ge 1$, $\theta x-\Lambda\left(\theta\right)=-\infty$ since $\Lambda \left(\theta\right)=\infty$ for $\theta \ge 1$. Therefore for $x\in \left(\frac{\lambda_1}{r_1}\frac{\mu e^{\lambda_1 y}}{\lambda_1+r},\infty\right)$,

\begin{equation*}
\sup\limits_{\theta\in \RR}\left[\theta x-\frac{\lambda_1 \mu \theta}{r_1}\int_{0}^{\infty}\frac{e^{-rs}}{e^{\lambda_1 s}-\theta}ds\right]=\sup\limits_{\theta\in \left(-\infty, 1\right)}\left[\theta x-\frac{\lambda_1 \mu \theta}{r_1}\int_{0}^{\infty}\frac{e^{-rs}}{e^{\lambda_1 s}-\theta}ds\right].
\end{equation*}

From Lemma \ref{lambda_second_derivative} we see that $\Lambda''\left(\theta\right)>0$ for $\theta <1$. Therefore, $\Lambda\left(\theta\right)$ is a convex function on $\left(-\infty, 1\right)$ and if we fix $x\in \left(\frac{\lambda_1}{r_1}\frac{\mu e^{\lambda_1 y}}{\lambda_1+r},\infty\right)$, $\theta x-\Lambda\left(\theta\right)$ is a concave function in $\theta$ on $\left(-\infty, 1\right)$. A sufficient condition for $\theta \in \left(-\infty, 1\right)$ to be a global maximizer for the concave function in the square brackets is $\Lambda'\left(\theta\right)=x$ or equivalent $\frac{\lambda_1 \mu}{r_1}\int_{0}^{\infty}\frac{e^{\left(\lambda_1-r\right)s}}{\left(e^{\lambda_1 s}-\theta\right)^2}ds=x$. Now if we evaluate $\Lambda'\left(\theta\right)$ at zero, we find that $\Lambda'\left(0\right)=\frac{\lambda_1 \mu}{r_1\left(\lambda_1+r\right)}$. We also know that $\lim\limits_{\theta \uparrow 1}\Lambda'\left(\theta\right)=\infty$. Since $\Lambda'\left(\theta\right)$ is differentiable on $\left(-\infty, 1\right)$, it is also continuous on $\left(-\infty, 1\right)$ and we may conclude that it obtains all values greater than $\frac{\lambda_1 \mu}{r_1\left(\lambda_1+r\right)}$ on $\left(0,1\right)$. Therefore, for all $x\in \left(\frac{\lambda_1}{r_1}\frac{\mu e^{\lambda_1 y}}{\lambda_1+r},\infty\right)$ there exists $\theta\in \left(0,1\right)$ such that $\Lambda'\left(\theta\right)=x$. This leads to the conclusion that the supremum is always attained by $\theta\in \left(0,1\right)$. In addition we identify $\theta^*(x)$ as the inverse of $\Lambda^\prime$.

To see that $\theta^*$ is continuous, observe from Lemma \ref{lambda_second_derivative}, that $\Lambda''\left(\theta\right)>0$ for all $\theta\in \left(0,1\right)$, and it follows that $\Lambda'$ is a strictly increasing and continuous function on $\left(0,1\right)$. Therefore, $\Lambda'$ has a strictly increasing, continuous inverse. \qed\\

%Recall that to satisfy the conditions of Sion's minimax theorem (ADD CItATION) we need a function to satisfy (ADD CONDiTIONS). The next result shows that we can apply this result.
%\begin{lemma}\label{minimax_condition_recurrence_1}
	%The function $f\left(x,\theta\right)=\theta x-\Lambda \left(\theta\right)$ satisfies the conditions of Sion's Minimax Theorem.
%\end{lemma}

%Proof. If we define $f\left(x,\theta\right)=\theta x-\Lambda \left(\theta\right)$, then $f\left(\cdot, \theta\right)$ is a continuous, linear function in $x$ on $\left[\frac{\mu\lambda_1e^{\left(\lambda_1+r\right)y}}{r_1\left(\lambda_1+r\right)}+\epsilon,z \right]$, and therefore a quasiconvex and lower-semicontinuous function over the compact set. Lemma \ref{lambda_second_derivative} shows that $f\left(x,\cdot\right)$ is a concave and differentiable function in $\theta$ on $\left(-\infty, 1\right)$, and therefore $f\left(x,\cdot\right)$ is a quasiconcave and upper-semicontinuous function over the convex set $\left(0,1\right)$. \qed
%%%%%%%%%%%%%%%%%%%%553################################
%%%%%%%%%%%%%%%%%%%%553################################
%%%%%%%%%%%%%%%%%%%%553################################
%%%%%%%%%%%%%%%%%%%%553################################
%%%%%%%%%%%%%%%%%%%%553################################
%%%%%%%%%%%%%%%%%%%%553################################
%%%%%%%%%%%%%%%%%%%%%%%%%%%%%%%%%%%%%%%%%%%%%%%%%%%%%%%%%%%%%%%%%%%%%%%%%%%%%%%%%%%%%%%%%%%%%%%%%%%%%%%%%

%%%%%%%%%%%%%%%%%%%%%%%%%%%%%%%%%%%%%%%%%%%%%%%%%%%%%%%%%%%%%%%%%%%%%%%%%%%%%%%%%%%%%%%%%%%%%%%%%%%%%%%%%%%%%%%%%%%

\subsection{Lemmas and Propositions for Proof of Theorem \ref{crossover_LD}}
Define the quantities
\begin{align}
& A_1\left(u,n\right)=Z_1^n\left(ut_n\right)-z_1^n\left(ut_n\right) \label{a1_crossover_1}\\
& A_2\left(u,n\right)=Z_2^n\left(ut_n\right)-z_2^n\left(ut_n\right) \label{a2_crossover_1}\\
& A_3\left(u,n\right)=z_0^n\left(ut_n\right)-Z_0^n\left(ut_n\right) \label{a3_crossover_1}\\
& A_4\left(u,n\right)=z_1^n\left(ut_n\right)+z_2^n\left(ut_n\right)-z_0^n\left(ut_n\right). \label{a4_crossover_1}
\end{align}
Then for $\delta\in \left(0,1\right)$, we have the following results (Proposition \ref{proposition_crossover_1}, \ref{proposition_crossover_2}, and \ref{proposition_crossover_3}).
\begin{proposition}\label{proposition_crossover_1}
	There exists $C>0$ such that
	\begin{equation*}
	    \PP\left(\sup\limits_{u\in\left[0,a\right]}\left(A_1\left(u,n\right)+A_2\left(u,n\right)+\delta A_4\left(u,n\right)\right)>0\right)=O\left(e^{-Cn^{1-a\left(1+\lambda_1/r\right)}}\right).
	\end{equation*}
\end{proposition}
\begin{proof}\\
Observe that
\begin{align*}
& \quad \PP\left(\sup\limits_{u\in\left[0,a\right]}\left(A_1\left(u,n\right)+A_2\left(u,n\right)+\delta A_4\left(u,n\right)>0\right)\right)\\
& \le \PP\left(\sup\limits_{u\in\left[0,a\right]}n^{u-1}\left(A_1\left(u,n\right)+A_2\left(u,n\right)\right)+\sup\limits_{u\in\left[0,a\right]}\delta n^{u-1}A_4\left(u,n\right)>0\right)\\
& =\PP\left(\sup\limits_{u\in\left[0,a\right]}n^{u-1}\left(A_1\left(u,n\right)+A_2\left(u,n\right)\right)>\inf\limits_{u\in\left[0,a\right]}-\delta n^{u-1}A_4\left(u,n\right)\right).
\end{align*}
We may compute that
\begin{align*}
& \quad -n^{u-1}A_4\left(u,n\right)\\
&=n^{u-1}\left(z_0^n\left(ut_n\right)-z_1^n\left(ut_n\right)-z_2^n\left(ut_n\right)\right)\\
&=n^{u-1}\left(n^{1-u}-X\left(n\right)n^{\frac{\lambda_1}{r}u}-\frac{\mu}{\lambda_1+r}n^{1-\alpha}e^{\lambda_1 u t_n}\left(1-e^{\left(\lambda_0-\lambda_1\right)u t_n}\right)\right)\\
&=1-X\left(n\right)n^{\left(1+\frac{\lambda_1}{r}\right)u-1}+\frac{\mu n^{-\alpha}}{\lambda_1+r}-\frac{\mu n^{u\left(1+\frac{\lambda_1}{r}\right)-\alpha}}{\lambda_1+r}
\end{align*}
is a monotone decreasing function in $u$ and consequently obtain that
\begin{align*}
& \quad \PP\left(\sup\limits_{u\in\left[0,a\right]}n^{u-1}\left(A_1\left(u,n\right)+A_2\left(u,n\right)\right)>\inf\limits_{u\in\left[0,a\right]}-\delta n^{u-1}A_4\left(u,n\right)\right)\\
& = \PP\left(\sup\limits_{u\in\left[0,a\right]}n^{u-1}\left(A_1\left(u,n\right)+A_2\left(u,n\right)\right)>\delta n^{a-1}\left(z_0^n\left(at_n\right)-z_1^n\left(at_n\right)-z_2^n\left(at_n\right)\right)\right).
\end{align*}
Now if we define $C\left(a,n\right)=n^{a-1}\left(z_0^n\left(at_n\right)-z_1^n\left(at_n\right)-z_2^n\left(at_n\right)\right)$ then
\begin{align*}
& \quad \PP\left(\sup\limits_{u\in\left[0,a\right]}n^{u-1}\left(A_1\left(u,n\right)+A_2\left(u,n\right)\right)>\delta C\left(a,n\right)\right)\\
& \le \PP\left(\sup\limits_{u\in\left[0,a\right]}n^{1-u\left(1+\lambda_1/r\right)}n^{u-1}\left(A_1\left(u,n\right)+A_2\left(u,n\right)\right)>\delta n^{1-a\left(1+\lambda_1/r\right)} C\left(a,n\right)\right)\\
& \le \PP\Bigg(\sup\limits_{u\in\left[0,a\right]}n^{-u\lambda_1/r}\left(Z_1^n\left(ut_n\right)+Z_2^n\left(ut_n\right)\right)+\sup\limits_{u\in\left[0,a\right]}\big(n^{-a\lambda_1/r}z_1^n\left(at_n\right)+n^{-a\lambda_1/r}z_2^n\left(at_n\right) \\
& \quad-n^{-u\lambda_1/r}z_1^n\left(ut_n\right)-n^{-u\lambda_1/r}z_2^n\left(ut_n\right)\big)> \delta n^{1-a\left(1+\lambda_1/r\right)} C\left(a,n\right)+  n^{-a\lambda_1/r}z_1^n\left(at_n\right)+n^{-a\lambda_1/r}z_2^n\left(at_n\right)\Bigg).
\end{align*}
We note that $n^{-u\lambda_1/r}z_1^n\left(ut_n\right)+n^{-u\lambda_1/r}z_2^n\left(ut_n\right)$ is monotone increasing in $u$. Therefore,
\begin{align*}
& \quad \sup\limits_{u\in\left[0,a\right]}\left(n^{-a\lambda_1/r}z_1^n\left(at_n\right)+n^{-a\lambda_1/r}z_2^n\left(at_n\right)-n^{-u\lambda_1/r}z_1^n\left(ut_n\right)-n^{-u\lambda_1/r}z_2^n\left(ut_n\right)\right)\\
& =n^{-a\lambda_1/r}z_1^n\left(at_n\right)+n^{-a\lambda_1/r}z_2^n\left(at_n\right)-X\left(n\right).
\end{align*}
We now arrive at 
\begin{align*}
& \quad \PP\left(\sup\limits_{u\in\left[0,a\right]}n^{u-1}\left(A_1\left(u,n\right)+A_2\left(u,n\right)\right)>\delta C\left(a,n\right)\right)\\
& \le \PP\left(\sup\limits_{u\in\left[0,a\right]}n^{-u\lambda_1/r}\left(Z_1^n\left(ut_n\right)+Z_2^n\left(ut_n\right)\right)>\delta n^{1-a\left(1+\lambda_1/r\right)} C\left(a,n\right)\right)\\
& =\PP\left(\sup\limits_{u\in\left[0,a\right]}\exp \left(\frac{\theta \lambda_1}{r_1}n^{-u\lambda_1/r}\left(Z_1^n\left(ut_n\right)+Z_2^n\left(ut_n\right)\right)\right)>\exp \left(\frac{\theta \lambda_1}{r_1}\delta n^{1-a\left(1+\lambda_1/r\right)} C\left(a,n\right)\right)\right),
\end{align*}
where $\theta \in \left(0,1\right)$. The rest of the proof directly follows that of Proposition \ref{prop_recurrence_1}. \qed
\end{proof}
\begin{proposition}\label{proposition_crossover_2}
	There exists $C>0$ such that
	\begin{equation*}
	   \PP\left(\sup\limits_{u\in\left[0,a\right]}\left(A_3\left(u,n\right)+\left(1-\delta\right) A_4\left(u,n\right)\right)>0\right)=O\left(e^{-Cn^{1-a}}\right). 
	\end{equation*}
\end{proposition}
\begin{proof}\\
Observe that
\begin{align*}
& \quad \PP\left(\sup\limits_{u\in\left[0,a\right]}\left(A_3\left(u,n\right)+\left(1-\delta\right) A_4\left(u,n\right)>0\right)\right)\\
& \le \PP\left(\sup\limits_{u\in\left[0,a\right]}\left(n^{u-1}A_3\left(u,n\right)+\sup\limits_{u\in\left[0,a\right]}\left(1-\delta\right) n^{u-1}A_4\left(u,n\right)>0\right)\right)\\
& =\PP\left(\sup\limits_{u\in\left[0,a\right]}\left(n^{u-1}A_3\left(u,n\right)>\inf\limits_{u\in\left[0,a\right]}-\left(1-\delta\right) n^{u-1}A_4\left(u,n\right)\right)\right).
\end{align*}
We may compute that
\begin{align*}
& \quad -n^{u-1}A_4\left(u,n\right)\\
&=n^{u-1}\left(z_0^n\left(ut_n\right)-z_1^n\left(ut_n\right)-z_2^n\left(ut_n\right)\right)\\
&=n^{u-1}\left(n^{1-u}-X\left(n\right)n^{\frac{\lambda_1}{r}u}-\frac{\mu}{\lambda_1+r}n^{1-\alpha}e^{\lambda_1 u t_n}\left(1-e^{\left(\lambda_0-\lambda_1\right)u t_n}\right)\right)\\
&=1-X\left(n\right)n^{\left(1+\frac{\lambda_1}{r}\right)u-1}+\frac{\mu n^{-\alpha}}{\lambda_1+r}-\frac{\mu n^{u\left(1+\frac{\lambda_1}{r}\right)-\alpha}}{\lambda_1+r},
\end{align*}
which is a monotone decreasing function in $u$ and consequently obtain that
\begin{align*}
& \quad \PP\left(\sup\limits_{u\in\left[0,a\right]}\left(n^{u-1}A_3\left(u,n\right)>\inf\limits_{u\in\left[0,a\right]}-\left(1-\delta\right) n^{u-1}A_4\left(u,n\right)\right)\right)\\
& = \PP\left(\sup\limits_{u\in\left[0,a\right]}\left(n^{u-1}A_3\left(u,n\right)>\left(1-\delta\right) n^{a-1}\left(z_0^n\left(at_n\right)-z_1^n\left(at_n\right)-z_2^n\left(at_n\right)\right)\right)\right).
\end{align*}
If we define $\theta_n=\theta n^{1-a}>0$ then
\begin{align*}
& \quad \PP\left(\sup\limits_{u\in\left[0,a\right]}\left(n^{u-1}A_3\left(u,n\right)>\left(1-\delta\right) n^{a-1}\left(z_0^n\left(at_n\right)-z_1^n\left(at_n\right)-z_2^n\left(at_n\right)\right)\right)\right)\\
& =  \PP\left(\sup\limits_{u\in\left[0,a\right]}\exp\left(\theta_n n^{u-1}A_3\left(u,n\right)\right)>\exp\left(\left(1-\delta\right) \theta_n n^{a-1}\left(z_0^n\left(at_n\right)-z_1^n\left(at_n\right)-z_2^n\left(at_n\right)\right)\right)\right).
\end{align*}
It's easy to show that $\exp\left(\theta_n n^{u-1}A_3\left(u,n\right)\right)$ is a submartingale in $u$. From Doob's Inequality we obtain that
\begin{align*}
& \quad \PP\left(\sup\limits_{u\in\left[0,a\right]}n^{u-1}A_3\left(u,n\right)>\left(1-\delta\right) n^{a-1}\left(z_0^n\left(at_n\right)-z_1^n\left(at_n\right)-z_2^n\left(at_n\right)\right)\right)\\
& \le \exp\left(- \theta_n\left(1-\delta\right) n^{a-1}\left(z_0^n\left(at_n\right)-z_1^n\left(at_n\right)-z_2^n\left(at_n\right)\right)\right)\EE\exp\left[\theta_n n^{a-1}A_3\left(a,n\right)\right]\\
& =\exp \left[\delta \theta_n+\theta\left(1-\delta\right)X\left(n\right)n^{\frac{\lambda_1 a}{r}}+\frac{\left(1-\delta\right)\theta \mu n^{1-\alpha+\frac{\lambda_1 a}{r}}}{\left(\lambda_1+r\right)}\left(1-n^{-a\left(1+\frac{\lambda_1}{r}\right)}\right)\right]\EE\exp\left[-\theta Z_0^n\left(at_n\right)\right]\\
& \le \exp \left[\delta \theta_n+\theta\left(1-\delta\right)X\left(n\right)n^{\frac{\lambda_1 a}{r}}+\frac{\theta \mu n^{1-\alpha+\frac{\lambda_1 a}{r}}}{\left(\lambda_1+r\right)}\right]\EE\exp\left[-\theta Z_0^n\left(at_n\right)\right],
\end{align*}
where the last inequality holds for $n$ sufficiently large. To continue our analysis, we will state and prove the following lemma:
\begin{lemma}\label{z0_decay}
For $\theta>0$, $\EE \exp \left[-\theta Z_0^n\left(at_n\right)\right]\le \exp\left[-k\left(\theta\right) n^{1-a}\left(1+o\left(1\right)\right)\right]$, where $k\left(\theta\right)$ is a positive number depending on $\theta$. Moreover, $\lim\limits_{\theta \downarrow 0}\frac{k\left(\theta\right)}{\theta}=1$.
\end{lemma}
Proof. Let $\{Z_0\left(t\right), t\ge 0\}$ be a subcritical branching process where $Z_0\left(0\right)=1$. If we define the function $g_t\left(\theta\right)=\EE \left[\exp\left(-\theta Z_0\left(t\right)\right)\right]$, then
\begin{equation*}
g_t\left(\theta\right)=\frac{e^{rt}\left(r_0e^{-\theta}-d_0\right)-d_0\left(e^{-\theta}-1\right)}{e^{rt}\left(r_0e^{-\theta}-d_0\right)-r_0\left(e^{-\theta}-1\right)}.
\end{equation*}
Let $c=r_0e^{-\theta}-d_0$, $e=d_0\left(e^{-\theta}-1\right)$, and $f=r_0\left(e^{-\theta}-1\right)$. Then
\begin{align*}
g_{at_n}\left(\theta\right)& =\frac{cn^a-e}{cn^a-f}\\
&=1+\frac{f}{cn^a-f}-\frac{e}{cn^a-f}\\
&=1+\frac{f}{cn^a}+\frac{f^2}{c^2n^{2a}-cfn^a}-\frac{e}{cn^a}-\frac{ef}{c^2n^{2a}-cfn^a}\\
& =1-\frac{e-f}{cn^a}-\frac{e-f}{cn^a}\left[\frac{f}{cn^a-f}\right]
\end{align*}
and note that
\begin{equation*}
g_{at_n}\left(\theta\right)=1-k\left(\theta\right)n^{-a}\left(1+o\left(1\right)\right).
\end{equation*}
where $k\left(\theta\right)=\frac{e-f}{c}>0$. From the independence of the branching process we know that
\begin{equation*}
\EE \exp \left[-\theta Z_0^n\left(at_n\right)\right]=\left(g_{at_n}\left(\theta\right)\right)^n\le \exp\left[-kn^{1-a}\left(1+o\left(1\right)\right)\right].
\end{equation*}
Moreover, 
\begin{align*}
\lim\limits_{\theta \downarrow 0}\frac{k\left(\theta\right)}{\theta}&=\lim\limits_{\theta \downarrow 0}\frac{d_0\left(e^{-\theta}-1\right)-r_0\left(e^{-\theta}-1\right)}{\left(r_0e^{-\theta}-d_0\right)\theta}\\
& =\lim\limits_{\theta \downarrow 0}\frac{-d_0e^{-\theta}+r_0e^{-\theta}}{-r_0e^{-\theta}\theta+r_0e^{-\theta}-d_0}\\
& =1.
\end{align*} \qed %proof checked
\\
We may now see from Lemma \ref{z0_decay} that
\begin{align*}
& \quad \PP\left(\sup\limits_{u\in\left[0,a\right]}\left(n^{u-1}A_3\left(u,n\right)>\left(1-\delta\right) n^{a-1}\left(z_0^n\left(at_n\right)-z_1^n\left(at_n\right)-z_2^n\left(at_n\right)\right)\right)\right)\\
& \le \exp \left[\delta \theta n^{1-a}+\theta\left(1-\delta\right)X\left(n\right)n^{\frac{\lambda_1 a}{r}}+\frac{\theta \mu n^{1-\alpha+\frac{\lambda_1 a}{r}}}{\left(\lambda_1+r\right)}-k\left(\theta\right) n^{1-a}\left(1+o\left(1\right)\right)\right]\\
& \le \exp \left[-n^{1-a}\left(k\left(\theta\right)\left(1+o\left(1\right)\right)-\delta \theta -\theta\left(1-\delta\right)X\left(n\right)n^{\frac{\lambda_1 a}{r}+a-1}-\frac{\theta \mu n^{a\left(1+\frac{\lambda_1}{r}\right)-\alpha}}{\left(\lambda_1+r\right)}\right)\right].
\end{align*}
Using the fact that $a\left(1+\frac{\lambda_1}{r}\right)-\alpha<0$, $X\left(n\right)=o\left(n^{1-\alpha}\right)$, and taking $\theta$ sufficiently small, we may define $C$ such that $0<C<k\left(\theta\right)-\delta \theta$ and obtain
\begin{equation*}
\exp \left[-n^{1-a}\left(k\left(\theta\right)\left(1+o\left(1\right)\right)-\delta \theta -\theta\left(1-\delta\right)X\left(n\right)n^{\frac{\lambda_1 a}{r}+a-1}-\frac{\theta \mu n^{a\left(1+\frac{\lambda_1}{r}\right)-\alpha}}{\left(\lambda_1+r\right)}\right)\right]=O\left(e^{-Cn^{1-a}}\right),
\end{equation*}
which completes the proof. \qed
\end{proof}
\\

\begin{proposition} \label{proposition_crossover_3}
There exists $C>0$ such that
\begin{equation*}
    \PP\left(\sup\limits_{u\in\left[a,u_n\left(y\right)\right]}\left(A_3\left(u,n\right)+\delta A_4\left(u,n\right)\right)>0\right)=O\left(\exp\left(-Cn^{1-\alpha r/\left(\lambda_1+r\right)}\right)\right).
\end{equation*}
\end{proposition}
\begin{proof}\\
Observe that
\begin{align*}
& \quad \PP\left(\sup\limits_{u\in\left[a,u_n\left(y\right)\right]}\left(A_3\left(u,n\right)+\delta A_4\left(u,n\right)\right)>0\right)\\
&\le \PP\left(\sup\limits_{u\in\left[a,u_n\left(y\right)\right]}n^{u-1}A_3\left(u,n\right)+\sup\limits_{u\in\left[a,u_n\left(y\right)\right]} \delta n^{u-1} A_4\left(u,n\right)>0\right)\\
& = \PP\left(\sup\limits_{u\in\left[a,u_n\left(y\right)\right]}n^{u-1}A_3\left(u,n\right)> \delta n^{u_n\left(y\right)-1} \left(z_0^n\left(u_n\left(y\right)t_n\right)-z_1^n\left(u_n\left(y\right)t_n\right)-z_2^n\left(u_n\left(y\right)t_n\right)\right)\right).
\end{align*}
It's easy to show that $n^{u-1}A_3\left(u,n\right)$ is a martingale in $u$. Set
\begin{align*}
& C\left(u_n\left(y\right),n\right)=n^{u_n\left(y\right)-1} \left(z_0^n\left(u_n\left(y\right)t_n\right)-z_1^n\left(u_n\left(y\right)t_n\right)-z_2^n\left(u_n\left(y\right)t_n\right)\right),
\end{align*}
and $\theta_n=\theta n^{1-u_n\left(y\right)}$, where $\theta>0$, then
\begin{align*}
& \quad \PP\left(\sup\limits_{u\in\left[a,u_n\left(y\right)\right]}n^{u-1}A_3\left(u,n\right)> \delta C\left(u_n\left(y\right),n\right)\right)\\
& =\PP\left(\sup\limits_{u\in\left[a,u_n\left(y\right)\right]}\exp \left(\theta_n n^{u-1}A_3\left(u,n\right)\right)> \exp\left(\theta_n \delta C\left(u_n\left(y\right),n\right)\right)\right)\\
& \le \exp\left(-\theta_n \delta C\left(u_n\left(y\right),n\right)\right)\EE \exp \left[\theta_n n^{u_n\left(y\right)-1}A_3\left(u_n\left(y\right),n\right)\right]\\
& = \exp\left[-\theta_n \delta \left(1-X\left(n\right)n^{\left(1+\frac{\lambda_1}{r}\right)u_n\left(y\right)-1}+\frac{\mu n^{-\alpha}}{\lambda_1+r}-\frac{\mu n^{u_n\left(y\right)\left(1+\frac{\lambda_1}{r}\right)-\alpha}}{\lambda_1+r}\right)\right]\EE \exp \left[\theta A_3\left(u_n\left(y\right),n\right)\right]\\
&=  \exp\left[-\theta_n \delta \left(1-X\left(n\right)n^{\left(1+\frac{\lambda_1}{r}\right)u_n\left(y\right)-1}+\frac{\mu n^{-\alpha}}{\lambda_1+r}-\frac{\mu n^{u_n\left(y\right)\left(1+\frac{\lambda_1}{r}\right)-\alpha}}{\lambda_1+r}\right)\right]\exp\left(\theta_n\right)\EE \exp \left[-\theta Z_0^n\left(u_n\left(y\right)t_n\right)\right]\\
&= \exp\left[\theta_n n^{u_n\left(y\right)-1}\left(\left(1-\delta\right)n^{1-u_n\left(y\right)}+\frac{\delta \mu n^{1-\alpha}}{\lambda_1+r}\left[n^{\lambda_1 u_n\left(y\right)/r}-n^{-u_n\left(y\right)}\right]+\delta X\left(n\right)n^{\lambda_1 u_n\left(y\right)/r}\right)\right]\\
& \quad \times \EE \exp \left[-\theta Z_0^n\left(u_n\left(y\right)t_n\right)\right]\\
& \le \exp\left[\theta_n \left(\left(1-\delta\right)+\frac{\delta \mu }{\lambda_1+r}n^{u_n\left(y\right)\left(1+\lambda_1/r\right)-\alpha}+\delta X\left(n\right)n^{u_n\left(y\right)\left(1+\lambda_1/r\right)-1}\right)\right]\EE \exp \left[-\theta Z_0^n\left(u_n\left(y\right)t_n\right)\right]\\
& = \exp\left[\theta \left(1-\delta\right) n^{1-u_n\left(y\right)}+\frac{\theta \delta \mu }{\lambda_1+r}n^{1+u_n\left(y\right)\lambda_1/r-\alpha}+\theta \delta X\left(n\right)n^{u_n\left(y\right)\lambda_1/r}\right]\EE \exp \left[-\theta Z_0^n\left(u_n\left(y\right)t_n\right)\right].
\end{align*}
We may apply the proof of Lemma \ref{z0_decay} to show that
\begin{align*}
\EE \exp\left[-\theta Z_0^n\left(u_n\left(y\right)t_n\right)\right]\le \exp\left(-k\left(\theta\right)n^{1-u_n\left(y\right)}\left(1+o\left(1\right)\right)\right).
\end{align*}
Hence,
\begin{align*}
& \quad \PP\left(\sup\limits_{u\in\left[a,u_n\left(y\right)\right]}n^{u-1}A_3\left(u,n\right)> \delta C\left(u_n\left(y\right),n\right)\right)\\
& \le \exp\left(n^{1-u_n\left(y\right)}\left[\theta \left(1-\delta\right)-k\left(\theta\right)\left(1+o\left(1\right)\right)\right]+\frac{\theta \delta \mu }{\lambda_1+r}n^{1+u_n\left(y\right)\lambda_1/r-\alpha}+\theta \delta X\left(n\right)n^{u_n\left(y\right)\lambda_1/r}\right)\\
& =\exp\left[-n^{1-u_n\left(y\right)} \left(\left[k\left(\theta\right)\left(1+o\left(1\right)\right)-\theta \left(1-\delta\right)\right]-\frac{\theta \delta \mu }{\lambda_1+r}n^{u_n\left(y\right)\left(1+\lambda_1/r\right)-\alpha}-\theta \delta X\left(n\right)n^{u_n\left(y\right)\left(1+\lambda_1/r\right)-1}\right)\right]\\
& =\exp\left[-n^{1-u_n\left(y\right)} \left(\left[k\left(\theta\right)\left(1+o\left(1\right)\right)-\theta\right]+\theta \delta\left[1-\frac{\mu }{\lambda_1+r}n^{u_n\left(y\right)\left(1+\lambda_1/r\right)-\alpha}-X\left(n\right)n^{u_n\left(y\right)\left(1+\lambda_1/r\right)-1}\right]\right)\right].
\end{align*}
We can show that
\begin{align*}
& \quad n^{u_n\left(y\right)\left(1+\lambda_1/r\right)-\alpha}\rightarrow \exp\left(-y\left(r+\lambda_1\right)\right)\frac{\lambda_1+r}{\mu} \text{ as } n\rightarrow \infty, \text{ and }\\
& \quad X\left(n\right) n^{u_n\left(y\right)\left(1+\lambda_1/r\right)-1}\rightarrow 0 \text{ as } n\rightarrow \infty.
\end{align*}
Hence if $\theta$ is chosen to be small enough, then by Lemma \ref{z0_decay}, for $n$ sufficiently large the quantity
\begin{align*}
& \left[k\left(\theta\right)\left(1+o\left(1\right)\right)-\theta\right]+\theta \delta\left[1-\frac{\mu }{\lambda_1+r}n^{u_n\left(y\right)\left(1+\lambda_1/r\right)-\alpha}-X\left(n\right)n^{u_n\left(y\right)\left(1+\lambda_1/r\right)-1}\right]
\end{align*}
is strictly positive. In fact, for $n$ sufficiently large, the quantity is bounded below by some positive constant $K$. Then we have that
\begin{equation*}
\PP\left(\sup\limits_{u\in\left[a,u_n\left(y\right)\right]}n^{u-1}A_3\left(u,n\right)> \delta C\left(u_n\left(y\right),n\right)\right)\le \exp\left(-Kn^{1-u_n\left(y\right)}\right).
\end{equation*}
We can show that for $n$ sufficiently large
\begin{align*}
& n^{1-u_n\left(y\right)}\ge  n \exp\left(ry\right)\left(2n^{\alpha}\left(\frac{\lambda_1+r}{\mu}\right)\right)^{-r/\left(\lambda_1+r\right)}.
\end{align*}
Hence, for $n$ sufficiently large,
\begin{align*}
& \quad \PP\left(\sup\limits_{u\in\left[a,u_n\left(y\right)\right]}\left(A_3\left(u,n\right)+\delta A_4\left(u,n\right)>0\right)\right)\\
& \le \PP\left(\sup\limits_{u\in\left[a,u_n\left(y\right)\right]}n^{u-1}A_3\left(u,n\right)> \delta C\left(u_n\left(y\right),n\right)\right)\\
& \le \exp\left(-Kn \exp\left(ry\right)\left(2n^{\alpha}\left(\frac{\lambda_1+r}{\mu}\right)\right)^{-r/\left(\lambda_1+r\right)}\right)\\
& =O\left(\exp\left(-Cn^{1-\alpha r/\left(\lambda_1+r\right)}\right)\right)
\end{align*}
where $C=K\exp\left(ry\right)\left(2\left(\frac{\lambda_1+r}{\mu}\right)\right)^{-r/\left(\lambda_1+r\right)}$, completing the proof. \qed
\end{proof}
\\

\begin{proposition}\label{proposition_crossover_4}
There exists $C>0$ such that
\begin{equation*}\label{sub_1}
\PP\left(Z_0^n\left(u_n\left(y\right)t_n\right)> \left(1+a_n\right)z_0^n\left(u_n\left(y\right)t_n\right)\right)=O\left(\exp\left(-Cn^{1-\alpha\left(2r+\lambda_1\right)/2\left(\lambda_1+r\right)}\right)\right).
\end{equation*}
\end{proposition}
\begin{proof}\\
Recall that the moment generating function of $Z_0\left(t\right)$ is defined as
\[ \EE\exp\left(\theta Z_0\left(t\right)\right)=\EE\exp\left(\theta \tilde{Z}\left(t\right)\right)=\begin{cases} 
\frac{d_0\left(e^{\theta}-1\right)-e^{rt}\left(r_0e^{\theta}-d_0\right)}{r_0\left(e^{\theta}-1\right)-e^{rt}\left(r_0e^{\theta}-d_0\right)}, & \theta<\bar{\theta}_t \\
\infty & \theta\ge\bar{\theta}_t,
\end{cases}
\]
where
\begin{equation*}
\bar{\theta}_t\doteq\log\left(1+\frac{r}{r_0\left(1-e^{-rt}\right)}\right).
\end{equation*}
From the fact $u_n\left(y\right)t_n\rightarrow \infty$ as $n\rightarrow \infty$, we see that for $n$ sufficiently large, $\bar{\theta}_{u_n\left(y\right)t_n}$ is bounded below by $\log\left(1+\frac{r}{r_0}\right)>0$. Therefore, if we define a positive sequence $\{b_n\}_{n\in \NN}$ such that $b_n\downarrow 0$ then there exists $N\in \NN$ such that for $n>N$, $0<b_n<\log\left(1+\frac{r}{r_0}\right)$, and $\bar{\theta}_{u_n\left(y\right)t_n}\ge \log\left(1+\frac{r}{r_0}\right)$. It follows for $n\ge N$ that
\begin{align*}
& \quad \PP\left(Z_0^n\left(u_n\left(y\right)t_n\right)> \left(1+a_n\right)z_0^n\left(u_n\left(y\right)t_n\right)\right)\\
&\le \left(\EE\left[\exp\left(b_nZ_0\left(u_n\left(y\right)t_n\right)\right)\right]\right)^n\exp\left(-b_n\left(1+a_n\right)z_0^n\left(u_n\left(y\right)t_n\right)\right)\\
&=\left(1+\frac{\left(d_0-r_0\right)\left(e^{b_n}-1\right)}{r_0\left(e^{b_n}-1\right)+e^{ru_n\left(y\right)t_n}\left(d_0-r_0e^{b_n}\right)}\right)^n\exp\left(-b_n\left(1+a_n\right)z_0^n\left(u_n\left(y\right)t_n\right)\right)\\
&\le \exp\left(\frac{n\left(d_0-r_0\right)\left(e^{b_n}-1\right)}{r_0\left(e^{b_n}-1\right)+n^{u_n\left(y\right)}\left(d_0-r_0e^{b_n}\right)}\right)\exp\left(-b_n\left(1+a_n\right)z_0^n\left(u_n\left(y\right)t_n\right)\right)\\
&\le \exp\left(\frac{n\left(d_0-r_0\right)\left(e^{b_n}-1\right)}{n^{u_n\left(y\right)}\left(d_0-r_0e^{b_n}\right)}\right)\exp\left(-b_n\left(1+a_n\right)z_0^n\left(u_n\left(y\right)t_n\right)\right)\\
&=\exp\left(\frac{n^{1-u_n\left(y\right)}\left(d_0-r_0\right)\left(e^{b_n}-1\right)}{d_0-r_0-r_0\left(e^{b_n}-1\right)}\right)\exp\left(-b_n\left(1+a_n\right)z_0^n\left(u_n\left(y\right)t_n\right)\right).
\end{align*}
Since $b_n\downarrow 0$ there exists $N_1\in \NN$ such that $0<b_n<\min\left(\log\left(1+\frac{r}{r_0}\right),1\right)$ and $d_0-r_0\left(\frac{1}{1-b_n}\right)>0$ for $n\ge N_1$, we may apply the inequality $e^{b_n}-1<\frac{b_n}{1-b_n}$ to obtain that
\begin{equation*}
\exp\left(\frac{n^{1-u_n\left(y\right)}\left(d_0-r_0\right)\left(e^{b_n}-1\right)}{d_0-r_0-r_0\left(e^{b_n}-1\right)}\right)\le \exp\left(\frac{n^{1-u_n\left(y\right)}\left(d_0-r_0\right)\frac{b_n}{1-b_n}}{d_0-r_0\left(\frac{1}{1-b_n}\right)}\right)=\exp\left(\frac{b_nn^{1-u_n\left(y\right)}\left(d_0-r_0\right)}{d_0-r_0-b_nd_0}\right).
\end{equation*}
Using the fact $z_0^n\left(u_n\left(y\right)t_n\right)=n^{1-u_n\left(y\right)}$ we now have that
\begin{align*}
&\quad \PP\left(Z_0^n\left(u_n\left(y\right)t_n\right)>\left(1+a_n\right)z_0^n\left(u_n\left(y\right)t_n\right)\right)\\
&\le \exp\left(-b_nn^{1-u_n\left(y\right)}\left[\left(1+a_n\right)-\frac{d_0-r_0}{d_0-r_0-b_nd_0}\right]\right)\\
&=\exp\left(-b_nn^{1-u_n\left(y\right)}\left[a_n-\frac{b_nd_0}{d_0-r_0-b_nd_0}\right]\right)\\
&=\exp\left(-b_na_nn^{1-u_n\left(y\right)}\left[1-\frac{b_nd_0}{a_n\left(d_0-r_0-b_nd_0\right)}\right]\right).
\end{align*}
It's easy to show that $n^{1-u_n\left(y\right)}=\Theta\left(n^{1-\alpha r/\left(\lambda_1+r\right)}\right)$. Now if we set $b_n=a_n^2$ and $a_n=n^{-\frac{\alpha \lambda_1}{6\left(\lambda_1+r\right)}}$ then
\begin{equation*}
\exp\left(-b_na_nn^{1-u_n\left(y\right)}\left[1-\frac{b_nd_0}{a_n\left(d_0-r_0-b_nd_0\right)}\right]\right)=\exp\left(-{a_n}^3n^{1-u_n\left(y\right)}\left[1-\frac{a_nd_0}{\left(d_0-r_0-b_nd_0\right)}\right]\right).
\end{equation*}
It follows that ${a_n}^3n^{1-u_n\left(y\right)}=\Theta \left(n^{1-\alpha\left(2r+\lambda_1\right)/2\left(\lambda_1+r\right)}\right)$, and the desired result follows. \qed
\end{proof}
\\

\begin{proposition}\label{G_E}
The sequence of random variables $\{\tilde{Z}_n\}_{n\ge 1}$ as defined in (\ref{rv}) satisfy the conditions of the Gartner-Ellis Theorem with $\Lambda\left(\theta\right)=\frac{\lambda_1\mu \theta}{r_1}\int_{0}^{\infty}\frac{e^{-rs}}{e^{\lambda_1s}-\theta}$ for all $\theta\in \RR$.
\end{proposition}
\begin{proof}\\
The proof directly follows that of Proposition \ref{GE_cond}. \qed
\end{proof}

%%%%%%%%%%%%%%%%%%%%%%%%%%%%%%%%%%%%%%%%%%%%%%%%%%%%%%%%%%%%%%%%%%%%%%%%%%%%%%%%%%%%%%%%%%%%%%%%%%%%%%%%%%%%%%%%

\subsection{Proof of Proposition \ref{LD_condition}}

The proof is very similar to the proof of Theorem \ref{recurrence_LD}. We first establish the large deviations upper bound. Let $\delta\in \left(0,1\right)$, we observe that
\begin{align}
& \quad \PP_{A_{n,a}}\left(\sup\limits_{t\in \left[0,\zeta_n-y\right]}\hat{Z}^n_S\left(t\right)+\hat{Z}^n_E\left(t\right)-n>0\right) \nonumber\\
& \le \PP_{A_{n,a}}\left(\sup\limits_{t\in \left[0,\zeta_n-y\right]}\hat{Z}^n_E\left(t\right)-\delta n>0\right)+\PP_{A_{n,a}}\left(\sup\limits_{t\in \left[0,\zeta_n-y\right]}\hat{Z}^n_S\left(t\right)-\left(1-\delta\right) n>0\right). \label{CD_1}
\end{align}
To complete the proof, we need to further denote by $C_n$ the number of clones generated in the time period $\left(0,\zeta_n-y\right)$, which will extinct eventually, and let $\hat{Z}_C^n\left(t\right)$ be the number of mutants at time $t$, which belong to those clones. Fix $0<\beta<\alpha$, we observe that for the first term in $\ref{CD_1}$
\begin{align*}
& \quad \PP_{A_{n,a}}\left(\sup\limits_{t\in \left[0,\zeta_n-y\right]}\hat{Z}^n_E\left(t\right)-\delta n>0\right)\\
& = \PP\left(\sup\limits_{t\in \left[0,\zeta_n-y\right]}\hat{Z}^n_E\left(t\right)-\delta n>0\right)\\
& \le \PP\left(\sup\limits_{t\in \left[0,\zeta_n-y\right]}\hat{Z}^n_C\left(t\right)-\delta n>0\right)\\
& = \PP\left(\sup\limits_{t\in \left[0,\zeta_n-y\right]}\hat{Z}^n_C\left(t\right)-\delta n>0, C_n\ge n^{1-\beta}\right)+\PP\left(\sup\limits_{t\in \left[0,\zeta_n-y\right]}\hat{Z}^n_C\left(t\right)-\delta n>0, C_n< n^{1-\beta}\right),
\end{align*}
where the first equality follows from the independence between $\hat{Z}_E^n$ and $\hat{Z}_S^n$. We know that $C_n$ follows Poisson distribution with mean of order $O\left(n^{1-\alpha}\right)$. If $\beta\in \left(0,\alpha\right)$, then
\begin{align}
\PP\left(\sup\limits_{t\in \left[0,\zeta_n-y\right]}\hat{Z}^n_C\left(t\right)-\delta n>0, C_n\ge n^{1-\beta}\right)& \le \PP\left(C_n\ge n^{1-\beta}\right) \nonumber \\
& = O\left(e^{-n^{1-\beta}}\right) \nonumber \\
& = o\left(e^{-n^{1-\alpha}}\right). \label{dom_1}
\end{align}
By a Gambler's Ruin argument, we obtain that
\begin{align}
\PP\left(\sup\limits_{t\in \left[0,\zeta_n-y\right]}\hat{Z}^n_C\left(t\right)-\delta n>0, C_n< n^{1-\beta}\right)=O\left(e^{-cn}\right), \label{dom_2}
\end{align}
for some positive constant $c$. For the second term in \ref{CD_1}, we have the following lemma.

\begin{lemma}\label{condition_ineq}
\begin{align*}
& \quad \PP_{A_{n,a}}\left(\sup\limits_{t\in \left[0,\zeta_n-y\right]}\hat{Z}^n_S\left(t\right)-\left(1-\delta\right) n>0\right)\\
& \le \EE_{A_{n,a}}\left[\exp\left(\theta e^{-\lambda_1 \left(\zeta_n-y\right)}\hat{Z}_S^n\left(\zeta_n-y\right)\right)\right]\exp\left(-\theta e^{-\lambda_1 \left(\zeta_n-y\right)}\left(1-\delta\right)n\right).
\end{align*}
\end{lemma}

\begin{proof}\\
We notice that conditioned on the event of $A_{n,a}$, $e^{-\lambda_1 t}\hat{Z}_S^n\left(t\right)$ is not adapted to the natural filtration generated by $\left(Z_0^n, Z_1^n, Z_2^n\right)$. Therefore, we need to construct an appropriate filtration.

For $0\le t_1 \le t_2 \le T$, and $t\in \left[0, T\right]$, denote by $\bar{Z}_T\left(t\right)$ the number of cells of $Z$ conditioned on the event that $Z\left(T\right)>0$. Let $\mathcal{\bar{F}}_t=\{Z(T)>0\} \cap \mathcal{F}_t$ for $t\in \left[0,T\right]$. We can observe that for $k\in \NN$, and $t\in \left[0,T\right]$,
\begin{align*}
\{\omega: \bar{Z}(t)=k\}&=\{\omega: Z(t)=k, Z(T)>0\}\\
&\in \{Z(T)>0\} \cap \mathcal{F}_t.
\end{align*}
Hence, $\bar{Z}_T\left(t\right)\in \mathcal{\bar{F}}_t$. We then obtain that for $i,j, z\left(u\right)\in \NN$, $0\le u\le t_1$,
\begin{align*}
& \quad \PP\left(Z(t_2)=j|Z\left(t_1\right)=k; Z\left(u\right)=z\left(u\right), 0\le u\le t_1; Z\left(T\right)>0\right)\\
& =\frac{\PP\left(Z(t_2)=j;Z\left(t_1\right)=k; Z\left(u\right)=z\left(u\right), 0\le u\le t_1; Z\left(T\right)>0\right)}{\PP\left(Z\left(t_1\right)=k; Z\left(u\right)=z\left(u\right), 0\le u\le t_1; Z\left(T\right)>0\right)}\\
& =\frac{\PP\left(Z\left(T\right)>0|Z(t_2)=j;Z\left(t_1\right)=k; Z\left(u\right)=z\left(u\right), 0\le u\le t_1\right)}{\PP\left(Z\left(t_1\right)=k; Z\left(u\right)=z\left(u\right), 0\le u\le t_1; Z\left(T\right)>0\right)}\\
& \quad \quad \times \PP\left(Z(t_2)=j;Z\left(t_1\right)=k; Z\left(u\right)=z\left(u\right), 0\le u\le t_1\right)\\
&=\frac{\PP\left(Z\left(T\right)>0|Z(t_2)=j;Z\left(t_1\right)=k\right)}{\PP\left(Z\left(T\right)>0|Z\left(t_1\right)=k\right)}\PP\left(Z(t_2)=j|Z\left(t_1\right)=k\right)\\
& =\PP\left(Z(t_2)=j|Z\left(t_1\right)=k; Z\left(T\right)>0\right).
\end{align*}
Therefore, $\bar{Z}_T\left(t\right)$ is still a Markov process, and we have
\begin{align}
\EE\left[e^{-\lambda_1 t_2}\bar{Z}_T\left(t_2\right)|\mathcal{\bar{F}}_{t_1}\right]&=\EE\left[e^{-\lambda_1 t_2}\bar{Z}_T\left(t_2\right)|\bar{Z}_T\left(t_1\right)\right] \nonumber\\
&=\EE\left[e^{-\lambda_1 t_2}Z\left(t_2\right)| Z\left(t_1\right) = \bar{Z}_T\left(t_1\right); Z\left(T\right)>0 \right] \nonumber\\
& = \frac{\EE\left[e^{-\lambda_1 t_2}Z\left(t_2\right)|Z\left(t_1\right)=\bar{Z}_T\left(t_1\right)\right]}{\PP\left( Z\left(T\right)>0|Z\left(t_1\right)=\bar{Z}_T\left(t_1\right)\right)} \nonumber \\
& \quad \quad -\frac{\EE\left[e^{-\lambda_1 t_2}Z\left(t_2\right)|Z\left(t_1\right)=\bar{Z}_T\left(t_1\right); Z\left(T\right)=0\right]\PP\left(Z\left(T\right)=0|  Z\left(t_1\right)=\bar{Z}_T\left(t_1\right)\right)}{\PP\left( Z\left(T\right)>0|Z\left(t_1\right)=\bar{Z}_T\left(t_1\right)\right)}\nonumber \\
& \ge e^{-\lambda_1 t_1}\bar{Z}_T\left(t_1\right).\label{supmartingale}
\end{align}
The last inequality follows from the following calculation:
\begin{align*}
& \quad \EE\left[e^{-\lambda_1 t_2}Z\left(t_2\right)|Z\left(t_1\right)=\bar{Z}_T\left(t_1\right); Z\left(T\right)=0\right]\\
& =\bar{Z}_T\left(t_1\right)\EE\left[e^{-\lambda_1 t_2}Z\left(t_2-t_1\right)| Z\left(T-t_1\right)=0;  Z\left(0\right)=1\right]\\
& =\sum_{n=1}^{\infty}n \left(\frac{\lambda_1}{r_1-d_1e^{-\lambda_1 \left(t_2-t_1\right)}}\right)\left(\frac{r_1\left(1-e^{-\lambda_1 \left(t_2-t_1\right)}\right)}{r_1-d_1 e^{-\lambda_1 \left(t_2-t_1\right)}}\right)^{n-1}\left(\frac{\lambda_1 e^{-\lambda_1 \left(t_2-t_1\right)}}{r_1-d_1 e^{-\lambda_1 \left(t_2-t_1\right)}}\right)\left(\frac{d_1\left(1-e^{-\lambda_1\left(T-t_2+t_1\right)}\right)}{r_1-d_1e^{-\lambda_1\left(T-t_2+t_1\right)}}\right)^n\\
& \quad \quad \times \bar{Z}_T\left(t_1\right)\\
& \le e^{-\lambda_1 t_1}\bar{Z}_T\left(t_1\right).
\end{align*}
Conditioned on the event $A_{n,a}$, 
%$\hat{Z}_S^n$ is not adapted to the natural filtration generated by $\left(Z_0^n, Z_1^n, Z_2^n\right)$. 
we can obtain the finite dimensional distributions of $\hat{Z}_S^n$ by considering the summation of $\lfloor a\frac{\lambda_1 \mu}{r r_1}n^{1-\alpha} \rfloor$ independent processes $Z_i, 1\le i\le \lfloor a\frac{\lambda_1 \mu}{r r_1}n^{1-\alpha} \rfloor$. $Z_i$ is constructed such that $s\stackrel{\mathrm{def}}{=} \inf\{t: t\in \left[0, \zeta-y\right]|Z_i\left(t\right)\ge 1\}$ has pdf $\frac{\frac{\lambda_1}{r_1-d_1e^{-\lambda_1\left(\zeta_n-y-s\right)}}e^{-rs}}{\int_{0}^{\zeta_n-y}\frac{\lambda_1}{r_1-d_1e^{-\lambda_1\left(\zeta_n-y-s\right)}}e^{-rs}ds}$, $Z_i\left(t\right)=0$ for $t\in \left[0,s\right)$, and is stochastically equivalent to $\bar{Z}_{\zeta_n-y-s}$ for $t\in \left[s,\zeta_n-y\right]$. We then construct the probability space $\left(\Omega, \mathcal{A}, P_{A_{n,a}}\right)$ with the natural filtration $\left(\mathcal{F}_t^{A_{n,a}}; t\in \left[0, \zeta_n-y\right]\right)$ generated by $\left(Z_1,...,Z_{\lfloor a\frac{\lambda_1 \mu}{r r_1}n^{1-\alpha} \rfloor}\right)$. From \ref{supmartingale}, we can obtain that for $0\le t_1 \le t_2 \le \zeta_n-y$,
\begin{align}\label{Z_S_supermartingale}
\quad \EE_{A_{n,a}}\left[e^{-\lambda_1 t_2}\hat{Z}_S^n\left(t_2\right)|\mathcal{F}^{A_{n,a}}_{t_1}\right]\ge e^{-\lambda_1 t_1}\hat{Z}_S^n\left(t_1\right).
\end{align}
With \ref{Z_S_supermartingale}, we have
\begin{align*}
& \quad \PP_{A_{n,a}}\left(\sup\limits_{t\in \left[0,\zeta_n-y\right]}\hat{Z}^n_S\left(t\right)-\left(1-\delta\right) n>0\right)\\
& = \PP_{A_{n,a}}\left(\sup\limits_{t\in \left[0,\zeta_n-y\right]}e^{-\lambda_1 t}\left(\hat{Z}^n_S\left(t\right)-\left(1-\delta\right)n\right)>0 \right)\\
& \le \PP_{A_{n,a}}\left(\sup\limits_{t\in \left[0,\zeta_n-y\right]}e^{-\lambda_1 t}\hat{Z}^n_S\left(t\right)+\sup\limits_{t\in \left[0,\zeta_n-y\right]}-e^{-\lambda_1 t}\left(1-\delta\right)n>0 \right)\\
& \le \PP_{A_{n,a}}\left(\sup\limits_{t\in \left[0,\zeta_n-y\right]}e^{-\lambda_1 t}\hat{Z}^n_S\left(t\right)>e^{-\lambda_1 \left(\zeta_n-y\right)}\left(1-\delta\right)n \right)\\
& = \PP_{A_{n,a}}\left(\sup\limits_{t\in \left[0,\zeta_n-y\right]}\exp\left(e^{-\lambda_1 t}\hat{Z}_S^n\left(t\right)\right)>\exp\left(e^{-\lambda_1 \left(\zeta_n-y\right)}\left(1-\delta\right)n \right)\right)\\
& \le \EE_{A_{n,a}}\left[\exp\left(\theta e^{-\lambda_1 \left(\zeta_n-y\right)}\hat{Z}_S^n\left(\zeta_n-y\right)\right)\right]\exp\left(-\theta e^{-\lambda_1 \left(\zeta_n-y\right)}\left(1-\delta\right)n\right),
\end{align*}
where we apply Doob's inequality to obtain the last inequality.
\end{proof} \qed
\\

By Lemma \ref{condition_ineq}, we can obtain that for $\theta<\frac{\lambda_1}{r_1}$,
\begin{align*}
& \quad \limsup\limits_{n\rightarrow \infty}\frac{1}{n^{1-\alpha}}\log \PP_{A_{n,a}}\left(\sup\limits_{t\in \left[0,\zeta_n-y\right]}\hat{Z}^n_S\left(t\right)-\left(1-\delta\right) n>0\right)\\
& \le a\frac{\lambda_1 \mu}{r r_1}\log\left( r\int_{0}^{\infty}\frac{\lambda_1 e^{\lambda_1 s}}{\lambda_1 e^{\lambda_1 s}-r_1 \theta}e^{-rs}ds\right)-\left(1-\delta\right)\frac{\theta \mu e^{\lambda_1 y}}{\lambda_1+r}.
\end{align*}
Therefore, let $\delta\rightarrow 0$, and from \ref{dom_1}, \ref{dom_2}, we obtain that 
\begin{align*}
& \quad \limsup\limits_{n\rightarrow \infty}\frac{1}{n^{1-\alpha}}\log \PP_{A_{n,a}}\left(\sup\limits_{t\in \left[0,\zeta_n-y\right]}Z^n_2\left(t\right)-n>0\right)\\
& \le -\sup\limits_{\theta\in \left(0, \frac{\lambda_1}{r_1}\right)}\left(\frac{\theta \mu e^{\lambda_1 y}}{\lambda_1+r}-a\frac{\lambda_1 \mu}{r r_1}\log\left( r\int_{0}^{\infty}\frac{\lambda_1 e^{\lambda_1 s}}{\lambda_1 e^{\lambda_1 s}-r_1 \theta}e^{-rs}ds\right)\right).
\end{align*}
For the large deviations lower bound, 
\begin{align*}
& \quad \PP_{A_{n,a}}\left(\sup\limits_{t\in \left[0,\zeta_n-y\right]}Z_2^n\left(t\right)-n>0 \right)\\
& \ge \PP_{A_{n,a}}\left(\sup\limits_{t\in \left[0,\zeta_n-y\right]}\hat{Z}_S^n\left(t\right)-n>0 \right)\\
& \ge \PP_{A_{n,a}}\left(\hat{Z}_S^n\left(\zeta_n-y\right)-n>0 \right)\\
& = \PP_{A_{n,a}}\left(n^{\alpha-1}e^{-\lambda_1\left(\zeta_n-y\right)}\hat{Z}_S^n\left(\zeta_n-y\right)-n^{\alpha-1}e^{-\lambda_1\left(\zeta_n-y\right)}n>0 \right)\\
& = \PP_{A_{n,a}}\left(n^{\alpha-1}e^{-\lambda_1\left(\zeta_n-y\right)}\hat{Z}_S^n\left(\zeta_n-y\right)+\frac{\mu e^{\lambda_1 y}}{\lambda_1+r}-n^{\alpha-1}e^{-\lambda_1\left(\zeta_n-y\right)}n>\frac{\mu e^{\lambda_1 y}}{\lambda_1+r} \right).\\
\end{align*}
We next show that the sequence of random variables
\begin{align*}
\bar{Z}_n=n^{\alpha-1}e^{-\lambda_1\left(\zeta_n-y\right)}\hat{Z}_S^n\left(\zeta_n-y\right)+\frac{\mu e^{\lambda_1 y}}{\lambda_1+r}-n^{\alpha-1}e^{-\lambda_1\left(\zeta_n-y\right)}n
\end{align*}  
satisfy the conditions of the Gartner-Ellis Theorem with
\[\bar{\Lambda}\left(\theta\right)= \begin{cases} 
a\frac{\lambda_1 \mu}{r r_1}\log\left( r\int_{0}^{\infty}\frac{\lambda_1 e^{\lambda_1 s}}{\lambda_1 e^{\lambda_1 s}-r_1 \theta}e^{-rs}ds\right) & \theta< \frac{\lambda_1}{r_1} \\
\infty & \theta\ge  \frac{\lambda_1}{r_1}.
\end{cases}
\]
In pursuit of this goal let us define the sequence of functions
\begin{equation*}
\bar{\Lambda}_n\left(\theta\right)=\log\EE\exp\left(\theta \bar{Z}_n\right),
\end{equation*}
and then show that all the five conditions for Gartner-ellis Theorem are satisfied.

Condition 1. We first notice that 
\begin{align*}
\lim\limits_{n\rightarrow \infty}n^{1-\alpha}\left(\frac{\mu e^{\lambda_1 y}}{\lambda_1+r}-n^{\alpha-1}e^{-\lambda_1\left(\zeta_n-y\right)}n\right)=0.
\end{align*} 
For $\theta<\frac{\lambda_1}{r_1}$,
\begin{align*}
& \quad n^{\alpha-1}\log \EE \exp\left(\theta e^{-\lambda_1 \left(\zeta_n-y\right)}\hat{Z}_S^n\left(\zeta_n-y\right)\right)\\
& \rightarrow a\frac{\lambda_1 \mu}{r r_1}\log\left( r\int_{0}^{\infty}\frac{\lambda_1 e^{\lambda_1 s}}{\lambda_1 e^{\lambda_1 s}-r_1 \theta}e^{-rs}ds\right), \text{ as } n \rightarrow \infty.
\end{align*}
We consider the case that $\theta \ge \frac{\lambda_1}{r_1}$. First we note that for $\theta\ge \frac{\lambda_1}{r_1}$ and a sequence of numbers $\{\tilde{\theta}_1, \tilde{\theta}_2,...\}$ such that $\tilde{\theta}_i<\frac{\lambda_1}{r_1}$, and $\lim\limits_{i\to \infty}\tilde{\theta}_i=\frac{\lambda_1}{r_1}$,
\begin{align*}
\int_{0}^{\infty}\frac{\lambda_1 e^{\lambda_1 s}}{\lambda_1 e^{\lambda_1 s}-r_1 \tilde{\theta}_i}e^{-rs}ds\le \liminf\limits_{n\rightarrow \infty}n^{\alpha-1}\bar{\Lambda}_n\left(\theta n^{1-\alpha}\right).
\end{align*}
As $i\rightarrow \infty$, $\int_{0}^{\infty}\frac{\lambda_1 e^{\lambda_1 s}}{\lambda_1 e^{\lambda_1 s}-r_1 \tilde{\theta}}e^{-rs}ds\rightarrow \infty. $. Therefore, $\liminf\limits_{n\rightarrow \infty}n^{\alpha-1}\bar{\Lambda}_n\left(\theta n^{1-\alpha}\right)=\infty$ and consequently $\lim\limits_{n\rightarrow \infty}n^{\alpha-1}\bar{\Lambda}_n\left(\theta n^{1-\alpha}\right)=\infty$.

%and $\tilde{\theta}<\frac{\lambda_1}{r_1}$,

%\begin{align*}
%\int_{0}^{\infty}\frac{\lambda_1 e^{\lambda_1 s}}{\lambda_1 e^{\lambda_1 s}-r_1 \tilde{\theta}}e^{-rs}ds\le \liminf\limits_{n\rightarrow \infty}n^{\alpha-1}\Lambda_n\left(\theta n^{1-\alpha}\right).
%\end{align*}

%As $\tilde{\theta}\rightarrow \frac{\lambda_1}{r_1}$, 

%\begin{align*}
%\int_{0}^{\infty}\frac{\lambda_1 e^{\lambda_1 s}}{\lambda_1 e^{\lambda_1 s}-r_1 %\tilde{\theta}}e^{-rs}ds\rightarrow \infty. 
%\end{align*}

Condition 2-5 are easy to verify and thus we omit the proofs here. By Gartner-Ellis Theorem, we obtain that
\begin{align*}
& \quad \liminf\limits_{n\rightarrow \infty}\frac{1}{n^{1-\alpha}}\log\PP_{A_{n,a}}\left(\sup\limits_{t\in \left[0,\zeta_n-y\right]}\hat{Z}_S^n\left(t\right)-n>0\right)\\
& \ge -\inf\limits_{x\in \left(\frac{\mu e^{\lambda_1 y}}{\lambda_1+r},\infty \right)}\sup\limits_{\theta\in \RR}\left[\theta x-\bar{\Lambda}\left(\theta\right)\right].
\end{align*}
We first examine the derivative of $f\left(\theta\right)=\theta x-\bar{\Lambda}\left(\theta\right)$ for $\theta<\frac{\lambda_1}{r_1}$, which is
\begin{align*}
f'\left(\theta\right)=x-a\frac{\lambda_1 \mu}{r r_1}\frac{\int_{0}^{\infty}\frac{r_1 \lambda_1 e^{\lambda_1 s}}{\left(\lambda_1 e^{\lambda_1 s}-r_1 \theta\right)^2}e^{-rs}ds}{\int_{0}^{\infty}\frac{\lambda_1 e^{\lambda_1 s}}{\lambda_1 e^{\lambda_1 s}-r_1 \theta}e^{-rs}ds}.
\end{align*}
We can obtain that 
\begin{align*}
f'\left(0\right)=x-\frac{a \mu}{ \lambda_1+r}.
\end{align*}
Hence, when $a< e^{\lambda_1 y}$, and $x\in \left(\frac{\mu e^{\lambda_1 y}}{\lambda_1+r},\infty \right)$, $f'\left(0\right)>0$. We notice that $f\left(0\right)=0$. By Jensen's inequality and the logarithm inequality ($\log$$\left(1+x\right)\ge\frac{x}{1+x}$ for $x>-1$), we can also obtain that when $\theta<0$,
\begin{align*}
f\left(\theta\right)& =\theta x-a\frac{\lambda_1 \mu}{r r_1}\log\left( r\int_{0}^{\infty}\frac{\lambda_1 e^{\lambda_1 s}}{\lambda_1 e^{\lambda_1 s}-r_1 \theta}e^{-rs}ds\right)\\
& \le \theta x-a\frac{\lambda_1 \mu}{r r_1}r\int_{0}^{\infty}\log\left(\frac{\lambda_1 e^{\lambda_1 s}}{\lambda_1 e^{\lambda_1 s}-r_1 \theta}\right)e^{-rs}ds\\
& \le \theta x -a\frac{\lambda_1 \mu}{r r_1}r\int_{0}^{\infty}\frac{\frac{r_1 \theta}{\lambda_1 e^{\lambda_1 s}-r_1\theta}}{\frac{\lambda_1 e^{\lambda_1 s}}{\lambda_1 e^{\lambda_1 \theta}-r_1 \theta}}e^{-rs}ds\\
& =\theta x-a\frac{\mu \theta}{\lambda_1+r}\\
& \le 0.
\end{align*}
Therefore, we conclude that $f\left(\theta\right)$ achieves its maximal value in $\left(0,\frac{\lambda_1}{r_1}\right)$. We then obtain that
\begin{align*}
& \quad \inf\limits_{x\in \left(\frac{\mu e^{\lambda_1 y}}{\lambda_1+r},\infty \right)}\sup\limits_{\theta\in \RR}\left[\theta x-\bar{\Lambda}\left(\theta\right)\right]= \inf\limits_{x\in \left(\frac{\mu e^{\lambda_1 y}}{\lambda_1+r},\infty \right)}\sup\limits_{\theta\in \left(0,\frac{\lambda_1}{r_1}\right)}\left[\theta x-\bar{\Lambda}\left(\theta\right)\right].
\end{align*}
Applying the same argument in the proof of Proposition \ref{recurrence_LDLB}, we can obtain that
%Define the function 
%$$
%h(x)=\sup_{\theta\in (0,\frac{\lambda_1}{r_1})}\left(\theta x-\Lambda(\theta)\right),
%$$
%then standard convex analysis tells us that $h$ is a convex function on $\mathbb{R}$ with %minimum at 
%$$
%\Lambda^\prime(0)=\frac{a \mu}{ \lambda_1+r}.
%$$
%In particular, $h$ is increasing on the set $(\frac{\mu e^{\lambda_1 %y}}{\lambda_1+r},\infty)$ and we conclude that
$$
\inf\limits_{x\in \left(\frac{\mu e^{\lambda_1 y}}{\lambda_1+r},\infty \right)}\sup\limits_{\theta\in \left(0,\frac{\lambda_1}{r_1}\right)}\left[\theta x-\bar{\Lambda}\left(\theta\right)\right]=\sup_{\theta\in (0,\frac{\lambda_1}{r_1})}\left(\frac{\theta \mu e^{\lambda_1 y}}{\lambda_1+r}-\bar{\Lambda}\left(\theta\right)\right).
$$ \qed

\subsection*{Acknowledgements}
The authors would like to thank Einar Gunnarsson for helpful comments on the draft. Also KL and ZW were supported by NSF grants CMMI-1552764.

\newpage

\bibliographystyle{plain}
\bibliography{main}

\newpage

\vspace{0.5in}

\noindent Pranav Hanagal \\
PNC Financial Services Group, Inc\\
New York, NY, USA \\
phanagal@gmail.com
\\

\noindent Kevin Leder and Zicheng Wang \\
Department of Industrial and Systems Engineering\\
University of Minnesota\\
Minneapolis, MN 55455, USA \\
kevin.leder@isye.umn.edu, wang2569@umn.edu

%\vspace{0.2in}

\end{document}